\documentclass[12pt]{amsart}
\title{
Uniformization of  Sierpi\'nski carpets in the plane}%
\author{Mario Bonk}
\address{Mario Bonk\\Department of Mathematics \\
University of California, Los Angeles\\Box 951555\\
Los Angeles, CA 90055\\USA} \email{mbonk@math.ucla.edu}
\thanks{The author was supported by NSF grants DMS 0244421, DMS 0456940,  and
DMS 0652915.}

\date{September 15, 2010. Revised: February 21, 2011.}

\usepackage{amsmath, amssymb, amsthm,latexsym}
\usepackage{graphicx}
\newcommand\C{{\mathbb C}}
\newcommand\oC{\widehat {\mathbb C}}
\newcommand\N{{\mathbb N}}
\newcommand\D{{\mathbb D}}

\newcommand\R{{\mathbb R}}
\newcommand\iu{{\textbf {\textit{i}}}}
\renewcommand\H{{\mathbb H}}
\newcommand\Sph{\widehat {\mathbb C}}

\newcommand\geo{\partial_\infty}
\newcommand\dee{\partial}
\newcommand\Mod{\operatorname{mod}}
\newcommand\M{\mathcal{M}}
\newcommand\dist{\operatorname{dist}}
\newcommand\diam{\operatorname{diam}}

\newcommand\inte{\operatorname{int}}
\newcommand\forr{\quad\text{for} \quad}
\newcommand\foral{\quad\text{ for all} \quad}
\newcommand\co{\colon}
\renewcommand\:{\colon}
\newcommand\sub {\subseteq}
\newcommand\ra {\rightarrow}

\def\length{\mathop{\mathrm{length}}}
\newcommand\Ga{\Gamma}
\newcommand\Om{\Omega}
\newcommand\ga{\gamma}
\newcommand\la{\lambda}
\newcommand\eps{\epsilon}
\newcommand\sig{\sigma}
\newcommand\Sig{\Sigma}
\newcommand\de{\delta}
\newcommand\OC{{\widehat{\mathbb C}}}
\newcommand\no{\noindent} 
\newcommand\Md{\operatorname{M}} 

\newtheorem{theorem}{Theorem}[section]

\newtheorem{conjecture}[theorem]{Conjecture}
\newtheorem{proposition}[theorem]{Proposition}
\newtheorem{corollary}[theorem]{Corollary}

\newtheorem{remark}[theorem]{Remark}
\newtheorem{lemma}[theorem]{Lemma}

\theoremstyle{definition}
\newtheorem{example}[theorem]{Example}

\begin{document}

\abstract{Let  $S_i$, $i\in I$,  be a  countable collection of Jordan  curves  in the extended complex plane
 $\Sph$
that bound pairwise disjoint closed Jordan  regions. If the Jordan  curves  are uniform quasicircles and are uniformly relatively separated, then there exists a quasiconformal map 
$f\: \Sph\ra \Sph$ such that $f(S_i)$ is a round circle  for all $i\in I$. This implies that every  Sierpi\'nski carpet in $\oC$ whose peripheral circles are uniformly relatively separated uniform quasicircles can be mapped to a round Sierpi\'nski carpet by a quasisymmetric map.}
\endabstract

\maketitle

%\no
%{\em J'ai pr\'ef\'er\'e  passer pour un peu bavard  
%(H.~Poincar\'e).}
\tableofcontents
\section{Introduction}\label{s:Intro}

\no 
Let $\Sph=\C\cup\{\infty\}$ denote the extended complex plane equipped with the chordal metric $\sigma$ given by 
\begin{equation}\label{eq:chmet}
\sigma(x,y)=\frac{2|x-y|}{\sqrt{1+|x|^2}
\sqrt{1+|y|^2}} \forr x,y\in \C,
\end{equation}   
and by  a suitable limit of  this expression if $x=\infty$ or 
$y=\infty$. 
As usual $\OC$ can be identified with the unit sphere in $\R^3$ equipped with the restriction of the Euclidean metric by stereographic projection. 

A Jordan  curve  $S\sub \Sph$ is called a {\em quasicircle} if the following condition holds: there exists a constant $k\ge 1 $ such that 
for all points  $x,y\in S$, $x\ne y$, we have the inequality 
\begin{equation}\label{eq:qcirc}
\diam(\ga)\le k \sig(x,y)
\end{equation}  for the diameter  of one of the subarcs 
$\ga$ of $S$ with endpoints $x$ and $y$.   
Essentially, this condition rules out cusps for $S$. 
Typical examples of quasicircles are  von Koch 
snowflake-type curves.     
It is well-known that $S$ is a quasicircle if and only  if there exists a quasiconformal map $f\:\Sph\ra \Sph$ such that $f(S)$ is a round circle. So  the quasicircles are precisely the images 
of round circles under quasiconformal homeomorphisms  on  $\Sph$. 

One can ask whether  a similar statement  is true 
for a collection $\mathcal{S}=\{S_i: i\in I\}$  of pairwise disjoint quasicircles $S_i$, where $I$ is a  countable index set.  So we want to find a  quasiconformal homeomorphism  $f$ on $\Sph$  that makes all the quasicircles in the collection simultaneously round. 

It is clear that such a map $f$ need not exist if we do not impose  further restrictions on the collection $\mathcal{S}$. Indeed, as follows   from standard distortion estimates for quasiconformal maps,  a necessary condition for the existence of the  map $f$ is that $\mathcal{S}$  consists of {\em uniform quasicircles}: there exists a constant $k\ge 1$ such that each $S_i$ for $i\in I$ is a $k$-quasicircle, i.e., it satisfies condition \eqref{eq:qcirc}.  Even if the Jordan curves  $S_i$ are uniform quasicircles, the desired map $f$ need not exist. An example can be obtained from  an infinite collection of disjoint squares that contains a sequence of pairs of squares with parallel sides 
of equal length such that the  distance between the sides goes to zero  faster than the sidelength (see 
Example~\ref{ex:nonunif}).

A way to exclude such examples is to impose {\em uniform relative separation} on the collection $\mathcal{S}$:
there exists a constant $s>0$ such that
\begin{equation} \label{eq:unifsep}
\frac{\dist(S_i,S_j)}{\min\{\diam(S_i),\diam(S_j)\}}
\ge s,
\end{equation}  
whenever $i,j\in I$, $i\ne j$. 
This requirement stipulates that the relative distance
of two distinct quasicircles in $\mathcal{S}$ 
(the distance rescaled by the smaller diameter 
of the sets) is uniformly  bounded from below. 
The condition of uniform relative separation still allows 
rather tight collections of quasicircles. For example, the peripheral circles of the standard Sierpi\'nski carpet $T$ (given by the boundaries of the squares used in the construction of $T$; see Section~\ref{s:carpet}) form a collection of uniformly relatively separated uniform quasicircles. 

Even if the collection $\mathcal{S}$ consists of  uniformly relatively 
separated uniform quasicircles, a map $f$ as desired need not exist due to possible nesting of the quasicircles $S_i$ (see 
Example~\ref{ex:nested}). This problem is ruled out if we require 
that the curves $S_i$ bound pairwise  disjoint closed Jordan  regions.

If we impose all the conditions on $\mathcal{S}$ as discussed, then we actually get 
a positive statement as our  first main result  shows.   

\begin{theorem}\label{thm:simulunif}
Suppose that $\mathcal{S}=\{ S_i:i\in I\}$ is a family of 
Jordan  curves  in $\Sph$ that bound pairwise 
disjoint closed Jordan  regions. If $\mathcal{S}$ consists of uniformly relatively separated uniform quasicircles, then there 
exists a quasiconformal map 
$f\: \Sph\ra \Sph$ such that $f(S_i)$ is a round circle for all $i\in I$.
\end{theorem}

The proof will show that this statement is quantitative in the following sense: if $\mathcal{S}$ 
consists of  $s$-relatively separated $k$-quasicircles, then the map $f$ can be chosen to be $H$-quasiconformal with $H$ only depending on $s$ and $k$.

One can ask to what extent the map $f$ is uniquely determined. Suppose that $\{D_i:i\in I\}$ is a  collection
of pairwise disjoint closed  Jordan  regions such that $\partial D_i=S_i$, where 
the collection $\mathcal{S}=\{S_i:i\in I\}$ is as in Theorem~\ref{thm:simulunif}.  Obviously, it is easy to perturb  $f$ on  the interior $\inte(D_i)$ of any of the sets $D_i$ while retaining the roundness of the circles $f(S_i)$. So it is only meaningful to ask for uniqueness of $f$ on the complementary set 
\begin{equation}\label{eq:T}
T=\Sph\setminus \bigcup_{i\in I}\inte(D_i) 
\end{equation}  of the regions $D_i$, $i\in I$. 
Again if $T$ has non-empty interior, then $f$ is not unique on $T$,
but it turns out that if $T$ has  measure zero, then $f$ is uniquely determined on $T$ up to post-composition with  a M\"obius transformation.  This follows from rigidity statements for {\em Schottky sets} in $\Sph$, i.e., compact subsets of $\Sph$ whose complementary 
components consist of pairwise disjoint open disks
\cite[Thm.~1.1]{BKM}. 

If one uses  this rigidity result in combination with 
Theorem~\ref{thm:simulunif}, one obtains the following 
existence and uniqueness statement for the uniformization of Sierpi\'nski carpets  by round Sierpi\'nski carpets, i.e., 
Sierpi\'nski carpets in $\OC$ 
whose complementary components are round disks (see Section~\ref{s:carpet} for terminology). 

\begin{corollary}\label{cor:carpetunif}
Suppose that $T\sub \Sph$ is a Sierpi\'nski 
carpet whose peripheral circles are uniformly relatively separated uniform quasicircles. Then $T$ can be mapped to    a round       
 Sierpi\'nski carpet $T'$ by a quasisymmetric
 homeomorphism $f\: T\ra T'$.

If  $T$ has spherical measure zero, then 
the quasisymmetric map $f$  is unique up to post-composition with a M\"obius transformation on $\Sph$. 
\end{corollary}

In particular, this corollary applies  to the standard Sierpi\'nski carpet $T$ (see Section~\ref{s:carpet}). For this special case the existence part of the statement was   proved earlier by methods different from the ones used in this paper in unpublished  joint work by B.~Kleiner and the author.   

Corollary~\ref{cor:carpetunif}
is an analog of a classical  uniformization theorem due to Koebe. It states 
that every finitely connected region  
in $\Sph$ can be mapped to  a {\em circle domain}
 (a region  whose complementary components are 
closed, possibly degenerate disks) by a conformal map. Moreover, this map is unique up to post-composition with  an 
orientation-preserving M\"obius transformation. Actually, we will use Koebe's theorem in the proof of Theorem~\ref{thm:simulunif}. 

Our investigations were partly motivated by a problem in Geometric Group Theory, the Kapovich-Kleiner conjecture. This conjecture predicts that if a   Gromov hyperbolic group  $G$ has a boundary  at infinity $\partial_\infty G$ that is a Sierpi\'nski carpet, then $G$ should arise from a standard situation in hyperbolic geometry. More precisely, $G$ should admit an action on a convex subset of hyperbolic $3$-space $\H^3$ with non-empty totally geodesic boundary where the action is  isometric,
properly discontinuous, and cocompact \cite{KK}.  
If $G$ admits  such an action on $\H^3$, 
then $\partial_\infty G$ can be identified with a round 
Sierpi\'nski 
carpet. The Kapovich-Kleiner conjecture is equivalent to the following 
uniformization conjecture for metric Sierpi\'nski carpets arising 
as boundaries of hyperbolic groups. 

\begin{conjecture} \label{coni:KK}
Suppose that $G$ is a Gromov hyperbolic group
such that $\partial_\infty G$ is a Sierpi\'nski carpet.
Then $\partial_\infty G$ is quasisymmetrically 
equivalent to a round Sierpi\'nski carpet $T\sub \Sph$. 
\end{conjecture} 

Here  the set $\partial_\infty G$ can be considered as a metric space in a natural way by equipping it with 
a ``visual" metric. Though in general there is no unique choice of such a metric, these metrics are quasisymmetrically equivalent by the identity map. 

For Gromov hyperbolic groups $G$ with  
Sierpi\'nski carpet boundary 
$\partial_\infty G$ one can show 
the following properties 
 of the collection of peripheral circles of $\partial_\infty G$. 
 
 \begin{proposition} \label{prop:groupcarpets}
 Let $G$ be a Gromov hyperbolic group
 such that $\partial_\infty G$ is a Sierpi\'nski carpet, 
 and let $\mathcal{S}$ be the collection of peripheral circles of $\partial_\infty G$. Then $\mathcal{S}$
 consists of uniform quasicircles that are uniformly 
 relatively separated and occur on all locations and scales. 
 \end{proposition}

This proposition will not be a surprise  to experts, but it cannot  be found in the literature.  We will record a proof 
in Section~\ref{s:hypgroups} where we also explain the terminology used  in the statement.  If one combines this proposition with  Corollary~\ref{cor:carpetunif}, then Conjecture~\ref{coni:KK} is reduced to showing that  every  Sierpi\'nski carpet $\partial_\infty G$ arising as the boundary at infinity 
 of a Gromov hyperbolic group $G$ can be mapped to a  Sierpi\'nski carpet $T\sub \oC$ by a quasisymmetry.

In view of Proposition~\ref{prop:groupcarpets}, 
one can ask the more general question whether any metric 
Sierpi\'nski carpet $T$ whose  peripheral circles are  uniformly relatively separated uniform  quasicircles   is quasisymmetrically equivalent to a round Sierp\'inski 
carpet in $\OC$.  This is  not  true in general, but 
in \cite{BK}  Corollary~\ref{cor:carpetunif}
is used to show that 
this holds under the additional assumption that 
$T$ has Ahlfors regular conformal dimension less than $2$.

Corollary~\ref{cor:carpetunif} is an instance of a new  phenomenon that can be formulated as a heuristic principle:
{\em Quasisymmetric maps on Sierpi\'nski carpets of measure zero behave similarly  as conformal maps on regions in $\OC$.}

The main point here is that we have analogies of quasisymmetric maps on Sierpi\'n\-ski carpets with {\em conformal}  maps, and not as 
expected, and less surprising, with 
quasiconformal maps on regions.

The following fact supports  our  heuristic principle. If $\Ga$ is a path family in $\OC$ and $T$ is a 
Sierpi\'nski carpet, then one can assign a type of conformal 
modulus, the {\em carpet modulus} $\M_T(\Ga)$, to $\Ga$ that is  preserved under quasiconformal maps (and not only quasi-preserved as expected).  See Section~\ref{s:carpet} for the details. This notion corresponds to the classical modulus of path families that is preserved 
under conformal maps.  Applications of the carpet modulus
to proving rigidity statements for Sierpi\'nski carpets
are studied in \cite{BM}. 

Corollary~\ref{cor:carpetunif} in combination with the main result 
in \cite{BKM} leads to surprising uniqueness results. Here is an example.

\begin{theorem}[Three-Circle Theorem]\label{thm:3circ}
Suppose that $T\sub \Sph$ is a Sier\-pi\'nski 
carpet of spherical measure zero whose peripheral circles are uniformly relatively separated uniform quasicircles. Let  $f\: T\ra T$ be an orientation-preserving quasisymmetric homeomorphism of $T$ onto itself. 

If there exist three distinct peripheral circles $S_1$, $S_2$, $S_3$ of $T$ with $f(S_i)=S_i$ for $i=1,2,3$, or if there exist three distinct fixed points of $f$ in $T$,   then $f$ is the identity map on $T$.
\end{theorem}
 
 This theorem is in complete contrast to the topological 
 flexibility of Sier\-pi\'nski 
carpets: if $T$ is a carpet, $ \{S_i: i=1, \dots,n\}$ a family of distinct peripheral circles  of $T$, and $\{S_i': i=1,\dots, n\}$ another such family, then there exists an homeomorphism 
$f\: T\ra T$ such that $f(S_i)=S_i'$ for $i=1, \dots, n$ (the author was unable to locate this result in the literature, but it can easily be established by using  the methods in \cite{Why}.)    
 
 We will prove another uniformization theorem for Sierpi\'nski 
 carpets that has an application to an extremal problem for carpet modulus.
 
 \begin{theorem} \label{thm:cylunif0} 
 Suppose that 
 $$T=\Sph\setminus \bigcup_{i\in \N_0} \inte(D_i),$$ is a Sierpi\'nski carpet, where the sets $D_i$, $i\in \N_0$,  are pairwise disjoint closed Jordan  regions, and  that the collection $\partial D_i$,  $i\in \N_0$, of peripheral circles of $T$ consists of uniformly 
relatively separated uniform quasicircles. 

Then there exists a finite $\C^*$-cylinder
$A$, pairwise disjoint $\C^*$-squares $Q_i\sub A$ for  $i\ge 2$,  and a  quasisymmetric   map  
$$f\: T\ra T':=\overline A\setminus   \bigcup_{i\ge 2} \inte(Q_i) $$ such that $f(\partial D_0)=\partial_iA$, $f(\partial D_1)=\partial_oA$, and $f(\partial D_i)= \partial Q_i$ for $i\ge 2$. 
 \end{theorem}
 
 See the discussion after Theorem~\ref{thm:Koebunif} for the terminology employed here.  One can show that    if $T$ has spherical measure zero, then $f$ is unique up to a Euclidean similarity fixing the origin in $\C$  (this follows from 
 \cite[Thm.~1.5] {BM}).
 
 To formulate the mentioned application  of this theorem to an extremal problem,   let   $\Ga$   be the family of all open   paths $\ga$ in $\OC\setminus(D_0\cup D_1)$ connecting  $\partial D_0$ and $\partial D_1$ (see the end of Section~\ref{s:mod} for the precise definition of such paths).   If $A=\{z\in \C: r<|z|<R\}$ we denote by 
 $h_A=\log(R/r)$ the   height of $A$. Then  the carpet modulus 
 of $\Gamma$ with respect to $T$ is given by 
 $$ \M_T(\Ga)=\frac{2\pi}{h_A}.$$
 See   Corollary~\ref{cor:annmod}, where  we will also identify   the unique 
 extremal weight sequence   for  $ \M_T(\Ga)$.  
 
 Theorem~\ref{thm:cylunif0} and Corollary~\ref{cor:annmod} can be considered  as   limiting cases of statements  in   classical uniformization (see Theorem~\ref{thm:squareunif}, Corollary~\ref{thm:cylunif}, and Proposition~\ref{prop:transcyl}) or  of combinatorial facts related to square tilings (see \cite{CFP, oS93}).  Our proof of Theorem~\ref{thm:cylunif0} relies on the corresponding uniformization statement  Corollary~\ref{thm:cylunif} and a limiting argument. It would be very interesting to find a different 
 proof that proceeds from   results on  square tilings as a starting point.

We will now give an outline of the proof of Theorem~\ref{thm:simulunif}. The main point is to find a quasisymmetric map $g$ of the set $T$ in \eqref{eq:T} 
 onto a set $T'\sub \OC$ whose complementary components are round disks. 
The desired quasiconformal map  $f$ is then found by ``filling the holes" (see Proposition~\ref{prop:extend}).    
Finding this extension $f$ of $g$ involves some subtleties, but  can be derived from the classical Beurling-Ahlfors extension theorem (see Theorem~\ref{thm:AB}) without too much trouble.

The construction of  $g$ is based on the obvious idea 
to obtain this map as a sublimit of conformal maps   that 
map finite approximations of $T$ to circle domains; more precisely, assuming $I=\N$ we let 
$$ T_n= \OC\setminus \bigcup_{i=1}^n\inte(D_i)$$ 
and invoke Koebe's Uniformization Theorem to  find a map $g_n$ 
for each $n\in \N$  that is suitably normalized and conformally maps the interior of $T_n$  to a circle domain.  We then  show that these maps $g_n$ are {\em uniformly quasisymmetric} 
(see Theorem~\ref{thm:unifqs}) and hence have a sublimit $g$ with the desired properties (see the  proof of Theorem~\ref{thm:simulunif} in Section~\ref{s:proof1}).  

The proof of the uniform quasisymmetry of the maps $g_n$ is 
the main difficulty. The standard method for establishing  distortion estimates as required for the  quasisymmetry of a map are modulus estimates. In our situation one cannot expect that this method gives  the required uniform bounds.  The reason is that 
by removing more and more of the sets $\inte(D_i)$ from $\OC$, 
the remaining  sets $T_n$  may carry smaller and smaller path families. In particular,  if 
$T$ has measure zero, then every path family in $T$ has vanishing modulus and it is unlikely that classical modulus will lead to the desired bounds.  

To overcome these obstacles we use {\em transboundary modulus} (see Section~\ref{s:mod}). This concept (under the different name ``transboundary extremal length") was  introduced by O.~Schramm \cite{oS95} and is a variant of classical conformal 
modulus. Since in its definition transboundary modulus  uses the ``holes" (i.e., the complementary components) of a region,  we can hope to  get  uniform
   positive lower bounds for the transboundary modulus of  path families  that are relevant for desired distortion estimates of the maps $g_n$ (see Proposition~\ref{prop:lowtrans} for a general statement in this direction). 

Unfortunately, while classical modulus is too small for our 
purpose, transboundary modulus will be too large in general.  Essentially,  one wants a quantity that is not too small in the source domain, but not too large  in the target. 
Subtle considerations are neccessary to navigate around this problem: one only considers the transboundary modulus of path families  in the complement of a controlled number of carefully selected holes of the target domain (see Proposition~\ref{thm:uptrans} and the further discussion following the statement of this proposition).    This will lead to the right balance  of modulus estimates 
 for source and target. 
 Carrying out the  details involves  substantial technicalities. The key steps in the proof are Propositions~\ref{thm:lowmod}, 
 \ref{prop:lowtrans}, and \ref{thm:uptrans}. They enter  the proof of the uniform quasisymmetry statement Theorem~\ref{thm:unifqs} from which Theorem~\ref{thm:simulunif} can easily be derived.

The paper is organized as follows. We fix notation and some terminology in Section~\ref{s:nt}. In Section~\ref{s:qc} we review quasiconformal and related maps, and in Section~\ref{s:quasicircles} relevant facts about quasicircles. Most of this material is  standard, but we have  included many details in order to make the paper as self-contained as possible.  The extension result Proposition~\ref{prop:extend}, already mentioned in the outline of the proof of Theorem~\ref{thm:simulunif}, is proved in the next Section~\ref{s:ext}. 
Classical and transboundary modulus appear 
in Section~\ref{s:mod}.  In Section~\ref{s:loewner} we discuss     Loewner domains and  establish  Proposition~\ref{thm:lowmod} which is  used in the proof of our main result.  

 Section~\ref{s:lowtrans} is devoted to estimates for transboundary modulus. The main results are 
the rather technical Propositions~\ref{prop:lowtrans}~and~\ref{thm:uptrans}, the former giving a lower and the latter  an upper bound for transboundary modulus. They are crucial  in the proof of Theorem~\ref{thm:simulunif}.  Proposition~\ref{thm:weakuptrans} is later applied  in Section~\ref{s:square}.

Results on classical uniformization are discussed in Section~\ref{s:unifcyl}. Apart from Koebe's Uniformization Theorem and some rather 
standard results on boundary extension of conformal maps, none of this material is used in the proof of Theorem~\ref{thm:simulunif}. The main results in this section are   Theorem~\ref{thm:squareunif} and Corollary~\ref{thm:cylunif}. 
This corollary  is  later  invoked  in the proof of Theorem~\ref{thm:cylunif0}. Theorem~\ref{thm:squareunif} 
can be derived from results by Schramm \cite{oS96}, but we decided to present  the details for the convenience of the reader. 

All the preparation is wrapped up in Section~\ref{s:proof1}, where a proof of Theorem~\ref{thm:simulunif} is finally given. It is based on Theorem~\ref{thm:unifqs} which  is of independent interest. Examples~\ref{ex:nonunif} and \ref{ex:nested}
 in this section show that one can neither  drop the assumption of uniform relative separation in our main theorem, nor the assumption that the quasicircles $S_i$ bound pairwise disjoint closed Jordan  regions. 
 
 In Section~\ref{s:square} we solve an extremal problem for transboundary modulus (Proposition~\ref{prop:transcyl}). As an application  we  prove a uniqueness statement for conformal maps   (Corollary~\ref{cor:uniqsquare}). We also  prepare and give  the proof of  Theorem~\ref{thm:cylunif1a} which is  a slightly more general version of   Theorem~\ref{thm:cylunif0}. In Section~\ref{s:carpet} we recall the definition of the standard Sierpi\'nski carpet and some related topological facts. 
In this section we  prove Corollary~\ref{cor:carpetunif},   Theorem~\ref{thm:3circ}, and Theorem~\ref{thm:cylunif0}, and define the concept of carpet modulus of a curve family. In the final Section~\ref{s:hypgroups} we establish Proposition~\ref{prop:groupcarpets}.

\medskip\no  {\bf Acknowledgements.} The author   thanks  
Pietro Poggi-Corradini and the anonymous referee for many useful comments on this paper. He   is greatly indebted to Bruce Kleiner, Sergei Me\-ren\-kov, and the late Juha Heinonen for many fruitful discussions. The existence part of Corollary~\ref{cor:carpetunif} was conjectured by Bruce Kleiner 
in a seminar talk at the University of Michigan in the spring of 2004.  Sergei Merenkov shared his  knowlegde  
on Schramm's papers on conformal uniformization with the author and contributed important ideas that led to the proof of Theorem~\ref{thm:cylunif0} and the notion of carpet modulus.  Discussions with Juha Heinonen were crucial for  clarifying  this  modulus concept.

\section{Notation and terminology}\label{s:nt}
\no 
For $a,b\in \R$ we set 
$a\wedge b=\min\{a,b\}$ and $a\vee b=\max\{a,b\}$. 
We let 
$\D=\{z\in \C: |z|<1\}$, and  $\C^*=\{z\in \C: z\ne 0\}$. The symbol $\iu$ stands for the imaginary unit in the complex plane $\C$. 
 
The chordal metric on $\OC$ is denoted by $\sig$ (see \eqref{eq:chmet}). We will  use the letter $\Sig$ to denote the spherical measure 
on $\OC$. So if $M\sub \C$ is measurable, then 
$$\Sig(M)=\int_M\frac{4\, dm_2(z)}{(1+|z|^2)^2}, $$
where integration is with respect to Lebesgue measure 
$m_2$ on $\C\cong \R^2$.  Integrals will be extended over $\oC$ unless otherwise indicated. We say that  a measurable set $M\sub \oC$ has  {\em
(spherical)  measure zero} if $\Sigma(M)=0$.

Let $(X,d)$ be  a metric space. If  $a\in X$ and $r>0$, we denote by $$B(a,r)=\{y\in X: d(y,x)<r\}$$ the open and by 
$$\overline B(a,r)=\{y\in X: d(y,x)\le r\}$$ the closed ball 
of radius $r$ centered at $a$.  If $A\sub X$, then we  write $\overline A$ for the closure, $\inte(A)$ for the interior,  $\partial A$ for the topological boundary, and $\diam(A)$ for the diameter of the set $A$.  

For sets  $A,B\sub X$ we write 
$$ \dist(A,B)=\inf\{d(x,y): x\in A, \, y\in B\}$$ 
for  their  distance, and $$\Delta(A,B)=\frac{\dist(A,B)}{\diam(A)\wedge \diam(B)} $$
for their {\em relative distance} if in addition 
$\diam(A)>0$ and $\diam(B)>0$. 
 For $x\in X$ and $A\sub X$, we set 
$\dist(x,A)=\dist(\{x\},A)$. If  $\eps>0$ we denote by 
$$N_{\eps}(A)=\{x\in X: \dist(x,A)<\eps\}$$ 
the open $\eps$-neighborhood of $A$ in $X$.

Mostly, it will be clear from the context what metric $d$ we are using. If necessary  we put the 
symbol for the metric as subscript on metric notation. For example,  $B_d(x,r)$ denotes the open ball with respect to the metric $d$, etc.    By default, all sets in $\Sph$ carry the restriction of the chordal metric $\sigma$. We sometimes  use the {\em spherical metric} on $\oC$, i.e., the Riemannian metric with length element 
$$ ds=\frac{2|dz|}{1+|z|^2}. $$ For sets in $\C$ we also  use the 
Euclidean metric $d_\C$ defined by $d_\C(x,y)=|x-y|$ for $x,y\in 
\C$, and for sets in $\C^*$ the flat metric $d_{\C^*}$ (see Section~\ref{s:unifcyl} for its definition).   To distinguish metric notions 
that refer to $d_\C$ or $d_{\C^*}$ from their counterparts with respect to the metric $\sigma$, we use the subscript $\C$ or $\C^*$. For example, we  denote by $\diam_\C(A)$  the Euclidean diameter of a set $A\sub\C$, by $\length_{\C^*}(\ga)$ the length 
of a path $\ga$ (see below)  in $\C^*$ with respect to the metric $d_{\C^*}$, etc.  

A {\em circle} in $\OC$ is a set of the form 
$$S(x,r):=\{y\in \OC: \sigma(y,x)=r\}, $$
where $x\in \OC$ and $0<r<\diam(\OC)=2$. Sometimes we call
these sets also {\em round circles} to emphasize their distinction from quasicircles or metric circles (see Section~\ref{s:quasicircles}).   Similarly, a {\em round disk} is a (closed or open) metric ball in $\oC$ with respect to the metric $\sigma$.   

If  $f\:X\ra Y$ is a map between two sets $X$ and $Y$, and $M\sub X$, then $f|M$ denotes the restriction of  $f$ to $M$. 

A {\em path} in a metric space $(X,d)$ is a continuous map  $\ga\: I\ra X$
 of an interval $I\sub \R$ (i.e, a non-empty connected subset 
 of $\R$) into $X$.  If now confusion can arise, we will also denote by $\ga$ the image set $\ga(I)\sub X$ of a path $\ga$.
 We denote by $\length (\ga)\in [0,\infty]$ the length of $\ga$.
 The  path $\ga\:I\ra X$ is {\em rectifiable}  if  $\length (\ga)<\infty$, and {\em locally 
 rectifiable}  if $\length(\ga|J)<\infty$ for each compact subinterval 
 $J\sub I$.

A {\em region}   $\Om$ in $\Sph$ is an open and connected set. 
A {\em Jordan  curve}   $S$ in $\Sph$  is a subset
of $\Sph$ homeomorphic to the unit circle $\partial \D=\{z\in \C:|z|=1\}$. A {\em closed Jordan  region} in $\Sph$ is a set homeomorphic to the closed unit disc $\overline \D$, and an {\em open Jordan  region}
a set in $\oC$ that is the interior of  a closed Jordan  region.   According to the
Sch\"onflies theorem for each 
Jordan  curve $S\sub \Sph$ there exists a homeomorphism 
$F\: \Sph\ra \Sph$ such that $F(\partial \D)=S$. In particular, 
every Jordan  curve $S\sub \Sph$ has two complementary components in $\Sph$ both of which are open Jordan  regions.

In this paper it is very important to keep track of the dependence of constants and functions on parameters (i.e., other constants and functions). We will write $C=C(a,b,\dots)$ if the constant $C$ can be chosen only depending 
on the parameters  $a,b,\dots$,  and $A\le C(a,b,\dots)$ if 
the quantity $A$ is bounded by a constant only depending on $a,b,\dots$  For the dependence of functions from parameters 
we use subscripts to distinguish this dependence from 
the evaluation of the function on elements of its domain  of definition; so $\phi=\phi_{a,b,\dots}$ means that 
$\phi$ is a function that can be chosen only  depending on the parameters $a,b,\dots$ 

Sometimes  a property of a space, function, etc.\ depending  on some parameters $a,b, \dots$ implies another  property depending on 
other parameters $a',b',\dots$ If we can choose the parameters $a',b',\dots$  as fixed functions of $a,b,\dots$, that is, only depending on $a,b,\dots$, then we say that the first property implies the second one  
{\em quantitatively}.  If we have implications of this type in both directions, we call the properties {\em quantitatively equivalent}.  See the remark after Proposition~\ref{prop:inter} for the discussion of a specific example. 

We  always assume that the  parameters are in their natural range or of appropriate type. So, for example,  in  the phrase ``the family 
 $\{S_i: i\in I\}$ is $s$-relatively separated" it is understood that $s>0$,  and in  ``$f$ is an $\eta$-quasisymmetry" (see Section~\ref{s:qc})  that $\eta$ 
 is a distortion function with the right properties, i.e., a homeomorphism $\eta\:[0,\infty)\ra [0,\infty)$.

 We also often omit  quantifying statements (``there exists" or ``for all") related to these  parameters  for ease of formulation if the intended meaning is clear.
For example, we say ``$J$ is a $k$-quasicircle" instead of 
 ``there exists $k\ge 1$ such that $J$ is a $k$-quasicircle" (cf.\ statement (i) in Proposition~\ref{prop:planeqcirc}) or ``the maps $f_n$ are $\eta$-quasisymmetries for $n\in \N$" 
  instead of ``there exists a homeomorphism $\eta\: [0,\infty)\ra [0,\infty)$  such that the maps $f_n$  are $\eta$-quasisymmetries for all $n\in \N$" 
  (cf.\  Lemma~\ref{lem:subcon}).

\section{Quasiconformal and related maps} \label{s:qc}

\no In this section we summarize  
basic facts on  quasiconformal and related maps (see \cite{He}, \cite{LV}, and \cite {Va} for general background).
Let $f\: \Sph\ra \Sph$ be a homeomorphism, and for $x\in \Sph$ 
and   small $r>0$ 
define 
\begin{equation}\label{Lf}
L_f(r,x)=\sup\{\sig(f(y),f(x))\co y\in \Sph \text{ and } \sig(y,x)=r\}, 
\end{equation}
\begin{equation}\label{lf}
l_f(r,x)=\inf\{\sig(f(y),f(x))\co y\in \Sph \text{ and } \sig(y,x)=r\},   
\end{equation}
and  
\begin{equation}\label{Hf}
H_f(x)=\limsup_{r\to 0}\frac{L_f(x,r)}{l_f(x,r)}. 
\end{equation}
The map $f$ is called 
{\em  quasiconformal} if
$$
\sup_{x\in\Sph}H_f(x)<\infty. 
$$
A quasiconformal map $f$ is called  $H$-{\em quasiconformal}, $H\ge 1$,
if 
$$ H_f(x)\le H \quad \text{
for almost every } x\in \Sph. $$  We refer to $H$ as the {\em dilatation} of the map $f$. 

Quasiconformality can be defined similarly in other  settings, for example for homeomorphisms between regions in $\Sph$ or $\R^n$, or between  Riemannian manifolds.

The composition of an $H$-quasiconformal and an $H'$-quasi\-conformal 
map   is an $(HH')$-quasiconformal map.  
If a homeomorphism  $f$  on  $\Sph$ is  $1$-quasiconformal, then  $f$ is a {\em M\"obius transformation}, i.e., a conformal or anti-conformal map on $\oC$, and so a fractional linear 
transformation, or the complex conjugate of such a map. Note 
that our definition of a  M\"obius transformation is slighly non-standard in complex analysis as we allow anti-conformal maps. 

Let $f\:X\ra Y$ be a homeomorphism between metric spaces $(X,d_X)$ and $(Y,d_Y)$.
The map $f$ is  called $\eta$-{\em quasisymmetric} or an 
$\eta$-{\em quasisymmetry}, 
where $\eta\:[0,\infty)\ra [0,\infty)$ is  a homeomorphism,  if
\begin{equation}\label{qsdef}
\frac{d_Y(f(x),f(y))}{d_Y(f(x),f(z))}
\leq \eta\left (\frac{d_X(x,y)}{d_X(x,z)}\right)
\end{equation}
for all $x,y,z\in X$ with  $x\ne z$. 
The map $f$ is called {\em quasisymmetric} or an {\em quasisymmetry} if it is 
 $\eta$-{\em quasisymmetric} for some distortion function $\eta$. If $f\:X\ra Y$ is a homeomorphism of $X$ onto a subset of 
$Y$ and satisfies the distortion condition \eqref{qsdef},
then $f$ is called an ${\eta}$-{\em quasisymmetric embedding}. Two metric spaces $X$ and $Y$ are called 
{\em quasisymmetrically equivalent} if there exists a quasisymmetry $f\: X\ra Y$.

If $x_1,x_2,x_3,x_4$ are four  distinct points  
in a metric space $(X,d)$, then   their  {\em cross-ratio} 
is the quantity 
$$[x_1,x_2,x_3,x_4]=\frac{d(x_1,x_3)d(x_2,x_4)}{d(x_1,x_4)d(x_2,x_3)}. $$

Let $\eta\:[0,\infty)\ra[0,\infty)$
be a homeomorphism, 
 and  $f\:X\ra Y$ a homeomorphism between metric spaces $(X,d_X)$ and $(Y,d_Y)$.
The map $f$ is  (an) {\em $\eta$-quasi-M\"obius  (map)} if
$$[f(x_1),f(x_2),f(x_3),f(x_4)]\leq \eta([x_1,x_2,x_3,x_4]).$$
for every $4$-tuple $(x_1,x_2,x_3,x_4)$ of distinct
points in $X$. For these maps we use terminology very similar as for quasisymmetry maps. For example, a  {\em quasi-M\"obius 
embedding}  $f\:X\ra Y$ is a quasi-M\"obius map of $X$ onto  
a subset of $Y$.

Note that a M\"obius transformation on $\Sph$ preserves cross-ratios of points. As a consequence every pre- or post-composition of an $\eta$-quasi-M\"obius map $f\:\Sph\ra \Sph$ with a M\"obius transformation is $\eta$-quasi-M\"obius.

The following proposition records interrelations between  the classes of maps we discussed (see
\cite{Va2} for the proof of the statements).

\begin{proposition}\label{prop:inter}
\begin{itemize}

\smallskip
\item[(i)]  
Every $H$-quasiconformal map $f\:\Sph\ra \Sph$
is $\eta$-quasi-M\"obius
with $\eta$ depending only on  $H$. Conversely, 
every $\eta$-quasi-M\"obius map $f\:\Sph\to\Sph$ is $H$-quasiconformal with 
$H$ depending only on $\eta$.

\smallskip
\item[(ii)] An $\eta$-quasisymmetric map between metric spaces is $\tilde \eta$-quasi-M\"obius
 with $\tilde \eta$ depending only on $\eta$.

\smallskip
\item[(iii)] Let $(X,d_X)$ and $(Y,d_Y)$ be bounded metric spaces,
     $f\: X\ra Y$ an  $\eta$-quasi-M\"obius map, $\la \ge 1$, 
and $x_1,x_2,x_3 \in X$. 
Set 
$y_i=f(x_i)$, and suppose that 
 \begin{equation}\label {sep1}
 d_X(x_i, x_j)\ge \diam(X)/\la
 \end{equation}
  and  
  \begin{equation}\label {sep2}
d_Y(y_i, y_j)\ge \diam(Y)/\la
\end{equation}
 for  $i,j=1,2,3$, $i\ne j$.  
Then $f$ is $\tilde  \eta$-quasisymmetric with $\tilde  \eta$
depending only on $\eta$ and $\la$.
\end{itemize} 
\end{proposition}

The first statement (i) says that  for a homeomorphism $f\:\Sph\ra \Sph$
the properties  of being a quasiconformal map and of being a quasi-M\"obius map are quantitatively equivalent, or,   more informally,  that $f$ is a quasiconformal map if and only if $f$ is   a  quasi-M\"obius map, quantitatively. 

Statements (ii) and (iii) of the previous proposition imply  that 
a homeomorphism $f\:\Sph\ra \Sph$ is a quasisymmetric map if and only $f$ is a quasi-M\"obius map. This statement is not quantitative. For if $f$ is $\eta$-quasi-M\"obius, then $f$ is $\tilde \eta$-quasisymmetric,  
but  we cannot choose $\tilde \eta$ just to depend on $\eta$. If one wants 
a quantitative implication for this direction, one has to introduce additional parameters (such as the parameter $\lambda$ in (iii)).

\medskip
A metric space  $(X,d)$ is called $N$-{\em doubling}, where 
$N\in \N$, if every ball of radius $r>0$ in $X$ can be covered by 
at most $N$ balls in $X$ of radius $r/2$.  Every subset of a doubling 
metric space is also doubling, quantitatively. 

A homeomorphism $f\: X\ra Y$ between metric  spaces
$(X,d_X)$ and $(Y,d_Y)$ is called $H$-{\em weakly quasisymmetric}, $H\ge 1$,  if for all $x,y,z\in X$ the following implication holds:
$$ d_X(x,y)\le d_X(x,z) \Rightarrow  
d_Y(f(x),f(y))\le H d_Y(f(x),f(z)).
$$

Under mild extra assumptions on the spaces weak quasisymmetry of a map implies its quasisymmetry (\cite[Thm.~10.19]{He}).

\begin{proposition} \label{prop:wkqs}
Let $(X,d_X)$ and $(Y,d_Y)$ be metric spaces, and $f\:X\ra Y$ be weakly $H$-quasisymmetric. If $X$ and $Y$ are connected and $N$-doubling, then $f$ is $\eta$-quasisymmetric  with $\eta$ only depending on $N$ and $H$.
\end{proposition}

A metric space is called {\em proper} if every closed ball in the space is compact.  The following lemma will allow  us to extract sublimits of a sequence of quasisymmetric embedding into a proper metric space. 
 
\begin{lemma}[Subconvergence lemma]
\label{lem:subcon}
Let $(X, d_X)$ and $(Y,d_Y)$ be   
 metric spaces such  that $X$ is compact and  $Y$ is proper,   and let 
$f_n\:X\ra Y$  be   
$\eta$-quasisymmetric embeddings for $n\in \N$. 
Suppose that there exists a constant $c>0$, a set $A\sub X$, and a compact set $K\sub Y$ 
  such that 
$$ \diam(f_n(A)) \ge c \quad \text{and}\quad 
f_n(A)\sub K$$
for all $n\in \N$.  
Then the sequence $(f_n)$ subconverges uniformly 
 to an $\eta$-quasisymmetric embedding $g\:X\ra Y$, i.e., there exists an 
increasing  sequence $(n_l)$ in $\N$  such that 
$$ \lim_{l\to \infty} \sup_{x\in X} d_Y( f_{n_l}(x),g(x)) =0.  $$ 
\end{lemma}

As  discussed  at the end of Section~\ref{s:nt},     the intended meaning of  the phrase  ``let $f_n\:X\ra Y$  be   
$\eta$-quasisymmetric embeddings for $n\in \N$"  in this lemma is that the maps $f_n$ are $\eta$-quasisymmetric embeddings with the same distortion function $\eta$ for {\em all} $n$.

\medskip 
\no{\em Proof.} We may assume that $K=
\overline B(y_0,R)$ for some $y_0\in Y$ and $R>0$.  
 
We claim  that the family $(f_n)$ is uniformly bounded (i.e., there exists $R'>0$ such that $f_n(X)\sub \overline B(y_0, R')$ for all $n$) and that  it is  equicontinuous. Let $u,v\in X$ be  arbitrary. We have to show    $d_Y(y_0, f_n(u))$ is uniformly bounded, and that if $d_X(u,v)$ is small,
then $d_Y(f_n(u),f_n(v))$ is uniformly small. 
To see this we consider a fixed map $f=f_n$. For  ease of notation we drop the subscript $n$.

Obviously, $A$ must contain more than one point; so $\diam(A)=a>0$.  
There exist points $x_1,x_2\in A$ such that $ d_X(x_1, x_2)\ge a/2.$ 
We  can pick $x\in \{x_1, x_2\}$ such that 
$d_X(x, u)\ge a/4$. Let $x'$ be the other point in $\{x_1,x_2\}$. 
Note that $f(x), f(x')\in K=\overline B(y_0, R)$ and so $d_Y(f(x'),f(x))\le 2R$. 

Since $f$ is an $\eta$-quasisymmetric embedding, this    implies
\begin{eqnarray*} d_Y(f(u),f(x)) &\le &d_Y(f(x'),f(x)) 
\eta\left(\frac{d_X(u,x)}{d_X(x',x)} \right)\\
&\le & 2R\eta(2\diam(X)/a),
\end{eqnarray*}
and so $ f(u)\in \overline B(y_0,R'), $
where $R'=R(1+2\eta(2\diam(X)/a))$. The uniform boundedness of the sequence $(f_n)$ follows. 

Moreover, 
\begin{eqnarray*} d_Y(f(u),f(v)) &\le &d_Y(f(u),f(x)) 
\eta\left(\frac{d_X(u,v)}{d_X(u,x)} \right)\\
&\le & 2R'\eta(4d_X(u,v)/a).
\end{eqnarray*}
Since $\eta(t)\to 0$ as $t\to 0$ this gives the desired  bound for $d_Y(f(u),f(v))$ that is uniformly small if $d_X(u,v)$ is small. The 
equicontinuity of the sequence $(f_n)$ follows.

By the compactness theorem of Arzel\`a-Ascoli the sequence 
$(f_n)$ subconverges to a continuous map $g\: X\ra Y$ uniformly on $X$. 
Since all the maps $f_n$  are $\eta$-quasisymmetric embeddings, the map 
$g$ satisfies the inequality
$$ d_Y(g(u),g(v))\le d_Y(g(u),g(w))\eta\left(\frac{d_X(u,v)}{d_X(u,w)}\right),$$
whenever $u,v,w\in X$, $u\ne w$.  This inequality implies that 
$g$ is injective and hence a quasisymmetric embedding, or a constant map; but the latter possibility is ruled out, because  a limiting argument shows that $ \diam (g(\overline A))\ge c>0$. 
The proof is complete. \qed  \medskip

\begin{lemma} \label{lem:3norm}
Let $a,b>0$, and $(X,d_X)$ and $(Y,d_Y)$
be metric spaces.  Suppose that $x_1,x_2,x_3\in X$ and $y_1,y_2,y_3\in Y$ are points such that 
  \begin{equation*}
 d_X(x_i, x_j)\ge a \quad \text{and}\quad 
  d_Y(y_i, y_j)\ge b \quad \text{for} \quad i,j=1,2,3,\   i\ne j.
\end{equation*}
  Then for all $x\in X$ and $y\in Y$ there exists an index $l\in \{1,2,3\}$ such that 
 $d_X(x,x_l)\ge a/2$ and $d_X(y,y_l)\ge b/2$.
\end{lemma}

\no{\em Proof.} At most one of the points $x_i$ can lie 
in the ball $B(x,a/2)$. So there are at least two of the points 
$x_i$, say $x_1$ and $x_2$,  that have distance $\ge a/2$ to $x$. At most one of the points $y_1$ and $y_2$ can lie in   $B(y,b/2)$; so one, say $y_1$, has to lie outside this ball.
Then $l=1$ is an index as desired.\qed \medskip
 
% \begin{lemma} \label{lem:2norm}
%Let  $(X,d_X)$ and $(Y,d_Y)$ be bounded metric spaces,
%and $f\: X\ra Y$ be a biiection. 
%Then there exist points $x_1,x_2\in X$ such that 
%  \begin{equation*}
% d_X(x_1, x_2)\ge \diam(X) /3\quad {and}\quad 
%  d_Y(f(x_1), f(x_2))\ge \diam(Y)/3. \
%\end{equation*}
% \end{lemma}

%\no
%{\em Proof.} Assume this is not true,  and let $u,v\in X$ be arbitrary.
%If  $d_X(u,v) \ge \diam(X) /3$, then $d_Y(f(u), f(v))<\diam(Y)/3$ by assumption. 

%In the other case, where  $d_X(u,v) <\diam(X) /3$,
%there exists a point $w\in X$ such that 
%$d_X(u,w) \ge \diam(X) /3$ and $d_X(v,w) \ge \diam(X) /3$.
%For otherwise,
%$$X\sub B(u, \diam(X) /3)\cup B(v, \diam(X) /3)$$ 
%which gives the contradiction
%$$ \diam(X)\le (2/3)\diam(X) +d_X(u,v)<\diam (X). $$
%Hence 
%$$d_Y(f(u), f(v))\le d_Y(f(u), f(w))+d_Y(f(w), f(v))<
%(2/3)\diam(Y). $$

% Since $u,v\in X$ were arbitrary and  $f$ is  a biiection, we conclude that 
% $d_Y(y,z)<(2/3)\diam(Y)$ for all $y,z\in Y$. This shows 
% that $\diam(Y)>0$ and $\diam (Y)\le (2/3)\diam(Y)$. This 
% is  contradiction. \qed

%

%We need the following  deep result 
%due to  Tukia and V\"ais\"al\"a~\cite{TV82}.  

%\begin{theorem}\label{TV-Thm} Let $n \ge 3$. 
%Every $H$-quasiconformal map  $f\: \R^{n-1} \ra \R^{n-1}$   
% has an $H'$-quasiconformal  extension $F\co \R^n \ra \R^n $,
%where $H'$ only depends on $n$ and  $H$.
%\end{theorem}

\section{Quasicircles} \label{s:quasicircles}

\no A Jordan  curve $J\sub \Sph$ is called a {\em $k$-quasicircle} for  $k\ge 1$ if it 
satisfies condition \eqref{eq:qcirc}, that is, whenever $x,y\in J$,
$x\ne y$,  are arbitrary, then 
$$\diam(\ga)\le k  \sig(x,y)$$
for one of the subarcs $\ga$ of $J$ with endpoints $x$ and $y$. The curve $J$ is called a {\em quasicircle} if it is a $k$-quasicircle for 
some $k\ge 1$. A family $\mathcal{S}=\{S_i: i\in I\}$ of Jordan  curves $S_i$ in $\oC$ is said to consist 
of {\em uniform quasicircles} if there exists $k\ge 1$ such that $S_i$ is a $k$-quasicircle for each $i\in I$.

Various equivalent characterizations  of quasicircles and {\em quasidisks} (Jordan domains bounded by quasicircles)  are known  (see, for example, 
\cite{Geh, Geh2}). 
Up to bi-Lipschitz equivalence all quasicircles can be constructed by a procedure similar to the one used in the definition of the von Koch snowflake curve \cite{Roh}.

The  following proposition is essentially due to 
Ahlfors \cite{Ahl}. See  \cite[Ch.~II, \S8]{LV} for a discussion of related facts. 

\begin{proposition}\label{prop:planeqcirc} 
Suppose that $J\sub \Sph$ is a Jordan  curve.  Then the following conditions are quantitatively equivalent: 

\begin{itemize} 

\smallskip
\item[{(i)}] $J$ is a $k$-quasicircle, 

\smallskip
\item[{(ii)}] $J$ is  the image of a round circle under an
 $H$-quasiconformal map $f\: \Sph\ra \Sph$, 

\smallskip
\item[{(iii)}]  $J$ is  the image of a round circle under an
 $\eta$-quasi-M\"obius map $f\: \Sph\ra \Sph$. 

\end{itemize}
\end{proposition} 

As discussed before, ``quantitative" equivalence  here  means that if condition (i) is true, then $H$ in condition (ii) can be chosen only depending on $k$ in (i), etc. 
The equivalence $\text {(i)}\Leftrightarrow \text{(ii)} $ is contained in \cite{Ahl}, while $\text {(ii)}\Leftrightarrow \text{(iii)} $ follows from Proposition~\ref{prop:inter}~(i). An immediate consequence of Proposition~\ref{prop:planeqcirc} is the following fact: if $D\sub \OC$ is a closed Jordan region whose boundary $\partial D$ is a $k$-quasicircle, then there exists an $\eta$-quasi M\"obius map $f\: \OC \ra \OC$ with $\eta=\eta_k$ such that $f(\overline {\D})=D$. 
\medskip

The following lemma shows that the diameter of every Jordan curve in $\oC$ is equal to  the diameter of one of the closed Jordan regions whose boundary it is.

\begin{lemma}\label{lem:Jordiam} Let $D\sub \oC$ be a closed Jordan region. Then $\diam (D)= \diam(\partial D)$, or  $\diam(D)=2$ and 
$\diam(\oC\setminus \inte(D))=  \diam(\partial D)$. 
\end{lemma}

\no{\em  Proof.} We first prove an  elementary  geometric fact. To state it we identify $\oC$ with the unit sphere in $\R^3$ by stereographic 
projection, and  denote by $A\: \oC\ra \oC$ the involution that assigns to each point $p\in \oC$ its antipodal point (so $A$  is the conjugate of the map $p\mapsto -p$  on the unit sphere by stereographic projection).  We say that $p,q\in \oC$ form a   {\em pair 
of antipodal points} in $\oC$ if  $q=A(p)$ (then also $p=A(q)$). 
Now suppose $J\sub \oC$ is a Jordan curve, and $U$ and $V$ are the the closures of the two components  of $\oC\setminus J$. We claim that if each of the sets  $U$ and $V$ contains a  pair of antipodal points,  then $J$ also contains such a pair.

To see this we argue by contradiction and assume that  $J$ contains no such pair. In this case  $J\cap A(J)=\emptyset$, and so   the Jordan curve 
$A(J)$ must be contained   in one of the  closed Jordan regions 
$U$ and $V$, say $A(J)\sub U$.  Then  $U$   must contain one  of the closed Jordan regions  $A(U)$ and $A(V)$ bounded by $A(J)$.  

Suppose that  $A(U)\sub U$. Since $A$ is an involution, this implies $U\sub A(U)$, and so $A(U)=U$. Then we  have $A(J)=A(\partial U)=\partial U=J$, which contradicts our assumption 
$J\cap A(J)=\emptyset$. 

So we must have $A(V)\sub U$.  By  our hypotheses there exists an
antipodal pair $\{v,A(v)\}\sub V$.   Then $\{v,A(v)\}\sub A(V)\sub U$, and so $\{v,A(v)\}\sub U\cap V=J$. This contradicts our assumption that $J$ contains no antipodal pair. The claim follows. 

Now let $D\sub \oC$ be an arbitrary closed Jordan region. Then there exist points $x,y\in D$ with $\sigma(x,y)=\diam(D)$. If $x,y\in \partial D$, 
then  $\diam (D)\le \diam(\partial D)$, and so $\diam(D)=\diam(\partial D)$.   

In the other case when  $x,y$ do not both belong to $\partial D$,  one of these points must be an interior point of $D$, say $y\in \inte(D)$. Consider a minimizing spherical geodesic segment joining $x$ and  $y$. If we were able to  slightly extend this  geodesic segment  beyond  $y$ to a minimizing geodesic segment, then we would obtain a point $y'\in D$ near $y$  whose spherical to $x$ is strictly 
larger than the distance of $y$ to $x$. Since there is a strictly monotonic relation between spherical and chordal distance (spherical distance $s\in [0,\pi]$ corresponds to chordal distance $2\sin(s/2)$), we would also have $\sigma(x,y')>\sigma(x,y)$. This is impossible, since 
$x,y'\in D$ and 
$\sigma(y,x)=\diam(D)$. So the  geodesic segment between $x$ and $y$ is not extendible as a minimizing geodesic segment and it must have length $\pi$. Then  $x$ and $y$ form  a pair of antipodal points which implies $\diam(D)=\sigma(x,y)=2$.   

It now follows that  $\diam(D')= \diam(\partial D)$ if $D'=\oC\setminus \inte(D)$  denotes   the other Jordan region  bounded by $\partial D$. Indeed, by applying the first part of the argument also to $D'$, we see that the only case where this may possibly fail  is if  both $D$ and $D'$  contain a pair of antipodal points. By our claim established  in the beginning of the proof, $\partial D$ then contains such a pair as well and we get  the desired relation  $\diam(D')=2=\diam(\partial D)$  anyway. \qed \medskip 

The following proposition is a standard fact. We record a proof for the sake of completeness.

\begin{proposition}\label{prop:qround}
Suppose $D\sub \Sph$ is a closed Jordan  region 
whose boundary $\partial D$ is a  $k$-quasicircle. Then there exists $\la=\la(k)\ge 1$, $x_0\in D$, and $r\in(0,2]=(0,\diam(\oC)]$  such that 
\begin{equation}
\label{eq:qball}
\overline B(x_0, r/\la) \sub D\sub \overline B(x_0,r).
\end{equation}
\end{proposition} 

\no 
{\em Proof.}  Let $d=\diam(\partial D)$. 
We first consider the case where 
\begin{equation}\label{cas1}
\diam(D)>2 d.
\end{equation} 
Since $\diam (D)\le \diam(\Sph)=2$, this implies 
$d<1.$
Pick a point $p\in \partial D$, and let $x_0$ be the antipodal point of $p$ on $\Sph$ (considered as the unit sphere in $\R^3$). Then $\partial D\sub \overline B(p, d)$. Therefore, 
the connected set $B(x_0, 2-d)\sub  \Sph\setminus 
\overline B(p, d)$ does not meet $\partial D$ and must hence be contained in one of the two closed Jordan  regions bounded by  $\partial D$.
The other Jordan  region must be contained in 
$\overline B(p, d) $, and so has  diameter 
$\le 2d$. By our assumption \eqref{cas1} this cannot be $D$. Hence $B(x_0, 2-d)\sub D$. 
Note that $2-d> 1$.
Picking $r=2$ and $\lambda=2$ we see that we get the desired inclusion
$$  \overline B(x_0, r/2) =\overline B(x_0, 1) \sub B(x_0, 2-d) \sub D\sub \Sph =\overline B(x_0,r).$$

In the remaining case we have
\begin{equation}\label{cas2}
\diam(D)\le2 d.
\end{equation} 
The set 
$D$ is the image of the closed unit disk
$$\overline \D=\{z\in \C: |z|\le1\}$$ under an $\eta$-quasi-M\"obius map $f\: \Sph \ra \Sph$, where $\eta$ only depends on $k$ (see the remark after Proposition~\ref{prop:planeqcirc}). 
We use a prime to denote image points under $f$, i.e., 
$x'=f(x)$ for $x\in \Sph$.

We can pick points $x_1,x_2, x_3\in \partial \D$
such that for their image points we have 
\begin{equation}
\sig(x_i',x_j')\ge d/2  \forr i\ne j.
\end{equation}
By pre-composing $f$ with a M\"obius transformation if necessary, we may assume  the points $x_i$ are the third roots of unity.
Then 
 $$ \sigma(x_i,x_j)=\sqrt 3  \forr i\ne j.$$

Define  $z=0$, and  let $u\in \overline \D$ and 
 $v\in \partial \D$ be arbitrary. By Lemma~\ref{lem:3norm} there exists 
 $w\in \{x_1, x_2,x_3\}$ such that 
 $$ \sigma(u,w)\ge \sqrt 3/2 \ge 1/2 \quad \text{and} \quad 
 \sigma(v',w')\ge d/4. $$
 We also have the inequalities $\sig(u,z)\le \sig(v,z)$,  
 $$\sig(v,w)\le \diam(\overline \D)= 2,$$ and 
 $$\sig(u',w')\le \diam( D)\le 2d.$$ 
 Since $f$ is $\eta$-quasi-M\"obius, we obtain   
 \begin{eqnarray}
 \label{crest10}
 \frac {\sig(u',z')} {\sig(v',z')}  &\le &
 \eta\biggl( 
  \frac{\sig(u,z)\sig(v,w)}
  {\sig(v,z)\sig(u,w)}  \biggr)
\frac {\sig(u',w')}{\sig(v',w')} \nonumber \\
&\le & \eta\biggl( 
  \frac{\sig(v,w)}
  {\sig(u,w)}  \biggr)
\frac {\sig(u',w')}{\sig(v',w')} \\
&\le &8 \eta(4)=:\lambda.  \nonumber
\end{eqnarray}
Since $\eta$ only 
depends on $k$, the same is true for $\lambda$ defined in the last line. 

Since $u\in 
\overline  \D$ and $v\in \partial \D$ in  \eqref{crest10} were arbitrary,  we conclude that 

\begin{equation}\label{crux}
 \sup_{x\in D}\sig(x, z')\le\la \inf_{x\in\partial D}\sig(x, z').
\end{equation}
 Now define $x_0=z'=f(0)\in \inte(D)$ and $r=\sup_{x\in D}\sig(x, z')\in (0,2]$.
 Then $D\sub \overline B(x_0, r)$ by definition of $x_0$ and $r$. 
 Moreover, by \eqref{crux} the set $B(x_0, r/\lambda)$
 is disjoint from $\partial D$. So  this disk must be contained in one of the open Jordan  regions bounded by $\partial D$. Since its center is contained in $\inte(D)$,  it follows
 that $B(x_0, r/\lambda)\sub \inte(D). $ Passing to closures we get the desired inclusion  \eqref{eq:qball}.
\qed
 
\medskip
  A  {\em metric circle} $S$ (that is, a metric space homeomorphic to a circle) is called a {\em (metric) quasicircle} if there exists a quasisymmetry $f\: \partial \D\ra S$ of the unit circle  
  $\partial \D\sub \OC$ onto $S$.   
  Four distinct points  $x_1,x_2,x_3,x_4$ on a metric circle $S$ are in {\em cyclic order}  if $x_2$ and $x_4$ lie in different components of $S\setminus\{x_1, x_3\}$.

  Similarly as for quasicircles 
 in $\Sph$,  metric quasicircles  admit various characterizations. We will record some of them in the next proposition. 
\begin{proposition}\label{prop:metricqcirc} 
Suppose $(S,d)$ is a metric space homeomorphic 
to a circle. Then the following conditions are quantitatively equivalent: 

\begin{itemize} 

\smallskip
\item[{(i)}] there exists an $\eta$-quasisymmetric map 
$f\:\partial \D\ra S$, 

\smallskip
\item[{(ii)}] there exists a round circle $S'\sub \OC$ and an $\tilde \eta$-quasi-M\"obius map $g\: S'\ra S$,  

\smallskip
\item[{(iii)}] $S$ is $N$-doubling and there exists $k\ge 1$ such that 
$$ \diam(\ga) \le k d(x,y)$$ for one of the subarcs of $S$ with endpoints $x$ and $y$, whenever $x,y\in S$, $x\ne y$,

\smallskip
\item[{(iv)}] $S$ is $\widetilde N$-doubling and there exists $\de >0$ such for all
points $x_1,x_2,x_3,x_4\in S$ in cyclic order on $S$ we have $[x_1,x_2,x_3,x_4]\ge \delta>0$.  
\end{itemize}
\end{proposition} 

Note that every subset of $\OC$ is $N$-doubling for  a universal constant $N$. So 
condition (iii) in the  proposition implies  that a  Jordan curve $S\sub \OC$ equipped with the chordal metric is a metric quasicircle if and only if  it is a quasicircle as defined in the beginning of this section. So for Jordan curves $S\sub \OC$
the notions of metric quasicircle and quasicircle agree, quantitatively.

We will prove  Proposition~\ref{prop:metricqcirc} below. It 
 goes back to  
 Tukia and V\"ais\"al\"a \cite{TV} whose work implies the equivalence of the first three conditions. For the fourth equivalence  
 it is convenient to introduce a  quantity that is quantitatively equivalent to the
cross-ratio and  is somewhat more manageable (see \cite[Sec.~2]{BK02}).

If $(x_1,x_2,x_3,x_4)$ is a $4$-tuple of distinct points in a metric space
$(X,d)$ define
\begin{equation} \label{cr1a} \langle x_1,x_2,x_3,x_4\rangle
:= \frac{d(x_1,x_3)\wedge d(x_2,x_4)}{d(x_1,x_4)\wedge
d(x_2,x_3)}.
\end{equation}
 Then the following is true (this is essentially \cite[Lem.~2.2]{BK02}; we include a proof for the convenience of the reader). 

\begin{lemma}\label{lem:modcr}
\label{cr2} Let $(X,d)$ be a metric space,  and define 
 $\eta_1(t)=\frac 13(t\wedge \sqrt t)$  and  $\eta_2(t)=3(t\vee \sqrt t)$ for $t> 0$.
Then whenever  $x_1,x_2,x_3,x_4$ are  distinct points in
$X$ we have
\begin{equation}  \label{cr3}
\eta_1([x_1,x_2,x_3,x_4])\le \langle x_1,x_2,x_3,x_4\rangle  \le \eta_2([x_1,x_2,x_3,x_4]).
\end{equation}
\end{lemma}
 
 The point of the lemma is that it shows that the cross-ratio $[x_1,x_2,x_3,x_4]$ is small if and only if the ``modified" cross-ratio $\langle x_1,x_2,x_3,x_4\rangle$ is small, quantitatively.

\proof  We first prove the second inequality. Suppose that there exist distinct points $x_1,x_2,x_3,x_4$ in $X$ 
 for which $$\langle x_1,x_2,x_3,x_4\rangle  > \eta_2([x_1,x_2,x_3,x_4]).$$
  Let $t_0=[x_1,x_2,x_3,x_4]$.
 We may assume $d(x_1,x_3)\le d(x_2,x_4)$.
 Then our assumption implies 
 $$d(x_1,x_4) \wedge d(x_2,x_3)< \frac1{\eta_2(t_0)}d(x_2,x_4). $$

Moreover, we have
\begin{eqnarray*}
d(x_1,x_4) & \le & d(x_1,x_3) + d(x_3, x_2) + d(x_2,x_4) \\
           & \le & 2d(x_2,x_4)+ d(x_2,x_3).
\end{eqnarray*}
Similarly, $ d(x_2,x_3)\le 2d(x_2,x_4) + d(x_1,x_4)$, and so $$ |d(x_1,x_4)-d(x_2,x_3)|\le 2d(x_2,x_4), $$
which implies  
$$d(x_1,x_4) \vee d(x_2,x_3) \le
 2 d(x_2,x_4) + d(x_1,x_4) \wedge d(x_2,x_3). $$

Hence
\begin{eqnarray*}
 d(x_1,x_4) \vee d(x_2,x_3) &\le&
 2 d(x_2,x_4) + d(x_1,x_4) \wedge d(x_2,x_3)\\
&\le & 
\left(2+\frac 1{\eta_2(t_0)}\right)  d(x_2,x_4),
\end{eqnarray*}
and so
\begin{eqnarray*}
 t_0=[x_1,x_2,x_3,x_4] &=&
\frac{d(x_1,x_3) d(x_2,x_4)}
{(d(x_1,x_4) \wedge d(x_2,x_3))(d(x_1,x_4)\vee d(x_2,x_3))}\\
&\ge &   
\frac {d(x_1,x_3)\eta_2(t_0)} {(d(x_1,x_4)\wedge d(x_2,x_3))( 1+2 \eta_2(t_0))}
\ge \frac{ \eta_2(t_0)^2}
{ 1+2 \eta_2(t_0)}> t_0.
\end{eqnarray*}
Here the last inequality follows from  a simple computation based on the cases $0< t_0\le 1$ and $t_0>1$ which is left to the reader. In conclusion, we obtain a contradiction showing 
the second inequality in \eqref{cr3}. 

The first inequality in  \eqref{cr3} follows from the second, if one uses the symmetry relations
$$ [x_2,x_1,x_3,x_4]= 1/[x_1,x_2,x_3,x_4] \quad \text{and}\quad
\langle x_2,x_1,x_3,x_4\rangle = 1/\langle x_1,x_2,x_3,x_4 \rangle, $$ and the fact that $\eta_1(t)=1/\eta_2(1/t)$ for $t>0$.
 \qed

\medskip\no 
{\em Proof of Proposition~\ref{prop:metricqcirc}.} The quantitative equivalence of the first three conditions is contained in \cite{TV}. 

To finish the proof it is enough to show that (iii) and (iv) are quantitatively equivalent. 

$\text{(iii)}\Rightarrow\text{(iv)}$: Let $x_1,x_2,x_3,x_4$ be four distinct points in cyclic order on $S$.   We may assume 
$d(x_1,x_3)\le d(x_2,x_4)$. Denote by  $\ga_1$ and $\ga_2$ the subarcs of $S$ with endpoints $x_1$ and $x_3$ that contain 
the points $x_2$ and $x_4$, respectively. Condition (iii) gives us 
the inequality
$$ d(x_2,x_3)\wedge d(x_1,x_4)\le \diam(\ga_1)\wedge \diam (\ga_2)
\le k d(x_1,x_3). $$
Hence 
$$ \langle x_1,x_2,x_3,x_4\rangle = \frac{d(x_1,x_3)}
{d(x_2,x_3)\wedge d(x_1,x_4)}\ge \frac 1k.$$ 
By Lemma~\ref{lem:modcr} this implies that $[x_1,x_2,x_3,x_4]\ge \de$, where $\de=\de(k)>0$ only depends on $k$. 
 
 $\text{(iv)}\Rightarrow\text{(iii)}$: Let $x,y\in S$ with 
 $x\ne y$ be arbitrary, and denote by $\ga_1$ and $\ga_2$ 
 the two subarcs of $S$ with endpoints $x$ and $y$. 
 Define $x_1:=x$ and $x_3:=y$. There exists a point 
 $x_2\in \ga_1\setminus \{x_1,x_3\}$  such that 
 $$d(x_2,x_3)\ge \frac 13 \diam(\ga_1). $$
 For otherwise, $\ga_1$ would be contained in the closed ball 
 of radius $\frac13 \diam(\ga_1)$ centered at $x_3$ which is impossible.
 
 Similarly, there exists a points $x_4\in \ga_2\setminus \{x_1,x_3\}$  such that 
 $$d(x_1,x_4)\ge \frac 13 \diam(\ga_2). $$
 
 The points $x_1,x_2,x_3,x_4$ are in cyclic order on $S$. Hence $[x_1,x_2,x_3,x_4]\ge \de>0$ by our  hypothesis (iv), and so  
 Lemma ~\ref{lem:modcr}  implies that 
 $\langle x_1,x_2,x_3,x_4\rangle \ge \eps_0$, where 
 $\eps_0=\eps_0(\de)>0$ only depends on $\de$.
It follows that 
\begin{eqnarray*} \diam (\ga_1)\wedge \diam(\ga_2) &\le & 
3 d(x_2,x_3)\wedge d(x_1,x_4)\\ &\le& \frac{3}{\eps_0}  
d(x_1,x_3)\wedge d(x_2,x_4)\le kd(x_1,x_3),
\end{eqnarray*}
where $k=k(\de)=3/\eps_0.$ This inequality shows that (iii) is true. 
\qed

\medskip 
 We will give another application of the modified cross-ratio defined in \eqref{cr1a}.  We require the following fact. 
 
 \begin{lemma} \label{lem:crrelsep}
 Let $(X,d)$ be a metric space, and $E$ and $F$ disjoint continua   in $X$.
 Define 
 $$D(E,F)=\inf_{x_1,x_4\in E,\, x_2,x_3\in F}\langle x_1,x_2,x_3,x_4\rangle.
 $$
 Then
 \begin{equation}\label{reldcr}  \Delta(E,F)
 \le D(E,F) \le 2\Delta(E,F).
\end{equation}
\end{lemma}

Recall that a {\em continuum} (in a metric space) is a compact connected set consisting of more than one point. 
The  inequality in the lemma shows that the relative separation $\Delta(E,F)$ of two continua  $E$ and $F$ is small if and only if $D(E,F)$ is small, quantitatively. 

\medskip
\no{\em Proof.} It follows from the definitions that 
\begin{eqnarray*} 
 \langle x_1,x_2,x_3,x_4\rangle &=& \frac{d(x_1,x_3)\wedge d(x_2,x_4)}
 {d(x_1,x_4)\wedge d(x_2,x_3)} \\
 &\ge & \frac{\dist(E,F)}{\diam(E)\wedge \diam(F)}
 \end{eqnarray*}
 whenever $x_1,x_4\in E$ and $x_2,x_3\in F$. The first inequality in 
 \eqref{reldcr} follows.
 
 For the second inequality choose $x_1\in E$ and $x_3\in F$ such that 
 $d(x_1,x_3)=\dist (E,F)$. Then we can select  points $x_4\in E$ and $x_2\in F$ such that $d(x_1,x_4)\ge \frac 12 \diam (E)$ and 
  $d(x_2,x_3)\ge \frac 12 \diam (F)$. Hence 
\begin{eqnarray*}    D(E,F)& \le& \langle x_1,x_2,x_3,x_4\rangle
  \,\le \, \frac{d(x_1,x_3)\wedge d(x_2,x_4)}
 {d(x_1,x_3)\wedge d(x_2,x_4)}\\
 &\le & 2\frac{\dist(E,F)}{\diam(E)\wedge \diam(F)}\,= \,2\Delta(E,F).
\end{eqnarray*}
The second inequality follows. 
\qed

\medskip 
Let $(X,d)$ be a metric space, and $\mathcal{S}=\{S_i:i\in I\}$ be a collection of pairwise disjoint continua in $X$.  We say that the sets in 
$\mathcal{S}$ are {\em $s$-relatively separated} for $s>0$ if 
$$ \Delta(S_i,S_j)\ge s$$
whenever $i,j\in I$, $i\ne j$. The sets in  $\mathcal{S}$ are  said to be  {\em uniformly relatively separated}  if they are  $s$-relatively separated for some $s>0$. 

\begin{corollary} \label{cor:qmobinvcond}
Let $\mathcal{S}=\{S_i: i\in I\}$ be a family of $s$-relatively separated 
$k$-quasicircles in $\Sph$, and $f\: \Sph \ra \Sph$ be an $\eta$-quasi-M\"obius map. Then the image family $\mathcal{S}'=\{f(S_i): i\in I\}$ 
consists  of $s'$-relatively separated $k'$-quasicircles, where 
$s'=s'(\eta, s)>0$ and $k'=k'(\eta,k)\ge 1.$
\end{corollary}

\no{\em Proof.} It follows from Lemmas~\ref{lem:modcr} and \ref{lem:crrelsep} that there 
exists a constant $s_1=s_1(s)>0$ such that 
$[x_1,x_2,x_3,x_4]\ge s_1$ whenever 
$i,j\in I$, $i\ne j$, $x_1,x_4\in S_i$, and $x_2,x_3\in S_j$.
Since the quasi-M\"obius map $f$  distorts cross-ratios of points  quantitatively  controlled by $\eta$, this implies  that there exists 
$s_2=s_2(\eta,s_1)=s_2(\eta,s)>0$ such that 
$[y_1,y_2,y_3,y_4]\ge s_2,$
whenever 
$i,j\in I$, $i\ne j$, $y_1,y_4\in f(S_i)$, and $y_2,y_3\in f(S_j)$.
Again invoking  Lemmas~\ref{lem:modcr} and \ref{lem:crrelsep} we see that the sets in $\mathcal{S}'$ are 
$s'$-relatively  separated, where 
$s'=s'(s_2)=s_2(\eta,s)>0$. 

By Proposition~\ref{prop:metricqcirc}  there exists $\eta'=\eta'_k$ such that each set in $\mathcal{S}$ is the image of a round circle under
an $\eta'$-quasi-M\"obius map on $\Sph$. Hence each 
set in $\mathcal{S}'$ is the image  of a round circle under
an $\eta''$-quasi-M\"obius map on $\Sph$, where $\eta''=\eta\circ 
\eta'=\eta''_k$. 
Another application of Proposition~\ref{prop:metricqcirc} shows that
the sets in  $\mathcal{S}'$ are $k'$-quasicircles, where 
$k'=k'(\eta'')=k'(\eta,k)$. \qed\medskip

We conclude this section with a lemma that implies that it does not matter in Theorem~\ref{thm:simulunif} whether we assume uniform relative 
separation  for the  
curves in $\mathcal{S}$,  or for the pairwise disjoint Jordan regions that the curves in $\mathcal{S}$ bound.

\begin{lemma}\label{lem:curvsepp} Let $D$ and $D'$ be 
disjoint Jordan regions in $\oC$. Then $\Delta(D,D')=\Delta(\partial D, \partial D')$. 
 \end{lemma}

An immediate consequence is that if   $\{D_i: i\in I\}$ is    a family of pairwise disjoint closed Jordan regions in $\oC$,  then this family is
$s$-relatively separated if and only if  the family $\{\partial D_i:i\in I\}$ of boundary curves is 
$s$-relatively separated.\medskip

\no{\em Proof.}   We can pick points $x\in D$ and $y\in D'$ such that 
$\dist(D,D')=\sigma(x,y)$. If we run on a minimizing spherical geodesic 
segment from $x$ to $y$, then we must meet $\partial D$ and $\partial D'$. 
This implies that  the spherical distance between the sets $\partial D$ and $\partial D'$ is no larger than the spherical distance between $D$ and $D'$.  Since spherical distances and chordal distances are monotonically related, it follows that $\dist(\partial D, \partial D')\le  \dist(D,D')$, and so 
\begin{equation}\label{eq:distDp}
\dist(\partial D, \partial D')=  \dist(D,D').
\end{equation}

Moreover, by Lemma~\ref{lem:Jordiam} we have 
$$\diam(\partial D)\ge \diam (D)\wedge \diam(\oC\setminus \inte(D))\ge \diam(D)\wedge \diam (D').$$
We also get the same lower bound for $\diam(\partial D')$, and so 
$$\diam(\partial D)\wedge \diam(\partial D')\ge \diam(D)\wedge \diam(D').$$
The reverse inequality is trivially true, which gives
$$\diam(\partial D)\wedge \diam(\partial D')= \diam(D)\wedge \diam(D'). $$
If we combine this with \eqref{eq:distDp}, the claim follows.\qed  

\section{Extending quasiconformal maps} 
\label{s:ext}

\no 
In this section we will prove the following proposition that will be used in the proof of Theorem~\ref{thm:simulunif}. Its proof is very similar to the considerations in \cite[Sec.~4]{BKM}. 

\begin{proposition}\label{prop:extend}
Suppose  that $\{D_i: i\in I\}$ is a non-empty family of pairwise disjoint closed 
 Jordan  regions  in $\Sph$, and let $f\: 
T=\Sph\setminus \bigcup_{i\in I}\inte(D_i) \ra \Sph$ be an 
$\eta$-quasi-M\"obius embedding. 
If the Jordan  curves  $S_i=\partial D_i$ are  $k$-quasicircles for $i\in I$, then 
there exists an $H$-quasiconformal map 
$F\: \Sph\ra \Sph$ such that $F|T=f$ where $H=H(\eta,k)$.  
\end{proposition}

We need  the classical Beurling-Ahlfors \cite{BA} extension theorem
that can be formulated as follows. 

\begin{theorem}[Beurling-Ahlfors 1956]\label{thm:AB}
Every $\eta$-quasisymmetric   map 
$f\co \R\ra \R$ 
 has an $H$-quasiconformal  extension $F\co \C \ra \C$,  
where $H$ only depends on $\eta$.    
\end{theorem}  
Here $\R$ and $\C$ are  equipped with the Euclidean metric.
See \cite[p.~83, Thm.~6.3]{LV} for a streamlined proof of an  equivalent version of this theorem. 

The next proposition is a  consequence of this result. 

\begin{proposition} \label{prop:qcircext}
 Let $D$ and $D'$ be closed Jordan  regions in $\Sph$,  
and $f\:\partial D\ra \partial D'$  be a homeomorphism. Suppose that the Jordan  curve 
$\partial D$ is a $k$-quasicircle.

\begin{itemize} 

\item[\textrm(i)]
If $f$ is $\eta$-quasi-M\"obius map, then it 
can be extended to 
an $\eta'$-quasi-M\"obius  map $F\:\Sph \ra \Sph$ with $F(D)=D'$,
where $\eta'$ only depends on $\eta$ and $k$.

\smallskip 
\item[\textrm(ii)]
If  $f$ is  $\eta$-quasisymmetric and 
\begin{equation} \label{complbig}
\diam(\Sph\setminus D)\wedge \diam(\Sph\setminus D')\ge \delta>0, 
\end{equation}  
then $f$ can be extended to an   $\eta'$-quasisymmetric map 
 $F\: D \ra D'$,
where $\eta'$ only depends on $\delta$, $k$ and $\eta$. 

\end{itemize} 
\end{proposition} 

For a related  result with a similar proof see   \cite[Prop.~4.3]{BKM}. 
\medskip 

\no
{\em Proof.} We first prove (i). In this case $\partial D'$ is the
 image of a $k$-quasicircle under an $\eta$-quasi-M\"obius map. Hence by the quantitative equivalence of conditions~(ii) and (iii) in  
 Proposition~\ref{prop:metricqcirc},  the curve 
 $\partial D'$ is a $k'$-quasicircle with $k'=k'(\eta, k)$. It follows from Proposition~\ref{prop:planeqcirc} that 
 there exist $\tilde \eta$-quasi-M\"obius maps  on $\Sph$ with 
$\tilde \eta=\tilde \eta_{k,\eta}$  that map $\partial D$ and $\partial D'$ to
  $\widehat \R=\R\cup\{\infty\}\sub \OC$  and the sets  $D$ and $D'$ to the closed upper half-plane 
  $U=\{z\in \C: \text{Im}\, z\ge0\} \cup\{\infty\} $ in $\oC$.  
  So we are reduced to the case where $D=D'=U$,
   and $f$ is an $\eta$-quasi-M\"obius  homeomorphism on $\widehat \R$.
      By pre-  and  post-composing $f$ with suitable M\"obius transformations, which does not change the  distortion function $\eta$ of the map, we may further assume that 
      $f(\infty)=\infty.$
      
  Note that here $\widehat\R$ 
has to be considered as equipped with the chordal metric. Cross-ratios for points in $\R$ are  the same 
if we take the chordal metric or the Euclidean metric.  
It follows that $ f|\R \: \R \ra \R$ 
is  $\eta$-quasi-M\"obius if $\R$ is equipped with the 
Euclidean metric.  Since $f(\infty)=\infty$ 
a limiting argument shows that $f|\R\: \R\ra \R$ 
is also  $\eta$-quasisymmetric when $\R$ carries this metric.
By  the Beurling-Ahlfors  Theorem~\ref{thm:AB} the map 
$f|\R$ has an $H$-quasi\-conformal extension $F\:\C \ra \C$ where 
$H=H(\eta)$.   Post-composing $F$ with the reflection of $\OC$ in $\overline \R$ if necessary, we 
may assume that this $H$-quasi\-conformal extension $F$ of $f$ satisfies
$F(U)=U$.

Letting $F(\infty)=\infty$ we get an $H$-quasiconformal 
mapping $ F\: \OC\ra \OC$ that extends $ f$. Note that points are ``removable singularities" for quasiconformal maps \cite[Thm.~17.3]{Va}. Moreover, the dilatation of $ F$ 
 does not change by the passage 
from the Euclidean metric on $\C$ to the chordal metric on $ \C \sub
\OC$, because these metrics are ``asymptotically" conformal, i.e., the identity map from $\C$ equipped with the Euclidean metric to $\C$ equipped with the chordal metric is $1$-quasiconformal.     
Then $F$ will be $\eta'$-quasi-M\"obius with $\eta'$ only depending 
on  $H$ and hence only on $\eta$.  So the map $F$ is an extension of $f$ with the desired properties.

To prove  part (ii) we first show that our assumption 
$\diam(\Sph\setminus D)\ge \delta$ implies that 
$$\diam(D)\le \frac 2\delta \diam(\partial D). $$ Indeed, note that 
$\delta \le \diam(\Sph\setminus D)\le \diam(\oC)=2$. Hence by 
Lemma~\ref{lem:Jordiam} we have that
$$ \frac 2\delta \diam(\partial D)\ge \frac 2\delta
 (\diam(D)\wedge \diam (\Sph\setminus D))\ge \diam(D)\wedge 2=\diam(D)$$ as desired.  
Similarly, 
$$\diam(D')\le \frac 2\delta  \diam(\partial D'). $$

Now suppose that $f\: \partial D \ra \partial D'$ is $\eta$-quasisymmetric.
Since quasisymmetric maps 
are quasi-M\"obius maps, quantitatively (Proposition~\ref{prop:inter}~(ii)), it follows from the first part 
of the proof that there exists
an $\tilde \eta$-quasi-M\"obius extension $F\: D \ra D'$, 
where  $\tilde \eta$ only depends on $k$ and $\eta$. 
We can pick points $x_1,x_2,x_3 \in \partial D$ such that 
$$ \sig(x_i,x_j)\ge \diam(\partial D)/2 \ge \frac \delta 4\diam(D) \forr i\ne j,$$ 
and define $y_i=F(x_i)=f(x_i)\in \partial D'$.  
Now, since $f$ is an $\eta$-quasisymmetry, we have  
$$ \sig(f(z),f(x_i))\le \eta(2) \sig(f(x_i),f(x_j))$$ 
for arbitrary $i\ne j$ and $z\in \partial D$.  
It follows that
$$  \diam(D')\le \frac 2\delta \diam(\partial D') \le \frac {4\eta(2)}\delta
\sig(y_i,y_j) \forr i\ne j.$$
This shows that $F$ satisfies the conditions \eqref{sep1} and \eqref{sep2} with $X=D$, $Y=D'$ and $\lambda =  \frac {4}\delta(1\vee \eta(2)).$
 Since $\lambda$ only depends on $\delta$ and $\eta$,
  and $F$ is $\tilde \eta$-quasi-M\"obius with  $\tilde \eta$ only depending 
  on $k$ and $\eta$, 
it follows from Proposition~\ref{prop:inter}~(iii)  that  $F$ is $\eta'$-quasisymmetric with $\eta'$ only depending $\delta$, $k$, and $\eta$.
\qed \medskip 

\begin{remark}\label{rem:ext}\textnormal{If  $D$ and $D'$ are  closed Jordan  regions in $\Sph$,  
and $f\:\partial D\ra \partial D'$  is  a homeomorphism,  
then $f$ can be extended to homeomorphism 
 $F\: D \ra D'$.}
 
 \textnormal{Indeed, by the Sch\"onflies theorem this statement can be reduced   to the special  case $D=D'=\overline \D$. Then $F$ is obtained from $f\:\partial \D \ra \partial \D$ by  ``radial" extension, i.e.,
 $$ F(re^{\iu t})=r f(e^{\iu t}) \forr r \in [0,1],\  t\in [0,2\pi]. $$}
\end{remark}

\medskip
\begin{lemma}\label{lem:misc}
Suppose that $\{D_i:i\in I\}$ is a family of pairwise disjoint closed Jordan  regions in $\OC$, where   $I=\{1, \dots, n\}$  with $n\in \N$,   or $I=\N$. If $I=\N$ assume in addition that $\diam (D_i)\to 0$ as $i\to \infty$.  Let $T=\OC\setminus \bigcup_{i\in I}\inte(D_i)$. 

\begin{itemize} 
\smallskip

\item[(i)] Suppose that  we have a set $T'$ with  $T\sub T'\sub\OC$,  and a  map $F\: T'\ra \OC$   such that the restrictions $F|T$ and
$F|T'\cap D_i$, $i\in I$, are continuous. If $I=\N$ assume in addition that 
$\diam (F(T'\cap D_i))\to 0$ as $i\to \infty$. 
Then $F$ is continuous.   

\smallskip
\item[(ii)] The sets   $T$ and $T\setminus \partial D_i$, $i\in I$, are path-connected. 

\smallskip
\item[(iii)] If $f\: T\ra \OC$ is   an embedding,    then the image of $T$ under $f$ can be written as $f(T)=\OC\setminus \bigcup_{i\in I}\inte(D'_i)$, where 
$\{D'_i:i\in I\}$ is a family of pairwise disjoint closed Jordan  regions in $\OC$ with $f(\partial D_i)=\partial D_i'$ for $i\in I$.
Moreover, if $I=\N$, then  we have $\diam (D'_i)\to 0$ as $i\to \infty$.

\smallskip
\item[(iv)] The set $\widetilde T=\OC\setminus \bigcup_{i\in I}D_i=T\setminus \bigcup_{i\in I}\partial D_i $ is non-empty and contains uncountably many elements.

\end{itemize}
\end{lemma}

\no{\em Proof.} (i) We claim that  $F$ is continuous at each point $x\in T'$. This is clear if $x\in \inte(D_i)\cap T'$ for some $i\in I$. Otherwise, $x\in T$.  Let $\eps>0$ be arbitrary. 
Since the Jordan  curves $\partial D_i\sub T$ are pairwise disjoint, the point $x$ can lie on at most one of them. 

Assume that $x \in \partial D_{i_0}$, where $i_0\in I$.  Then 
$F|T\cup ( D_{i_0}\cap T')$ is continuous and so we can choose 
$\delta>0$  so that 
$\sigma(F(y), F(x))< 
\eps/2$ for all $y\in B(x,\delta)$ that lie in  $T\cup (D_{i_0}\cap T')$. 

We have   $x\notin D_i$ for $i\ne i_0$.  Since  there are only finitely many $i\in I$ with
 $\diam (F(T'\cap D_i))\ge \eps/2$ by our hypothesis,  we can assume that  $\delta>0$ is  so small  
   that $\diam(F(T'\cap D_i))<\eps/2$ whenever $i\in I\setminus \{i_0\}$ and  $D_i\cap B(x,\delta)\ne \emptyset$.

     If $D_i\cap B(x,\delta)\ne \emptyset$ for $i\ne i_0$, then also $ \partial D_i\cap B(x,\delta)\ne \emptyset$, and so  there exists a point $y\in  \partial D_i\cap B(x,\delta)\sub T$. It follows that   
$ \sigma (F(x), F(y))<\eps/2$ and $F(D_i\cap T')\sub B(F(y), \eps/2)$  by choice of $\delta$. This  implies $F(D_i\cap T')\sub  B(F(x), \eps)$. 
We conclude that $F(B(x,\delta)\cap T')\sub B(F(x), \eps)$, and the continuity 
of $F$ at $x$ follows.

 A similar argument shows that $F$ is continuous at $x$ if $x\in T\setminus \bigcup_{i\in I}\partial D_i$.

(ii) Pick a point $p_i\in \inte(D_i)$ for each $i\in I$. Let $P=\{p_i:i\in I\}$ and $T'=\OC\setminus P\supset T$. 

For each $i\in I$ there is a retraction of 
$D_i\setminus\{p_i\}$ onto $\partial D_i$, i.e., a continuous map $D_i\setminus\{p_i\}\ra\partial D_i$
that is the identity on $\partial D_i$. These maps and the identity on $T$ paste together to a map 
$R\:T'\ra T$.  
By (i) the map $R$ is continuous, and so it is a continuous retraction of $T'$ onto $T$. 

Since $P$ is countable,
the set  $T'=\OC\setminus P$ is path-connected. Indeed, to find a 
path between any two points $x,y\in T'$ pick an uncountable family of arcs in $\OC$ with endpoints $x$ and $y$ that have no common interior points. One of these arc will lie in 
$T'$. 

Since $T'$ is path-connected and $R$ is a retraction, the image $T=R(T')$ is also path-connected. 

 For each $i\in I$ the set $T'\setminus D_i$ is path-connected
 as it is homeomorphic to the open unit disk $\D$ with at most countably many points  removed.  Hence $T\setminus \partial D_i=R(T'\setminus D_i)$ is path-connected.

(iii) By (ii)   the set $T\setminus \partial D_{i}$ is connected  for each 
$i\in I$. Hence 
$f(T\setminus \partial D_{i})$ is also connected.
Since this set is non-empty and does not meet the Jordan  curve $f(\partial D_i)$, it must be contained in exactly one of the two  components of $\OC\setminus 
f(\partial D_i)$. Let $D'_i$ be the closure of the other 
complementary component of $\OC\setminus 
f(\partial D_i)$. Then $D'_i$ is a closed Jordan  region with $\partial D'_i=f(\partial D_i)$ for each $i\in I$, and we have 
\begin{equation}\label{incl}
f(T)\sub \OC\setminus \bigcup_{i\in I} \inte(D'_i). 
\end{equation} 
Since the family $\{\partial D'_i=f(\partial D_i):i\in I\}$ of Jordan  curves   consists of  pairwise disjoint sets, the last inclusion implies that the  family  $\{D'_i: i\in I\}$ also consists of pairwise disjoint sets.

In order to prove that we have  equality in \eqref{incl},  we first show  that if $I=\N$, then $\diam(D_i')\to 0$ as $i\to \infty$. 
By our hypotheses we have $\diam(\partial D_i)\to 0$.
Since $T$ is compact, the map $f$ is uniformly continuous on $T$,  and so $\diam(\partial D'_i)=\diam (f(\partial D_i))\to 0$ as $i\to \infty$.  
The argument in the proof of Proposition~\ref{prop:qround} shows that for each $i\in I$ we have 
\begin{equation}\label{dimbd} \diam(D'_i)\le 2\diam (\partial D'_i),
\end{equation}
or else $D'_i$ contains a disk of radius $1$.
Since the Jordan  regions $D'_i$ are pairwise disjoint,   
it follows that inequality \eqref{dimbd} holds for all $i\in I$ with at most  finitely many  exceptions. This implies    the desired statement 
$\diam(D_i')\to 0$ as $i\to \infty$.

By Remark~\ref{rem:ext}  the map 
$f|\partial D_i\: \partial D_i\ra \partial D'_i$ 
extends to a homeomorphism of $D_i$ onto $D_i'$
for each $i\in I$. These extensions  and the map $f$ paste together to an injective map $F\: \OC \ra \OC$ so that   $F|T=f$
and $F|D_i$ is continuous for each $i\in I$.  Moreover, 
$\diam (F(D_i))=\diam (D'_i)\to 0$ as $i\to \infty$ if $I=\N$. Hence by (i) the map $F$ is 
continuous on $\OC$. 

Since $F\: \Sph\ra \Sph$ is injective and continuous, this  map is a homeomorphism onto its image. By ``invariance of domain" this image is open. Since it is also compact, and hence closed, it follows that $F(\OC)=\OC$, and so $F$ is a homeomorphism of $\OC$ onto itself. 
Hence $$f(T)=F(T)=F\biggl(\OC\setminus \bigcup_{i\in I}\inte(D_i)\biggr)=\OC\setminus  \bigcup_{i\in I}F(\inte(D_i))= \OC\setminus \bigcup_{i\in I} \inte(D_i')$$
as desired.

(iv) The statement is  clear if $I$ is a finite set, because then
 $\widetilde T$ has  interior points (for example, points in 
 $\oC\setminus D_i$ sufficiently close to $\partial D_i$ are interior points of $\widetilde T$);  if $I$ is an infinite set, then by (ii) we can find a path $\alpha\:[0,1]\ra T$ 
 with endpoints on different Jordan  curves $\partial D_{i_0}$ and $\partial D_{i_1}$, $i_0,i_1\in I$, $i_0\ne i_1$. We claim that $\alpha\cap \widetilde T$ is an uncountable  set. Otherwise, this set consists of  a countable (possibly empty) collection of distinct points $x_\lambda$, $\lambda\in  \Lambda
 \sub \N$. Then $[0,1]$ is the disjoint union of the countably 
 many closed sets $\alpha^{-1}(\partial D_i)$, $i\in I$, and 
 $\alpha^{-1}(\{x_\lambda\})$, $\lambda\in \Lambda$. 
 Hence $[0,1]$ must be contained in one of these sets (one cannot represent  $[0,1]$ as a countable union of pairwise disjoint closed sets in a non-trivial way; see  \cite[p.~219]{StSe}). 
 This implies that $\alpha$ is contained in  one of the Jordan  curves 
 $ \partial D_i$ or is a constant path. Both alternatives are impossible, since 
 $\alpha$ has one endpoint on $\partial D_{i_0}$, and one on the 
 disjoint set $\partial D_{i_1}$. 
  \qed 

\medskip
\no \emph{Proof of Proposition~\ref{prop:extend}.}  By pre- and post-composing $f$ with  suitable M\"obius transformations, we may assume that 
there exists an index $i_0\in I$ such that $0,1,\infty\in \partial 
D_{i_0}$,  and $f(0)=0$, $f(1)=1$,  $f(\infty)=\infty$.  In this reduction we used 
the fact that both the hypotheses of the proposition and the desired conclusion remain essentially unaffected by applying 
 such auxiliary M\"obius transformations; 
indeed, the  image of a family of uniform quasicircles under a M\"obius
transformation 
consists of uniform quasicircles, quantitatively (this was shown  in the proof of Corollary~\ref{cor:qmobinvcond}), and pre- and post-composition with  M\"obius transformations changes  neither the distortion 
function of a quasi-M\"obius map  nor the dilatation of a quasiconformal map on $\Sph$. 

The map $f$ then satisfies conditions \eqref{sep1} and  \eqref{sep2}
in Proposition~\ref{prop:inter} with  $X=T$, $Y=f(T)$, $x_1=y_1=0$,
$x_2=y_2=1$, $x_3=y_3=\infty$, and 
$\lambda=\sqrt2$. Hence $f$ is $\tilde\eta$-quasisymmetric with 
$\tilde\eta=\tilde\eta_{\eta,k}$. 

The set $I$ is finite,  or countably infinite in which case we may assume that $I=\N$. We show that if $I=\N$, then $\diam(D_i)\to 0$ as $i\to \infty$. Indeed, by our hypotheses and Proposition~\ref{prop:qround}, there exists $\lambda\ge 1$, $r_i>0$,  and points $x_i\in \Sph$  such that 
$$\overline  B(x_i, r_i/\lambda)\sub D_i \sub \overline B(x_i, r_i)\foral i\in I. $$ 
Since the regions $D_i$, $ i\in I$, are pairwise disjoint , the first inclusion shows that $r_i\to 0$ as $i\to \infty$. Hence 
$\diam(D_i)\to 0$ as $i\to \infty$  by the second inclusion, as desired.

By Lemma~\ref{lem:misc} (iii)
 there exist pairwise disjoint closed Jordan  regions $D_i'$ for $i\in I$  such that $\partial D_i'=f(\partial D_i)$ and 
$$T'=f(T)=\Sph\setminus \bigcup_{i\in I} \inte(D'_i). $$ 
Moreover, if $I=\N$ we have $\diam (D'_i)\to 0$ as $i\to \infty$. 
  
 By the normalization imposed in the beginning of the proof,   the complement of each open  Jordan  region $\inte(D_i)$ and $\inte(D'_i)$, $i\in I$,  contains  the points $0,1,\infty$. This implies that condition  \eqref{complbig} in 
Proposition~\ref{prop:qcircext} for $D=D_i$ and $D'=D_i'$ is true with 
$\delta=\diam\{0,1,\infty\}=2$. It follows that for each $i\in I$ we can extend the map $f|\partial D_i\:\partial D_i\ra \partial D_i'$ to an 
$\eta'$-quasisymmetric map from $D_i$ onto $D_i'$, where $\eta'=\eta'_{\tilde \eta, k}=\eta'_{\eta, k}$.

These maps paste together to  a bijection $ F\: \Sph \ra \Sph$ 
whose restriction to $T$ agrees with the $\tilde \eta$-quasisymmetric  map  $f\:T\ra T'=f(T)$ and whose restriction 
to each set $D_i$, $i\in I$,  is an $\eta'$-quasisymmetric  map onto $
D_i'$. 

By Lemma~\ref{lem:misc}~(i) the map $F$ is continuous and hence a homeomorphism. 
We claim that  $F$  
is  $H$-quasiconformal with $H=H(\eta,k)$.  
We need to show that  there exists a constant $H=H(\eta,k)\ge 1$ such that for every
$x\in\Sph$, 
\begin{equation}\label{E:Qc}
\limsup_{r\to 0}\frac{L_F(x,r)}{l_F(x,r)}\leq H,  
\end{equation}
where $L_F$ and $l_F$ are defined as in \eqref{Lf} and \eqref{lf}. 
Below  we will write $a\lesssim b$ for two quantities $a$ and $b$,  
if there exists
a constant $C$ such that $a\leq Cb$ that depends only on
the functions $\eta$ and $\eta'$, and hence only on $\eta$ and $k$.  We will write $a\simeq b$
if both $a\lesssim b$ and $b\lesssim a$ hold.

If $x$ is in one of  complementary components $\inte(D_i)$ of $T$, then  (\ref{E:Qc}) with $H=\eta'(1)$
follows   from the definition of $F$. Thus it is enough 
to only consider  the case $x \in T$.

Since $T$ is 
connected (Lemma~\ref{lem:misc}~(ii)), there exists small $r_0>0$ such that the circle 
$$S(x,r):= \{y\in \Sph\co \sig(y,x)=r\}$$
has non-empty intersection with  $T$ for each $0<r\leq r_0$. Suppose that $r$ is 
in this range. 
Since $F|T=f$ is 
$\tilde \eta$-quasisymmetric, it suffices  to show that for each  $y\in S(x,r)$, there exists a 
point $v \in T \cap S(x,r)$   
such that 
\begin{equation}\label{E:Ineq}
\sig(F(v),F(x))\lesssim\sig (F(y),F(x))
\lesssim\sig(F(v),F(x)).  
\end{equation}
For then $L_F(x,r)/l_F(x,r)$ will be bounded by a quantity 
 comparable to $\tilde \eta(1)$. 

This is trivial if $y$ itself is in $T$. Thus we assume that $y$ is not 
in $T$. Then $y$  lies in one of the complementary components $\inte(D_i)$ of $T$. For simplicity we drop the index $i$ and write $D=D_i$.

Since $S(x,r)$ contains $y\in D$ and points in $T$, and hence in the complement of $\inte(D)$, we have $S(x,r)\cap \partial D  \ne \emptyset$.
 For  $v$ we pick  an arbitrary point in  $S (x,r)\cap \partial D$, and 
let $u$ be a point in the intersection of $\partial D$ and a
minimizing spherical geodesic segment  joining $x$ and $y$. Since
 $\sig(y,u)\leq\sig(v,u)$,
$\sig(u,x)\leq \sig(v,x)$,  and $\sig(v,u)\le 2r=2 \sig(v,x)$, and since 
$\{x, v, u\}\sub T$ and  $\{y, v, u\}\sub D$, we have
\begin{align}
\sig (F(y),F(x))&\leq\sig(F(y),F(u)))+\sig(F(u),F(x)) \notag\\ 
&\lesssim\sig(F(v),F(u))+\sig(F(v),F(x))\lesssim\sig(F(v),F(x)).\notag
\end{align}
This shows the right-hand side of~(\ref{E:Ineq}). To prove the inequality on the  left-hand
side, we choose a point  $u'$ as the  a preimage under $F$ of a point in
the intersection of $F(\dee D)$ with a minimizing  spherical geodesic joining  $F(x)$
and $F(y)$.  
Again, we have 
$\{x, v, u'\}\sub T$ and $\{y, v, u'\}\sub D$.
We need to consider two cases:

 \smallskip \no 
\emph{Case 1.} $\sig(u',x)\geq\frac12 r$. In this case we have 
$r=\sig(v,x)\le2\sig(u',x)$, and therefore
$$
\sig(F(v),F(x))\lesssim\sig(F(u'),F(x))\lesssim\sig(F(y),F(x)).
$$

 \smallskip \no 
\emph{Case 2.} $\sig(u',x)\leq\frac12 r$. Then we have 
$\sig(v,u')\le 2r \le 4 \sig(y,u')$. Spherical distances are additive along minimizing spherical geodesic segments. By choice of $u'$ this gives the inequality 
$$ \sig(F(y),F(u'))
+\sig(F(u'),F(x))\le 2\sig(F(y),F(x))  
$$ for chordal distances, 
and so \begin{align}
\sig(F(v),F(x))&\leq\sig(F(v),F(u'))+\sig(F(u'),F(x))\notag\\
&\lesssim\sig(F(y),F(u'))
+\sig(F(u'),F(x))\lesssim\sig(F(y),F(x)).\notag
\end{align}
This completes the proof of~(\ref{E:Ineq}), and thus of~(\ref{E:Qc})
and the proposition. \qed
\medskip 

\section{Classical and transboundary modulus} 
\label{s:mod} 

\no 
 A {\em density} 
is a non-negative Borel function $\rho\: M\ra [0,\infty]$ defined on some Borel set $M\sub \oC$. 
 Let $\Gamma$ be a family of paths in $\Sph$, and $\rho$ a density on $\oC$. Then $\rho$  is called {\em admissible} (for $\Gamma$) if $$ \int_\ga\rho\, ds\ge 1$$ for all  locally rectifiable paths $\gamma$ in $\Gamma$.  Here integration is with respect to spherical arclength. 
 The  {\em modulus} of the  family
$\Gamma$ is defined as  
\begin{equation*}
\Mod(\Gamma) = \inf_{\rho} \int \rho^2\,d\Sig,
\end{equation*}
where the infimum is taken over all densities 
 $\rho$
that are admissible  and integration is with respect to 
spherical measure $\Sigma$ on $\Sph$. We    refer to the densities $\rho$ over which the infimum is taken here  also as the densities that are {\em admissible 
for $\Mod(\Gamma)$}.
 Note that if $\Gamma$ is a family of paths in a region  $\Om\sub \Sph$, then we can restrict ourselves to considering densities $\rho$ that vanish on $\Sph\setminus \Om$.   A density for which the infimum is attained is called {\em extremal} for $\Mod(\Gamma)$.

\begin{remark} \label{rem:chb}\textnormal{
The modulus of a path family $\Gamma$ in a region $\Om$
does not change if the  the spherical base metric that was used to compute $\int_\ga \rho\,ds $ and $\int\rho^2\, d\Sig$ is changed to a conformally equivalent metric.}

\textnormal{
More precisely, suppose that $\Om$ is a  region in $\Sph$, $\Gamma$ is a path family in $\Om$, and $\lambda\:\Om \ra (0,\infty)$ is a continuous and positive ``conformal factor". 
Consider the conformal metric on $\Om$ with length element
$ds_{\lambda}:=\lambda ds$
and associated area element 
$dA_{\lambda}:=\lambda^2\, d\Sig $.}

\textnormal{
Call a Borel function $\tilde \rho\:\Om\ra [0,\infty]$  admissible  if 
$$\int_{\gamma} \tilde \rho\, ds_{\lambda} \ge 1$$
for all locally rectifiable paths $\ga$ in $\Gamma$, and define 
$$\Mod_{\lambda}(\Gamma)=\inf_{\tilde \rho} \int_{\Om}\tilde \rho\,dA_{\lambda}, $$  where the infimum is taken over all admissible $\tilde \rho$. 
Then $\Mod_{\lambda} (\Gamma)=\Mod(\Gamma)$. 
This follows from the fact that the class of locally rectifiable paths  is the same for the spherical metric and the conformal metric with length element $ds_{\lambda}$ and that 
$\rho \leftrightarrow \tilde \rho=\rho/ \lambda$ gives a bijection between admissible densities for $\Mod(\Gamma)$ and   
$\Mod_{\lambda} (\Gamma)$, respectively, that is mass preserving in the sense that
$$ \int_{\Om}\rho^2\, d\Sig= \int_{\Om}\tilde \rho^2\,dA_{\lambda}. $$}
\end{remark}

\medskip
If $f\: \Om \ra \Om'$ is  a continuous map between sets $\Om$ and $\Om'$  in $\Sph$ and $\Gamma$ is a family of paths in $\Om$, then we denote 
by $f(\Gamma)=\{f\circ \ga: \ga \in \Gamma\}$ the family of image paths. 

Conformal maps do not change the modulus of a path family: if $f\:\Om\ra \Om'$ is  a conformal map between regions $\Om,\Om'\sub\Sph$ and $\Gamma$ is a path family in $\Om$, then $\Mod(\Gamma)=\Mod(f(\Gamma))$. This  is the fundamental property of modulus and easily follows from the previous remark on conformal change of the base metric.  

Quasiconformal maps distort the moduli of  path families in a controlled way (in \cite[Ch.~2]{Va} this is the basis of the definition of a quasiconformal map; it is well-known that it is quantitatively equivalent to our definition \cite[Thm.~34.1 and Rem. 34.2]{Va}).

\begin{proposition}\label{prop:moddist}Let $\Om$ and $\Om'$ be regions  in $\Sph$, 
$\Ga$ be a path family in $\Om$, and $f\: \Om\ra \Om'$ be an $H$-quasiconformal map. Then 
\begin{equation} \label{moddistort}
\frac 1K \Mod(\Ga) \le \Mod(f(\Ga)) \le K \Mod(\Ga),
\end{equation}
where $K=K(H)\ge 1$. 
\end{proposition}

%If $E$ and $F$ are subsets of $\Sph$ with positive diameter,
%we denote by 
%\begin{equation*} \label{reldist}
%   \Delta(E,F) = \frac {\dist(E,F)} {\diam(E)\wedge  \diam(F)   }
%\end{equation*}
%the relative distance of $E$ and $F$. 

Let $\Om$ be a region in $\Sph$ and 
$\mathcal{K}=\{K_i: i\in I\}$ be a finite  collection of pairwise disjoint 
compact subsets of  $\Om$. Here $I$ is a finite index set. 
 Define $K:=\bigcup_{i\in I}K_i$.

Let   $\ga\: J \ra \oC$  be a path   defined on an interval 
  $J\sub \R$.  Since  $\Om\setminus K$ is open, the set   
 $\ga^{-1}(\Om\setminus K)$ is  relatively open  in $J$ and so can it be written as
 $$ \ga^{-1}(\Om\setminus K) =\bigcup_{l\in \Lambda} J_l,$$
 where $\Lambda$ is  a countable (possibly empty) index set, and the sets  $J_l$, $l\in \Lambda$, are pairwise disjoint intervals in $J$. We call  $\ga$ {\em locally rectifiable in $\Om\setminus K$}, if  the path $\ga|J_l$ is  locally rectifiable for each $l\in \Lambda$. 
 
 In this case the path integral $\int_{\ga|J_l}\rho\, ds\in [0,\infty]$ is defined whenever   $\rho\: \Om \setminus K\ra [0,\infty]$ is  a Borel function.   We set   
 $$  \int_{\ga \cap (\Om\setminus K)} \rho\, ds:=
 \sum_{l\in \Lambda }\int_{\ga|J_l}\rho\, ds.$$  

A  {\em transboundary mass distribution} on $\Om$  consists of  a density on $\Om\setminus K$, i.e., a Borel function $\rho\:\Om\setminus K\ra [0,\infty]$, 
and non-negative weights $\rho_i\ge 0$ for $i\in I$ (so  each of the sets $K_i$ has a corresponding weight $\rho_i$). We call 
$$ \int_{\Om\setminus K}\rho^2\, d\Sigma+\sum_{i\in I}\rho_i^2 \in [0,\infty]$$
its {\em total mass}. 
The transboundary  mass distribution is called {\em admissible} with respect to a path family $\Gamma$ in $\oC$
if 
\begin{equation}\label{eq:defadmis}
 \int_{\ga \cap (\Om\setminus K)} \rho\, ds+
\sum_{\ga\cap K_i\ne \emptyset}\rho_i\ge 1,  \end{equation}
 whenever $\ga$ is a path in $\Gamma$ that is locally rectifiable in $\Omega\setminus K$.     Note that we do not require that $\Gamma$ consists of paths  in $\Om$.  
 
The {\em transboundary modulus} of $\Gamma$ 
with respect to $\Omega$ and $\mathcal{K}$ is defined as  
\begin{equation}\label{eq:deftrans}
\Md_{\Om,\mathcal{K}}(\Gamma)=\inf_{\rho} \biggl\{ \int_{\Om\setminus K}\rho^2\, d\Sig +\sum_{i\in I}
\rho_i^2\biggr\},
\end{equation}
where the infimum is taken over all transboundary mass distributions on $\Om$ that are admissible for  
$\Gamma$.  A transboundary mass distribution realizing the infimum is called {\em extremal} for $\Md_{\Om,\mathcal{K}}(\Gamma)$.

The concept of transboundary modulus is due to Schramm. He introduced it in  equivalent  equivalent form as  {\em transboundary extremal length} (the reciprocal of transboundary modulus) in \cite{oS95}. 

 As the next lemma shows, the transboundary modulus of a path family  in $\Om$ is invariant under homeomorphisms that are conformal on $\Omega\setminus K$ 
(see \cite[Lem.~1.1]{oS95} for a similar statement). 

\begin{lemma} [Invariance of transboundary modulus]\label{lem:invtrans}
Let $\Om$ and $\Om'$ be  regions in $\Sph$,  and  
$\mathcal{K}=\{K_i: i\in I\}$  be a finite  collection of pairwise disjoint 
compact sets  in $\Om$. Suppose that $\Ga$ is  a path family in 
$\Om$ and that $f\:\Om \ra \Om'$ is a homeomorphism that is conformal on $\Om\setminus K$. Set $\mathcal{K}'=\{f(K_i): i\in I\}$ and $\Gamma'=f(\Gamma)$. 
Then  
$$ \Md_{\Om,\mathcal{K}}(\Gamma)=\Md_{\Om', \mathcal{K}'}(\Ga').$$
\end{lemma}

\no
{\em Proof.}  Note that the sets $f(K_i)$, $i\in I$, are pairwise disjoint 
compact subsets of $\Om'$; so  $\Md_{\Om', \mathcal{K}'}(\Ga')$ is defined.

We denote by  $Df(p)\: T_p\Sph\ra T_{f(p)}\Sph$ the differential of $f$ at $p\in \Om\setminus K$. This is a linear map between the tangent spaces $T_p\Sph$ and $T_{f(p)}\Sph$ of $\Sph$ (considered as a smooth manifold) at the points $p$ and $f(p)$, respectively. Using the   Riemannian structure on $\Sph$ induced by the spherical metric on $\Sph$, we can assign  an operator norm $\Vert Df(p)\Vert$ to this map. If $p,f(p)\in \C$, then
$$ \Vert Df(p)\Vert=\frac{(1+|p|^2)|f'(p)|}{1+|f(p)|^2}. $$
Let  $\Vert Df\Vert$  be 
the map $p\mapsto  \Vert Df(p)\Vert$. 

 A transboundary 
mass distribution   on $\Om'$ consisting of the 
Borel function $\rho\: \Om'\setminus K'\ra [0,\infty]$   and the discrete weights $\rho_i\ge 0$, $i\in I$, is admissible for $\Gamma'$ if and only if 
the transboundary mass distribution on $\Om$ consisting of 
the density $(\rho\circ f) \Vert Df\Vert$ on $\Om\setminus K$ 
and the discrete weights $\rho_i$, $i\in I$, is admissible for $\Ga$. Indeed,   in the admissibility conditions the total  contributions 
from the discrete weights are obviously equal; this is also true 
for the contributions from the densities, since we have   
the equation 
$$ \int_{\ga\cap( \Om\setminus K)} (\rho\circ f) \Vert Df\Vert\, ds=
\int_{(f\circ \ga)\cap (\Om'\setminus K')} \rho \,ds$$
valid for all paths in $\Gamma$ that are locally rectifiable in $\Om\setminus K$. Note that a path $\ga\in \Gamma$ is locally rectifiable in $\Om\setminus K$  if and only if the path $f\circ \ga$ is locally rectifiable in $\Om'\setminus K'$. 

Moreover, by the  conformality of $f$ on $\Om\setminus K$ we have 
$$ \int_{\Om\setminus K} (\rho\circ f)^2\Vert Df\Vert^2\, d\Sig =
\int_{\Om'\setminus K'} \rho^2\, d\Sig.$$
This shows that every transboundary mass distribution that is  admissible for $M_{\Om',\mathcal{K}'}(\Gamma')$ gives rise to a mass distribution that is admissible for $M_{\Om,\mathcal{K}}(\Gamma)$ of the same total mass. This implies that  $M_{\Om,\mathcal{K}}(\Gamma)\le M_{\Om',\mathcal{K'}}(\Gamma')$. The reverse inequality follows by applying the same argument to $f^{-1}$.  \qed 
\medskip

Let $\Om\sub \Sph$ be a region  
and $E,F\sub \overline \Om$. We say that $\ga$ is a {\em (closed)  path in $\Om$ connecting $E$ and $F$} if the path  is a continuous map  $\ga\: 
 [a,b]\ra \Sph$ defined on a closed interval 
 $[a,b] \sub \R$ such that $\ga(a)\in E$, $\ga(b)\in F$, and 
 $\ga((a,b))\sub \Om$. So   $\ga$ lies in $\Om$ with the possible exception of its endpoints. We denote  by $\Ga(E,F; \Om)$ 
 the family of  all closed paths $\ga$ in $\Om$ that connect 
 $E$ and $F$.   In Section~\ref{s:carpet} it will be more convenient 
 to consider  {\em open  paths in $\Om$ connecting $E$ and $F$}. By definition these are   paths $\alpha$ for which there  exists a 
path  $\ga\: 
 [a,b]\ra \Sph$ in $\Ga(E,F; \Om)$ such that $\alpha =\gamma|(a,b)$. 
The family of these paths  $\alpha$ is denoted by $\Ga_o(E,F; \Om)$
(so the subscript ``$o$" indicates ``open" paths).

 Let $\mathcal{K}=\{K_i: i\in I\}$ be a finite   collection of pairwise disjoint 
compact subsets of  $\Om$. Set  $K=\bigcup_{i\in I}K_i$.
Note that 
%have\footnote{This identity is not as obvious as it seems at first sight, since 
%in   $\Gamma_o(E,F; \Om)$ more curves have to be checked for the admissibility of a given mass distribution. For example, if $\ga\:[a,b]\ra 
%\Sph$ is a path in $\Ga(E,F; \Om)$ with $\ga\sub \Omega\setminus K$,
%then this path is locally  rectifiable in  $\Om\setminus K$ precisely 
%if it is rectifiable. In contrast,  for  the corresponding path 
%$\alpha =\gamma|(a,b)$ in  $\Ga_o(E,F; \Om)$  locally  rectifiability  in  $\Om\setminus K$  only means that the subpath $\gamma|[a',b']=\alpha|[a',b']$ is rectifiable for each interval $[a',b']\sub (a,b)$. As we will not use this equation, we will leave the iustification to the reader (it is clear that the lefts hans side is no larger than the right hand side; to prove the  
%the other  inequality  add $\eps$  } 
 $ M_{\Om, \mathcal{K}}(\Gamma(E,F; \Om))$ can be different from  
$M_{\Om, \mathcal{K}}(\Gamma_o(E,F; \Om))$.   One can easily obtain 
an example by assuming that  $E$ or $F$ is contained in one of the sets in
$\mathcal{K}$. Setting the discrete weight equal to $1$ on this set
and all the other discrete weights and the density equal to $0$ produces an admissible mass distribution 
for $ M_{\Om, \mathcal{K}}(\Gamma(E,F; \Om))$, but not necessarily for $M_{\Om, \mathcal{K}}(\Gamma_o(E,F; \Om))$. Hence 
$ M_{\Om, \mathcal{K}}(\Gamma(E,F; \Om))\le 1$, but it is not hard to find a situation where $ M_{\Om, \mathcal{K}}(\Gamma_o(E,F; \Om))>1$.

If $\Om'\sub \oC$ is another region, $f\:\overline \Om\ra \overline \Om'$ is a homeomorphism with $f(\Om)=\Om'$, and if we define $E'=f(E)$ and $F'=f(F')$, then 
 $$ f(\Gamma(E,F; \Om))=\Gamma(E',F'; \Om').$$ 
 Moreover, if $\mathcal{K}':=\{f(K_i):i\in I\}$ and $f|\Om\setminus K$ 
 is conformal, then the same argument as in the proof of Lemma~\ref{lem:invtrans} shows that 
 $$M_{\Om, \mathcal{K}}(\Gamma(E,F; \Om))=
 M_{\Om', \mathcal{K'}}(\Gamma(E',F'; \Om')). $$

 \begin{remark} \label{rem:chbase} \textnormal{Similarly as for classical modulus (see Remark~\ref{rem:chb}),  transboundary modulus does not change if we replace the integrals  $\int_{\ga\cap (\Om\setminus K)}\rho\,ds$
and $\int_{\Om\setminus K} \rho^2\, d\Sig$ in its definition by similar integrals with respect to a different base metric that is conformally equivalent to the spherical metric.
This will be important in Section~\ref{s:square} where it is convenient to use the flat metric with length element $|dz|/|z|$  as a base metric on $\C^*$.} 
\end{remark}

\section{Loewner regions}
\label{s:loewner} 

\no
 Let $\Om\sub \Sph$ be a region  
in $\oC$. 
  If  there exists a non-increasing function 
$\phi\: (0,\infty)\ra (0, \infty) $ such that 
$$  \Mod(\Ga(E,F;\Om) )\ge \phi( \Delta(E,F)), $$
whenever $E$ and $F$ are disjoint continua in $\overline \Om$, then we call 
$\Om$ a {\em Loewner region}  (or a $\phi$-{\em Loewner region} if we want to emphasize $\phi$).   
A region $\Om$ is Loewner if and only if the following statement is true: for each $t>0$ there exists  $m=m(t)>0$ such that if $E$ and $F$ are disjoint continua in $\overline \Om$ with    $\Delta(E,F)\le t$ and $\rho$ a  density on $\oC$ with $\int\rho^2\, d\Sigma<m$, then there exists a rectifiable path in $\Om$ connecting  $E$ and $F$ such that 
$\int_\ga \rho\, ds<1$.   Indeed, if $\Om$ is $\phi$-Loewner, then we can take $m=m(t):=\phi(t)$ for $t>0$ in this  condition. Conversely, if 
the condition   is satisfied, then $\Om$ is $\phi$-Loewner with 
 $\phi(s):=\sup\{ m(t):t\ge s\}$ for $s>0$.  

Loewner regions  are examples for  {\em Loewner spaces}  as introduced by Heinonen and Koskela \cite{HK98}. 

Let $\Om$ be a proper subregion of $\Sph$. Then $\Om$ is called 
$A$-{\em uniform}, where  $A\ge 1$, if the following condition holds:
for any points $x,y$ in $\Om$ there exists a parametrized arc $\ga\:[0,1]\ra \Om$
such that $\ga(0)=x$, $\ga(1)=y$, 
$$ \length (\ga) \le A \sigma(x,y), $$ 
and 
$$ \dist(\ga(t), \partial \Om)\ge \frac 1{A} (\length(\ga|[0,t]) \wedge
\length(\ga|[t,1])) $$
for all $t\in [0,1]$. 

The unit disk $\D$ is an example of a uniform region.
If $\delta>0$ and  $\Om=N_\delta(\D)\setminus \overline \D$, 
then $\Om$  is an annulus for $\delta\in (0,\sqrt 2)$ and so  this region is $A$-uniform with $A=A(\delta)$ (recall from section Section~\ref{s:nt} that $N_\delta(A)$ denotes the open $\delta$-neighborhood of a set $A$).  We will use these     facts below. They are essentially well-known and so we omit the easy (and tedious) proof.

Uniform regions are Loewner regions, quantitatively.  
\begin{proposition} \label{prop:unifmod}
Every $A$-uniform region $\Om\sub \OC$ is $\phi$-Loewner 
with $\phi=\phi_A$ only depending on $A$.
\end{proposition} 

Again this  statement   is essentially well-known and goes back to \cite{GM}. See \cite[Ch.~6]{BHK}, and in particular  \cite[Rem.~6.6] {BHK},   for more background. The  statement can be derived from the fact that $\oC$ is Loewner and from  \cite[Rem.~6.38 and  Thm.~6.47] {BHK}.

Images of Loewner regions under  quasi-M\"obius maps on $\OC$ are Loewner, quantitatively. 

\begin{proposition} \label{prop:Loewqcinv}
Let $\Om\sub  \OC$ be a $\phi$-Loewner region and 
$f\:\Sph \ra\Sph$ be  an $\eta$-quasi M\"obius map. Then 
$\Om'=f(\Om)$ is  a $\psi$-Loewner region with $\psi$ only depending on $\eta $ and $\phi$.  
\end{proposition} 

\begin{proof} Let $E'$ and $F'$ be disjoint continua in 
$ \overline \Om'$. Then   $E=f^{-1}(E')$ and $F=f^{-1}(F')$ are disjoint continua in $\overline \Om$.  Since  $f$ is $\eta$-quasi-M\"obius, it follows from  Lemma~\ref{lem:modcr} and Lemma~\ref{lem:crrelsep} that there exists a homeomorphism $\theta\: [0,\infty)\ra [0,\infty)$ 
that can be chosen only depending on $\eta$  such that 
$$\Delta (E,F)\ge \theta(\Delta(E',F')). $$ 
Moreover, we have 
$f(\Gamma(E,F; \Om))=\Gamma(E',F'; \Om')$, which by  Proposition~\ref{prop:inter} (i) and Proposition~\ref{prop:moddist}  
implies that  
$$\Mod (\Gamma(E',F'; \Om'))\ge \frac 1K  \Mod( \Gamma(E,F; \Om)),$$ where 
$K=K(\eta)\ge 1$. Since $\Om$ is $\phi$-Loewner we conclude that 
$$ \Mod( \Gamma(E',F'; \Om'))\ge \frac 1K\phi(\Delta(E,F))\ge 
\psi(\Delta(E',F')),$$ where 
$\psi(t)=\frac 1K \phi(\theta(t))>$ for $t>0$. Since $\psi$ can be chosen 
only depending on $\eta$ and $\phi$, the statement follows. 
\end{proof}

Open Jordan  regions in $\Sph$ bounded by 
quasicircles  are Loewner regions, quantitatively.

\begin{proposition}\label{prop:qdisksLoewner}
Let $\Om\sub\Sph$ be an open Jordan  region whose boundary $\partial \Om$ is a $k$-quasicircle. Then $\Om$ is $\phi$-Loewner with $\phi$ only depending on $k$.
\end{proposition}

\no{\em Proof.} By the remark following Proposition~\ref{prop:planeqcirc} the region 
$\Om$ is the image of the unit $\D$ under an $\eta$-quasi M\"obius  map,  
where $\eta=\eta_k$. Since the unit disk $\D$ is  a uniform region   and hence Loewner by 
Proposition~\ref{prop:unifmod}, it follows from Proposition~\ref{prop:Loewqcinv} that $\Om$ is $\phi$-Loewner, where $\phi=\phi_k$. 
\qed 

\medskip
The goal of this section is to prove a similar statement for regions with finitely many complementary components (see Proposition~\ref{thm:lowmod}). 
We first prove the following lemma for preparation.  

\begin{lemma} [Collar Lemma]\label{lem:collar} Let $n\ge 2$, and 
let  $\Om$ be  a region in $\Sph$ such that 
$$\Om= \Sph\setminus \bigcup_{i=1}^n  D_i,$$ where 
the sets $D_i$ are pairwise disjoint closed Jordan  regions. 
Suppose that  the boundaries $\partial D_i$ are  $k$-quasicircles and  the regions $D_i$ are 
$s$-relatively separated for $i=1, \dots,n$, and  that $d=\diam(D_n)\le \diam (D_i)$ for $i=1,\dots,n$.

Then there exists an open Jordan  region $V\supseteq D_n$ in $\oC$ with the following properties: 

\begin{itemize} 

\smallskip
\item[(i)] $U:=V\setminus D_n \sub \Omega$,   

\smallskip
\item[(ii)] $N_{cd}(D_n)\sub V $ where 
$c=c(s,k)>0$ is a constant only depending on $s$ and $k$,

\smallskip
\item[(iii)] $U$ is a $\phi$-Loewner region 
with $\phi=\phi_{s,k}$  only depending on $s$ and $k$.

\end{itemize} 
\end{lemma}

This lemma says that under the given hypothesis one can put a 
``Loewner collar" 
$U$ around the smallest complementary component $D_n$ of $\Omega$ 
that lies in $\Omega$, has  a definite thickness  proportional  to the diameter $d$ of 
$D_n$ with a proportionality constant depending on $s$ and $k$, and  is $\phi$-Loewner with $\phi$ controlled by $s$ and $k$. 

\medskip
\no{\em Proof.} 
Let $D=D_n$. Since $\partial  D$ is  a $k$-quasicircle,
there exists an $\eta$-quasi-M\"obius map $f\:\OC\ra \OC$ with $f(\overline \D)=  D$, where $\eta$ only depends on $k$
(see the remark after Proposition~\ref{prop:planeqcirc}).  We denote by  $u'=f(u)$ the image  of an arbitrary point  $u\in \OC$.

Since $n\ge 2$ and $ D=D_n$ has the smallest diameter of all the sets $D_i$, $i=1, \dots, n$, 
we have $\diam(\OC\setminus  D)\ge \diam(  D)$.
Hence Lemma~\ref{lem:Jordiam} implies that 
$$d= \diam(  D)= \diam(D)\wedge \diam(\OC\setminus  D)\le \diam(\partial  D). $$

We can pick points $x_1,x_2, x_3\in \partial \D$
such that for their image points we have 
\begin{equation}
\sig(x_i',x_j')\ge \diam(\partial  D)/2\ge d/2 \forr i\ne j.
\end{equation}
By pre-composing $f$ with a M\"obius transformation if necessary, we may assume  the points $x_1,x_2,x_3$ are the third roots of unity.
Then 
 $$ \sigma(x_i,x_j)=\sqrt 3  \forr i\ne j.$$
 Since the sets  $D_i$, $i=1, \dots, n$,  are  $s$-relatively separated, and 
 $ D=D_n$ has the smallest diameter of the sets, we have 
 $$\dist(D_i, D)\ge s d \forr i=1, \dots, n-1. $$
 In particular, $N_{sd}(D)\setminus D\sub \Omega$. 
 We claim that we can thicken up $ \D$  by a definite amount only depending on $s$ and $k$ to a larger set  that 
is mapped  into $N_{sd}( D)$ by $f$. More precisely, 
we claim that 
\begin{equation}\label{desincl}
f(N_\de(\D)) \sub N_{sd}( D)
\end{equation}
if  $\delta=\delta(s,k)\in (0,1)$ is suitably chosen.

So assume $\delta>0$ is a small constant whose precise value will be chosen later. Let $v=0$, and $u\in N_\de(\D)\setminus \overline \D$ be arbitrary. Let   $z$ be the closest point to $u$ on $\partial \D$.  
Then  $\sigma(u,z)< \de$.  By Lemma~\ref{lem:3norm} there  there exists 
 $w\in \{x_1, x_2,x_3\}$ such that 
 $$ \sigma(u,w)\ge \sqrt 3/2 \ge 1/2 \quad \text{and} \quad 
 \sigma(v',w')\ge d/4. $$
 We also have the relations $\sig(v',z')\le \diam(  D)=d$,
 $\sig(v,z)=\sqrt 2\ge 1$,  $\sig(v,w)\le \diam(\overline \D)= 2,$ and 
 $\sig(z',w')\le \diam( D)=d.$ 
 Since $f$ is $\eta$-quasi-M\"obius, we obtain   
 \begin{eqnarray}
 \frac {\sig(u',z')}{d} &\le & 
 \frac {\sig(u',z')} {\sig(v',z')}\nonumber \\ &\le &
 \eta\biggl( 
  \frac{ \sig(u,z)\sig(v,w)}
  {\sig(v,z)\sig(u,w)}  \biggr)
\frac {\sig(u',w')}{\sig(v',w')} \nonumber \\
&\le & \eta\biggl( \de
  \frac{\sig(v,w)}
  {\sig(u,w)}  \biggr)
\frac {\sig(u',w')}{\sig(v',w')} \\
&\le &\frac 4 d\eta(4\de )(\sig(u',z')+\sig(z',w '))\nonumber
\\
&\le & 4 \eta(4\de ) \bigg(1+\frac{\sig(u',z')}{d}\biggr)  \nonumber
\end{eqnarray}
Since $\eta(t)\to 0$ as $t\to 0$ this implies that 
$4 \eta(4\de)<1$ if $\de>0$ is small. For such $\de $ we have 
$$ \frac{\sig(u',z')}{d}\le \frac {4 \eta(4\de )} {1-4 \eta(4\de )}.$$
Again using $\eta(t)\to 0$ as $t\to 0$, this  shows  that we can choose
$\delta=\delta(s,\eta)=\delta(s,k)>0$ such that the left hand side 
in the last inequality is less that $s$.
For such $\delta$ we have 
$\dist(p,   D )<sd$ whenever $p\in f(N_\de(\D))$. This gives the desired inclusion
\eqref{desincl}.

Now define $V=f(N_\de(\D))$. Then (i) is true, because
$$V\setminus D_n\sub N_{sd}(D)\setminus D\sub \Om. $$

To show an inclusion of type 
(ii), let $v\in \oC\setminus N_\de(\D)$ and  $z\in \overline \D$ be arbitrary.  
Then $\sig(v,z)\ge \de$. 
We can choose $u\in \overline \D$ such that 
$\sig(u', z')\ge d/2$. Similarly as above, we can then choose
$w\in \{x_1, x_2,x_3\}$ such that 
 $$ \sigma(u,w)\ge 1/2 \quad \text{and} \quad 
 \sigma(v',w')\ge d/4. $$
 We also have $\sig(u,z)\le 2$, $\sig(v,w)\le 2$, and $\sig(u', w')\le d$. 
 Then using these estimates, we get
 \begin{eqnarray}
\frac {d}{2\sig(v',z')} &\le & 
 \frac {\sig(u',z')} {\sig(v',z')} \\ &\le &
 \eta\biggl( 
  \frac{ \sig(u,z)\sig(v,w)}
  {\sig(v,z)\sig(u,w)}  \biggr)
\frac {\sig(u',w')}{\sig(v',w')} \nonumber \\
&\le & \eta( 8/\de)
\frac {\sig(u',w')}{\sig(v',w')} \le 4 \eta( 8/\de).\nonumber
\end{eqnarray}
This shows that that for $c:=c(s,k)=1/(8\eta( 8/\de))>0$ we have
$$ \sig(v',z')\ge cd,$$ 
whenever $v\in \oC\setminus  N_\de(\D)$ and  $z\in \overline  \D$.
It follows that $$\dist (\oC\setminus V,   D) \ge cd.$$ 
This implies $N_{cd}( D)\sub V$ as desired.

It remains to show (iii). The annulus
$N_\de(\D)\setminus\overline \D$ is an $A$-uniform region with $A=A(\delta)=A(s,k)$. Thus this annulus is $\psi$-Loewner with $\psi$ 
only depending on $s$ and $k$ by Proposition \ref{prop:unifmod}. 
Since by Proposition~\ref{prop:Loewqcinv} quasi-M\"obius images of Loewner regions are Loewner
regions,  quantitatively, it follows that 
$U=V\setminus D=f(N_\de(\D)\setminus\overline \D)$
is $\phi$-Loewner with $\phi$ only depending on $\eta$, $s$, $k$. But since $\eta$ was chosen to depend only on $s$ and $k$, this means that $\phi$ can also be chosen to depend only on these parameters. 
\qed \medskip

\begin{proposition} \label{thm:lowmod} 
Let $n\ge 1$, and  $\Om$ be  a region in $\Sph$ such that 
$$\Om= \Sph\setminus \bigcup_{i=1}^n  D_i,$$ where 
the sets $D_i$ are pairwise disjoint closed Jordan  regions. 
Suppose that  the boundaries $\partial D_i$ are $k$-quasicircles and the  regions $D_i$ are
$s$-relatively separated for  $i=1, \dots,n$.    Then $\Om$ is a
 $\phi$-Loewner region with   
 $\phi=\phi_{n,s,k}$ only depending on $n$, $s$, and 
 $k$. \end{proposition} 

In the proof we need a simple fact about the existence of subcontinua. Namely, if $x\in \Sph$, $r>0$,
and $E\sub \oC$ is a continuum with $x\in E$ and $E\setminus B(x,r)\ne \emptyset$, then there exists a 
subcontinuum $E'\sub E$ with $x\in E'$, $E'\sub \overline 
B(x,r)$, and $E'\cap \partial B(x,r)\ne \emptyset$. Note that then  $r\le \diam(E')\le 2r$. So every continuum can be ``cut to size" near each of its points. 

  To see that this  statement is true let $E'$ be the connected component of $E\cap \overline B(x,r)$ containing 
$x$. Then $E'$ is a closed subset of $E$ with $x\in E'\sub\overline B(x,r)$. If we had $E'\cap \partial B(x,r)=\emptyset$, then $E'$ would be relatively open in $E$ and so $E=E'\sub B(x,r)$. This is impossible since $E\setminus B(x,r)\ne \emptyset$. 
So $E'$ is a continuum  with the desired properties. 

\medskip 
\no
{\em Proof of Proposition~\ref{thm:lowmod}.}  The proof is by induction on $n$ with $s$ and $k$ fixed. The induction beginning 
$n=1$ is covered by Proposition 
\ref{prop:qdisksLoewner} (the requirement of $s$-relative separation is vacuous is this case). For the induction step suppose that $n\ge 2$ and that the statement is true for regions with the stated properties and $n-1$ complementary components.

We may assume that 
$D_n$ is the complementary component of $\Omega$ with smallest 
diameter $d:=\diam(D_n)$. Let $E$ and $F$ be arbitrary continua 
in $\overline \Omega$ with relative separation 
$\Delta(E,F)\le t$ where $t>0$. 
We have to show that if $\rho$ is an arbitrary non-negative Borel function  on $\OC$
with sufficiently small mass 
$ \int_\Om\rho^2\,d\Sig< m$, where $m=m(n,s,k,t)>0$, then there exists a  rectifiable path $\gamma$ in  $\Omega$ 
connecting $E$ and $F$ with $\int_\gamma \rho\,ds <1$.

By induction hypothesis we can can find $m_1=m_1(n,s,k,t)>0$ such that if 
$$ \int \rho^2\,d\Sig< m_1,$$ then there exists a rectifiable path $\alpha$ in 
$\widetilde \Omega:=\Omega\cup D_n$ 
that connects $E$ and $F$ and satisfies 
\begin{equation}\label{eq:shortalpha}  
\int_\alpha \rho \,ds<1/2. 
\end{equation}

A suitable constant $m$ will be found in the course of the proof.  We make the preliminary choice $m=m_1$. 
Then there exists a  rectifiable path   $\alpha$ in $\widetilde \Om$ connecting $E$ and $F$ satisfying \eqref{eq:shortalpha}.  If $\alpha$ stays inside $\Omega$ (with the possible exception of its endpoints),  we can take $\gamma=\alpha$.  So we may assume that  $\alpha$ hits $D_n$. Let  $U$  be the  Loewner collar  around $D_n$ found in Lemma~\ref{lem:collar}.  The idea now is to remove $\alpha\cap D_n$
from $\alpha$ and to connect suitable  pieces of $\alpha\setminus D_n$  
by a rectifiable path $\beta$ in $U$  such  that $\int_\beta\rho\,ds<1/2$. A  concatenation of $\beta$ with  pieces of $\alpha$ will  
then give a rectifiable path $\gamma$ in $\Omega$ with $\int_\ga\rho\,ds < 1$ as desired. 

For carrying out the details of this argument, 
we consider several cases.  Let $c=c(s,k)>0$ be the constant from  
Lemma~\ref{lem:collar} with $N_{sd}(D_n)\setminus D_n\sub U$. 

\smallskip
{\em 1.\ Case.} Neither $E$ nor $F$ is contained in $N_{\frac13cd}(D_n)$.

We choose  a closed, possibly degenerate,  subpath $\alpha'$ of $\alpha$ by 
starting at the endpoint $x$ of $\alpha$ in $E$ and  traveling  along 
$\alpha$ until we first hit $\overline {N_{{\frac 16}cd}(D_n)}$ at 
the point $x'\in \overline {N_{{\frac 16}cd}(D_n)}$, say.  Since $\alpha$ meets $D_n$, there exists such a point $x'$. 
Then  $\alpha'\setminus \{x\}\sub\widetilde  \Om\setminus D_n=\Omega$.

 The  set    
 $\alpha'\cup E$ is a continuum that contains the point $x'$, but 
that  is not contained in 
$N_{\frac13cd}(D_n)$ by our  assumption in this case. So if  
  we choose  $r={\frac16cd}$, then $(\alpha'\cup E)\setminus B(x',r)\ne \emptyset$. By  
the statement about  the existence of subcontinua discussed before the proof, we can find a continuum 
$E'\sub \alpha' \cup E$ that is contained in 
$\overline B(x',r)$ 
such that $\diam(E')\ge r=\frac16 cd$.   Then $E'
\sub N_{cd}(D_n)\cap \overline \Omega\sub \overline U$. 

In the same way, we choose  a closed subpath $\alpha''$  of $\alpha$  with endpoints $y\in F$ and $y'\in \overline {N_{{\frac 16}cd}(D_n)}$ such that $\alpha'\setminus \{y\}\sub \Om$.  Again we can  
find a subcontinuum 
$F'$ of $\alpha'' \cup F$ that is contained in $N_{cd}(D_n)\cap \overline \Om \sub \overline U$
such that $\diam(F')\ge\frac16 cd.$  
Then $E',F'\sub \overline U$ and 
$$\dist(E',F')\le (2c+1)d\le (12+6/c) (\diam(E')\wedge \diam(F')).$$
The last inequality implies
that $\Delta(E',F')\le C(s,k)$. Since $U$ is $\phi$-Loewner with 
$\phi=\phi_{s,k}$ there exists  a contant $m_2=m_2(s,k)>0$ with the following property: if we impose the additional condition 
$$ \int  \rho^2\,d\Sig< m_2$$ on $\rho$ (as we may), then there exists  a rectifiable 
path $\beta$ in $U\sub \Om$ with $\int_\beta\rho\,ds<1/2$ that connects $E'$ and $F'$. 
The path $\beta$ will lie in $\Om$  with the possible exception of it endpoints. One endpoint of $\beta$ lies in $E'\sub \alpha'\cup E\sub \Om \cup E$, and one in $F'\sub \alpha''\cup F\sub \Om \cup F$. So by concatenating 
$\beta$ with suitable pieces of $\alpha'$ and $\alpha''$, we obtain a rectifiable path $\gamma$ in
$\Omega$  with $\int_\gamma \rho\,ds<1$ that connects $E$ and $F$. 

\smallskip
 {\em 2.\ Case.}  $t(\diam (E)\wedge \diam(F))\ge \frac 13 cd$. 
 
 We choose subpaths $\alpha'$ and $\alpha''$ of $\alpha$ as in Case 1. Arguing similarly as in this case, we can find continua  $E'\sub \alpha' \cup  E$
 and $F'\sub  \alpha'' \cup  F$ with $E',F'\sub N_{cd}(D_n)\cap \overline \Om\sub \overline U$ such that 
 $\diam(E')\ge \frac 13 (\diam(E)\wedge cd)$
 and $\diam(F')\ge \frac 13 (\diam(F)\wedge cd)$.
 Then again  we have
$$ \dist(E',F')\le (2c+1)d,$$
and also
$$\diam(E')\wedge \diam(F')\ge \tfrac13(\diam(E)\wedge \diam(F)\wedge cd)
\ge \frac{cd}{9(t\vee 1)}.$$
Hence  $$ \Delta(E',F')\le C(s,k,t).$$
In other words, the relative distance of $E'$ and $F'$ is  controlled 
by $s$, $k$, and $t$.   By the Loewner property of $U$ we know that if 
$$ \int \rho^2\,d\Sig< m_3,$$ where $m_3=m_3(s,k,t)>0$, then there exists a  
exists 
a rectifiable path $\beta$ in $U\sub \Om$ with $\int_\beta\rho\,ds<1/2$ that connects $E'$ and $F'$. 
As in Case 1, this leads to a path $\gamma$ as desired. 

\smallskip
{\em 3.\ Case.} $E$ or $F$ lies in $N_{\frac13cd}(D_n)$,
and  we have
$$t(\diam (E)\wedge \diam(F))< \tfrac13cd.$$

We may assume  $E\sub N_{\frac13 cd}(D_n)$. Then 
$E\sub N_{cd}(D_n)\cap \overline \Om\sub \overline U$, and 
 $$\dist(E,F)\le t (\diam (E)\wedge \diam(F))\le \tfrac13cd.$$
Pick points $x\in E$ and $y\in F$ with $\sigma(x,y)=\dist(E,F)$, and let $r=\frac13(\diam(F)\wedge  cd)$. 
There exists a continuum $F'\sub F\cap \overline B(y,r)$ 
with $y\in F$ and $\diam(F')\ge r=\frac13(\diam(F)\wedge  cd)$. Then $F'\sub N_{cd}(D_n)\cap \overline \Om\sub \overline U$ and  
 $\dist (E,F')=\sigma(x,y)=\dist (E,F)$. 

This implies that  
 $$\dist (E,F')=\dist (E,F)\le t(\diam (E)\wedge \diam(F))< \tfrac13cd, $$ 
 and so  
$$\dist (E,F')\le t\diam (E)\wedge t\diam(F)\wedge \tfrac13cd
\le 3(t\vee 1) (\diam (E)\wedge \diam(F')).$$ 
We conclude that  $\Delta(E,F')\le3 (t\vee 1)$.
Since $U$ is Loewner we know that if 
$$ \int \rho^2\,d\Sig< m_4,$$ where $m_4=m_4(s,k,t)>0$, then there exists a  
rectifiable  path $\beta$ in $U\sub \Om$ with $\int_\beta\rho\,ds<1$ that connects  $E$ and $F'\sub F$. 
In this case we can take $\gamma=\beta$.

In conclusion,  if $$ \int \rho^2\,d\Sig< m,$$
where $m=\min\{m_1,m_2,m_3,m_4\}$, then we can 
find a  rectifiable path $\gamma$ in $\Omega$ that connects  $E$ and $F$ and satisfies 
$\int_\gamma\rho\,ds<1$ .
Since $m>0$ only depends on $n,s,k,t$,  the statement follows. \qed

\section{Bounds for transboundary modulus}
\label{s:lowtrans} 

\no
The present chapter is the technical core of the paper. We will prove various bounds for transboundary modulus.  We use   the chordal metric $\sigma$ on $\Sph$ and  the spherical measure $\Sig$.  We will make repeated use of the relation 
$\Sig(B(x,r))=\Sig(\overline B(x,r))\sim r^2$ for $x\in \Sph$ and small enough $r>0$. Actually, we have 
 \begin{equation} \label{measball}
  \Sig(B(x,r))=\Sig(\overline B(x,r))=\pi r^2
 \end{equation} valid for all $x\in \Sph$ and $0<r\le 2=\diam(\oC)$.

A set $M\sub \Sph$ is called $\lambda$-{\em quasi-round}, where $\la \ge 1$,  if there exist $x_0\in \Sph$ and $r\in (0, \diam(\oC)]=(0,2]$ such that $\overline B(x_0, r/\la)\sub M\sub \overline B(x_0,r)$.  Note that in this case $\diam (M)\ge r/\lambda$.  
By Proposition~\ref{prop:qround} every Jordan  region whose boundary is a 
quasicircle is quasi-round, quantitatively.

\begin{proposition} \label{prop:lowtrans} 
Let  $\Om$ be  a $\phi$-Loewner region in $\Sph$, and $
\mathcal{K} =\{ K_i:i\in I\}$  a finite  collection of pairwise disjoint  compact sets  in $\Om$. 
Suppose that  the sets $K_i$ are $\lambda$-quasi-round and are $s$-relatively separated for $i\in I$.  

Then there 
is a non-increasing function $\psi\:(0,\infty)
\ra (0,\infty)$ that can be chosen only depending on $\phi$, $\lambda$, and $s$ with the following property: if $E$ and $F$ are arbitrary disjoint  continua
in $\overline \Om$, then 
$$M_{\Om,\mathcal{K}}(\Gamma(E,F;\Om))\ge \psi(\Delta(E,F)). $$ 
\end{proposition} 

So we get Loewner type bounds for the transboundary modulus in $\Om$ with a Loewner 
function $\psi$ that only depends on the Loewner function $\phi$ for classical modulus  in $\Om$, 
and the parameters $\lambda$ and $s$.

For the proof we need two lemmas.

\begin{lemma} \label{lem:curvemeets}

Under the assumptions of Proposition~\ref{prop:lowtrans}
let $A\sub \oC$ be an arbitrary set, and $t>0$.
Let $N$ be the number of sets $K_i$ such that 
$K_i\cap A\ne \emptyset $ and 
$$\diam (K_i)\ge t \diam(A).$$
Then $N\le C(s,t)$.
\end{lemma}

This means that an arbitrary set $A\sub \oC$ can only meet
a controlled number of those  sets $K_i$ whose diameters are  not much smaller than  the diameter of $A$.

\medskip
\no
{\em Proof.}  For each set $K_i$ that meets $A$ and
satisfies  $\diam(K_i)\ge t \diam (A)$ pick a point 
$x_i\in A \cap K_i$. In this way we obtain a collection
 $\{x_i:i\in I'\}$, $I'\sub I$,  of distinct points in $A$. 
Now if $x_i$ and $x_j$, $i\ne j$, are points in this collection, then we have 
$$\sig(x_i,x_j)\ge \dist(K_i,K_j)\ge s(\diam (K_i)\wedge 
\diam(K_j))\ge st\diam(A).$$
From  \eqref{measball} it easily follows  that the number $N=\#I'$ of these points 
is bounded above by $C/(st)^2$, where $C$ is a universal 
constant.
\qed

\medskip
If $M\sub \Sph$ is an arbitrary set, we denote by $\chi_M$ its characteristic function.

\begin{lemma} \label{lem:boia} 
For each $\lambda\ge 1$, there exists a constant 
$C(\la)\ge 1$ with the following property:
if $\{\overline B(x_i, r_i) :i\in I\}$ is a collection of closed disks in
 $\Sph$ indexed by a countable index set $I$, and if  $a_i\ge 0$ for $i\in I$,
then 
$$  \int\biggl(\sum_{i\in I}a_i\chi_{\overline B(x_i,r_i)}\biggr)^2\,d\Sig
\le C(\lambda) 
\int\biggl(\sum_{i\in I}a_i\chi_{\overline B(x_i,r_i/\lambda)}\biggr)^2\,d\Sig.$$
\end{lemma}

The  lemma is a special case of  a  well-known more general fact. It follows from a duality argument and the $L^2$-boundedness of the Hardy-Littlewood maximal operator (see \cite[p.~58, Lem.~4.2]{Boj} for a very similar statement whose proof can easily be adapted to the present situation).   

\medskip
\no
{\em Proof of Proposition~\ref{prop:lowtrans}.} 
Let $E$ and $F$ be arbitrary continua 
in $\overline \Omega$ with relative separation 
$\Delta(E,F)\le t$ where $t>0 $. 
It is enough  to show that if an arbitrary transboundary  mass distribution on $\Om$
has sufficiently small total mass 
$$ \int_{\Om\setminus K}\rho^2\,d\Sig+\sum_{i\in I}\rho_i^2< m, $$ where $m=m(\phi, s,\lambda,t)>0$, then there exists a rectifiable path $\gamma$ in  $\Omega$ 
connecting $E$ and $F$ with 
$$ \int_{\ga \cap (\Om\setminus K)} \rho\, ds+
\sum_{\ga\cap K_i\ne \emptyset}\rho_i< 1.  $$
Here $K=\bigcup_{i\in I} K_i$.

Since each set $K_i$ is $\lambda$-quasi-round, we can find a disk $
\overline B(x_i,r_i)$ with $x_i\in \oC$ and $r_i\in(0,2]$ such that
 $$\overline B(x_i,r_i/\lambda) \sub K_i \sub \overline B(x_i, r_i).$$ 
If we have an arbitrary transboundary mass distribution on $\Om$, we define a  density $\tilde \rho$ on $\Sph$ as follows:
$$ \tilde \rho=\rho + \sum_{i\in I}\frac{\rho_i}{r_i}\chi_{\overline B(x_i,2r_i)} .$$ 
Here we consider $\rho$ as a function on $\oC$ by setting it equal to $0$ outside its original domain  of definition $\Om\setminus K$.

Then 
\begin{eqnarray*}\int\tilde \rho^2\,d\Sigma&\le &
2\int\rho^2\,d\Sigma +2\int\biggl(\sum_{i\in I}\frac{\rho_i}{r_i}\chi_{\overline B(x_i,2r_i)}\biggr)^2\,d\Sigma\\
&\le &2\int\rho^2\, d\Sigma+C_1(\lambda)\int\biggl(\sum_{i\in I}\frac{\rho_i}{r_i}\chi_{\overline B(x_i,r_i/\lambda)}\biggr)^2\,d\Sigma\\
&\le &2\int\rho^2\,d\Sigma+C_1(\lambda)\sum_{i\in I}\frac{\rho_i^2}{r_i^2}\int \chi_{\overline B(x_i,r_i/\lambda)}\,d\Sigma\\
&\le & C_2(\lambda)\biggl(\int\rho^2\,d\Sigma+\sum_{i\in I}\rho_i^2\biggr).
\end{eqnarray*}
In this estimate we used Lemma~\ref{lem:boia}, the fact that the disks $\overline B(x_i,r_i/\lambda)$, $i\in I$, are pairwise disjoint, and \eqref{measball}. 
 
Since $\Om$ is $\phi$-Loewner, the previous estimate  implies that there exists a constant $m_1=m_1(\phi, \lambda,t)>0$ with the following property:
if  we impose the restriction 
$$ \int_{\Om\setminus K}\rho^2\,d\Sigma+\sum_{i\in I}\rho_i^2< m_1$$ on the transboundary mass distribution (as we may),  
then there exists a rectifiable path $\gamma$ in $\Om$ with
$\int_\gamma\tilde \rho\,ds<1/2$ that connects  $E$ and $F$.

Using this path $\ga$ we define   two disjoint subsets $I_1$ and $I_2$ of $I$.  Let $I_1$ be the set of all $i\in I$ such that 
 $K_i\cap \ga\ne \emptyset$ and $4r_i<\diam(\ga)$,
 and  $I_2$ be the set of all $i\in I$ such that 
 $K_i\cap \ga\ne \emptyset$ and $4r_i\ge\diam(\ga)$.
 Note that $I_1\cup I_2$ is the set of all $i\in I$ with $K_i\cap \ga \ne \emptyset$. 
 
If $i\in I_1$, then $\ga$ meets $K_i\sub \overline B(x_i,r_i)$, but is not contained in $B(x_i,2r_i)$.
Hence $$\int_\ga \chi_{\overline B(x_i,2r_i)}\,ds \ge r_i.$$
This implies 
\begin{eqnarray*}
\int_{\ga \cap (\Om\setminus  K)}\rho\,ds+\sum_{i\in I_1}\rho_i&\le& 
\int_{\ga\cap (\Om\setminus  K)} \rho\,ds+ \sum_{i\in I_1}\frac{\rho_i}{r_i}
\int_\ga\chi_{\overline B(x_i,2r_i)}\,ds\\ &\le& \int_\ga \tilde \rho\,ds<1/2.
\end{eqnarray*}
 
If $i\in I_2$, then $\diam(K_i)\ge r_i/\lambda\ge 
\diam(\ga)/(4\lambda).$ Using Lemma~\ref{lem:curvemeets} for $A=\gamma$, we conclude that  
$N:=\#I_2\le C_3=C_3(s,\lambda)$.

If we impose the additional restriction 
$$ \int_{\Om\setminus K}\rho^2\,d\Sigma+\sum_{i\in I}\rho_i^2< m_2$$
on our transboundary mass distribution, where 
$m_2=m_2(s,\lambda)=\frac1{4C_3^{2}}$, then 
$\rho_i<\frac1{2C_3}$ for all $i\in I$ and so 
$$\sum_{i\in I_2}\rho_i<\frac{N}{2C_3}\le1/2.$$
It follows that 
if 
$$ \int_{\Om\setminus K}\rho^2\,d\Sigma+\sum_{i\in I}\rho_i^2<m=m(\phi,s,\lambda,t):=\min\{m_1, m_2\},$$
then there exists a rectifiable path $\gamma$ in $\Om$ connecting 
$E$ and $F$ with
\begin{eqnarray*}
\int_{\ga \cap (\Om\setminus K)} \rho\,ds+
\sum_{\ga\cap K_i\ne \emptyset}\rho_i&=&
\int_{\ga \cap( \Om\setminus K)} \rho\, ds+
\sum_{i\in I_1} \rho_i+\sum_{i\in I_2}\rho_i\\&<&1/2+1/2=1.
\end{eqnarray*}
Since $m>0$ only depends on $\phi$, $\lambda$, $s$,
and $t$, the proof is complete.
\qed

\medskip 
Before we formulate the next proposition we will discuss some facts  that will be  useful for  estimating  path integrals. Let $\Omega \sub \oC$ be region,   
 $\pi\: \Om \ra \R$ a continuous map, and   $\alpha\: I\ra \Om$  a locally rectifiable path in $\Om$. If  $K\sub \oC$ 
is   compact and $U\sub \oC$ is open,  then $\pi(\alpha\cap U \cap K)$ 
  is a Borel subset of $\R$. This follows from the fact that 
  both the image set of  $\alpha$   and the open set $U$ are countable unions of compact sets. Hence $\pi(\alpha\cap U \cap K)$ is a countable union of compact sets and so indeed  a Borel set. In particular,  if we denote by $m_1$ Lebesgue measure on $\R$, then
   $m_1(\pi(\alpha\cap U \cap K))$  is defined. 
  
If  $\pi$ is a $1$-Lipschitz map, i.e., if 
  $|\pi(u)-\pi(v)|\le \sigma (u,v)$ for all $u,v\in \Om$,  then we have 
  $m_1(\pi(\alpha))\le \length(\alpha)$, and more generally
  \begin{equation} \label{proiectionest}
  m_1(\pi(\alpha\cap U))\le \int_\alpha \chi_U\, ds, 
  \end{equation}
 whenever $U\sub \oC$ is open. 
We will use these statements  in the proof of the next proposition. 

\begin{proposition} \label{thm:weakuptrans} 
Let  $\Om$ be  a  region in $\Sph$, and $\mathcal{K}=\{ K_i:i\in I\}$ a finite  collection of pairwise disjoint  compact sets  in $\Om$. 
Suppose that  the sets $K_i$ are $\lambda$-quasi-round and  $s$-relatively separated for $i\in I$.  

Then there 
is a non-increasing function $\phi\:(0,\infty)
\ra (0,\infty)$ that can be chosen only depending on $\lambda$ and $s$ with the following property: if $E$ and $F$ are arbitrary disjoint  continua
in $\overline \Om$, then 
$$M_{\Om,\mathcal{K}}(\Gamma(E,F;\Om))\le \phi(\Delta(E,F)). $$ 
\end{proposition} 

Here we cannot guarantee that $\phi(t)\to 0$ as $t\to \infty$. The point of the lemma is to have an upper bound for $M_{\Om,\mathcal{K}}(\Gamma(E,F;\Om))$ if $\Delta(E,F)$ is small. 

\medskip
\no {\em Proof.} 
Let $E$ and $F$ be arbitrary disjoint continua in $\overline \Om$, 
 $\Gamma=\Gamma(E,F;\Om)$, and 
 $\Delta(E,F)\ge t>0$. It suffices to produce 
 an admissible transboundary mass distribution for $\Ga$ whose mass can be bounded above by a constant only depending on 
 $s$, $\la$, and $t$. 
 
 For this we may assume 
 $d:=\diam(E)\le \diam(F)$. Then $\dist(E,F)\ge t d$.
 Let $K=\bigcup_{i\in I}K_i$ and  define
 $$\rho(u)= \frac 1{td}\quad \text{if}\quad u\in  N_{td}(E)\cap(\Om\setminus K) $$
 and $\rho(u)=0$ elsewhere. 
 Moreover, for $i\in I$ set 
 $$\rho_i= 1\wedge \biggl(\frac {\diam(K_i)} {td}\biggr)\quad \text{if} \quad K_i\cap N_{td}(E)\ne \emptyset$$
 and $\rho_i=0$ otherwise. 
 
 We claim that this transboundary mass distribution is admissible for 
 $\Ga$. To see this  let $\ga\in \Ga$ be an arbitrary  path that is locally rectifiable in $\Omega\setminus K$, and consider the  map $\pi\:  \oC\ra 
 [0,\infty)$ defined by $u \mapsto \dist(u,E)$. Since $\ga$ has an endpoint in $E$, but leaves the set $N_{td}(E)$, we have  \begin{eqnarray}\label{inclchain}
  [0, td)&=&\pi(\ga \cap N_{td}(E)) \nonumber \\ &\sub&
\{0\}\cup  \pi(\ga \cap N_{td}(E)\cap(\Om\setminus K))\cup
 \bigcup_{i\in I} \pi(\ga \cap N_{td}(E)  \cap K_i) \\
 &\sub & \{0\} \cup \pi(\ga \cap N_{td}(E)\cap(\Om\setminus K))\cup
 \bigcup_{\ga\cap N_{td}(E)\cap K_i\ne \emptyset} \pi( N_{td}(E)  \cap K_i).  \nonumber 
 \end{eqnarray}
All  subsets of $\R$ appearing in these inclusions are Borel sets as follows from the discussion before the statement of the proposition.
 
Inequality \eqref{proiectionest} applied to the map $\pi$, the set $U=  
 N_{td}(E)$, and the pieces of the path $\ga$ in $\Omega\setminus K$ 
 implies  that
 $$ m_1(\pi(\ga \cap N_{td}(E)\cap(\Om\setminus K) )) \le \int_{\ga \cap (\Om\setminus K)} \chi_{N_{td}(E)}\, ds = td \int_{\ga\cap (\Om\setminus K)} \rho \,ds. $$
Combining this with \eqref{inclchain} we obtain 
  \begin{eqnarray*}
  1&\le&  \frac1{td}m_1(\pi(\ga \cap N_{td}(E))) \\
  &\le &\frac 1{td} m_1(\pi(\ga \cap N_{td}(E)\cap(\Om\setminus K)))+\frac 1{td}
   \sum_{\ga \cap N_{td}(E)\cap  K_i\ne \emptyset}
  m_1( \pi( N_{td}(E)  \cap K_i))\\
  &\le&
 \int_{\ga\cap (\Om\setminus K)} \rho \,ds+ \frac1{td}\sum_{\ga \cap N_{td}(E)\cap  K_i\ne \emptyset} ((td)\wedge \diam (K_i))\\
 &\le & \int_{\ga\cap (\Om\setminus K)} \rho \,ds+\sum_{\ga\cap K_i\ne 0} \rho_i.
   \end{eqnarray*} 
    The admissibility of our transboundary mass distribution follows.   
 
 To estimate the total mass for our transboundary mass distribution, we define two subsets $I_1$ and $I_2$ of $I$
similarly as in the proof of 
 Proposition~\ref{prop:lowtrans}.  Namely, let $I_1$ be the set of all $i\in I$ such 
 that $N_{td}(E)\cap K_i\ne \emptyset$ and $\diam(K_i)<td $, and let $I_2$ be the set of all $i\in I$ such 
 that $N_{td}(E)\cap K_i\ne \emptyset$ and $\diam(K_i)\ge td $.
 
 Since each set $K_i$ is $\lambda$-quasi-round, we can find  disks $B(x_i,r_i)$ with $x_i\in \oC$ and $r_i\in (0,2]$ such that
 $$\overline B(x_i,r_i/\lambda) \sub K_i \sub \overline B(x_i, r_i).$$ 

If $i\in I_1$, then $\overline B(x_i,r_i/\la)\sub K_i\sub N_{2td}(E)$ and so 
$$\bigcup_{i\in I_1}  B(x_i,r_i/\lambda)\sub N_{2td}(E). $$
 Note that the balls in this union are pairwise disjoint and that the  set $N_{2td}(E)$ is contained in a ball of radius $(2t+1)d$ centered at any point in $E$. So using  \eqref{measball} we obtain
 \begin{eqnarray*}\sum_{i\in I_1}\rho_i^2
&\le &
\frac 1{t^2d^2}
\sum_{i\in I_1}\diam(K_i)^2 
\le \frac {4}{t^2d^2} \sum_{i\in I_1}r_i^2\\
&\le & \frac {4\lambda^2}{\pi t^2d^2} \sum_{i\in I_1}
 \Sigma(B(x_i, r_i/\lambda))\\
&=& \frac {4\lambda^2}{\pi t^2d^2} \Sigma\biggl
(\bigcup_{i\in I_1} B(x_i, r_i/\la)\biggl)\\
&\le& \frac {4\lambda^2}{\pi t^2d^2} 
\Sigma(N_{2td}(E))\\ &\le & \frac {4\lambda^2(2t+1)^2}{ t^2} =C_1(\la,t). 
\end{eqnarray*}

If $i\in I_2$, then $K_i\cap N_{td}(E) \ne \emptyset$ and 
$$\diam(K_i)\ge td\ge \frac{t}{2t+1}\diam(N_{td}(E)). 
$$

 Using Lemma~\ref{lem:curvemeets} for 
 $A=N_{td}(E)$, we conclude that  
$\#I_2\le C_2=C_2(s,t)$. Hence
 
 \begin{eqnarray*}\int \rho^2\,d\Sigma+\sum_{i\in I}\rho_i^2
& \le &
\frac1 {t^2d^2} \Sigma(N_{td}(E))+\sum_{i\in I_1}\rho_i^2+
\sum_{i\in I_2}\rho_i^2\\
&\le &\pi \frac{(t+1)^2}{t^2}+C_1(\la,t)+\sum_{i\in I_2}1\\
&\le &\pi \frac{(t+1)^2}{t^2}+ C_1(\la,t)+ C_2(s,t)=C(\la,s,t).
\end{eqnarray*}
The claim follows. 
 \qed

\medskip Let $(X,d)$ be a locally compact metric space, and
 $\nu$ be a Borel measure on $X$. 
A measurable set $M\sub X $ is called $\mu$-{\em fat} (for given $(X,d,\nu)$), where 
$\mu>0$, if for all $x\in M$ and all $0<r\le \diam(M)$ 
we have
$$\nu(M\cap B(x,r))\ge \mu\nu(B(x,r)).$$
In other words, a set $M$ is fat if the intersection of $M$  with every sufficiently small ball centered at a point in $M$ has  measure comparable to the measure of the whole ball. The notion of a fat set in the context of conformal mapping
theory  was introduced in \cite[Sect.~2]{oS95}.

A {\em (metric) annulus} in a metric space $(X,d)$ is a set $A\sub X$ of the form
$$A=A(x; r,R):=\{y\in X: r<d(y,x)<R\}, $$
where $x\in X$, $0<r<R< \diam(X)/2$. Note that by the restriction 
on $R$ both sets $\overline B(x,r)$ and $X\setminus 
B(x,R)$ are non-empty. We call  them 
the  {\em complementary parts}  of $A(x; r,R)$.
The {\em width} $w_A$ of the annulus $A=A(x; r,R)$
is defined as $w_A=\log(R/r)$.

If  $K\sub X$ is  a compact set with $K\cap A\ne \emptyset$, 
then we define two numbers that describe how the set lies relative to the annulus $A=A(x; r,R)$, namely 
$$ r_A(K):= \inf_{y\in K\cap A} d(y,x) 
\quad \text{and}\quad R_A(K):=
  \sup_{y\in K\cap A} d(y,x). $$
  Then $r\le r_A(K)\le R_A(K)\le R$.  
  We define the {\em width} $w_A(K)$ 
  {\em of $K$ relative to $A$} as 
  $$ w_A(K)=\log(R_A(K)/r_A(K)).$$ 
If $K\cap A= \emptyset$ it is useful to set 
$w_A(K)=0$. 
 
 In the following we consider annuli and fat sets in the metric space
 $(\oC, \sigma)$ equipped with the measure $\Sigma$. 
 Note that in this space a  closed disk $M=\overline B(a,R)$ with  $a\in \oC$ and   $0<R\le 2$,  is $\mu$-fat with $\mu=1/4$. Indeed, let  $x\in M$ and $r\le\diam(M)\le 2R$.
If $\sigma(a,x)\ge r/2$ we can pick a point $y\in M$ on the minimizing spherical geodesic segment connecting $x$ and $a$ with $\sigma(x,y)=r/2$. If 
 $\sigma(a,x)<r/2$ pick $y=a$. In both cases $B(y,r/2)\sub M\cap B(x,r)$ and so 
 $$\Sigma(M\cap B(x,r))\ge \Sigma(B(y,r/2))= \pi r^2/4=
 \Sigma(B(x,r))/4 . $$

\begin{lemma} \label{lem:ringfat}
Let  $K_1,\dots, K_n$
be  pairwise disjoint $\mu$-fat sets in $(\oC, \sigma, \Sigma)$, and  
suppose that there exists a metric   annulus $A\sub \oC$ with $w_A\ge 1$ such that each set $K_i$ meets both complementary parts  of $A$.
Then $n\le N(\mu)\in \N$.
\end{lemma}

So if pairwise disjoint $\mu$-fat sets 
meet both complementary parts  of a  sufficiently thick annulus in $\oC$, then the number of these sets is bounded by a constant 
only depending on $\mu$. 

\medskip\no
{\em Proof.}  Suppose that $A=A(x;r,R)$. Since $w_A\ge 1$, we have $R\ge er\ge 2r$. By our assumption each set 
$K_i$ meets $B=B(x,r)$ and the complement of $B'=B(x,2r)
\sub B(x,R)$. Hence  $\diam(K_i)\ge r$.  Picking a point 
$a_i\in K_i\cap B$, we see that 
$$ \Sigma(B'\cap  K_i)\ge\Sigma(B(a_i,r)\cap  K_i)  
\ge \mu \Sigma(B(a_i,r))=\pi \mu r^2.$$
Since the sets $K_i\cap B'$, $i=1, \dots, n$,  are pairwise disjoint 
and contained in $B'=B(x,2r)$, we conclude that the number 
of these sets is bounded above by 
$$\Sigma(B(x,2r))/(\pi \mu r^2)= 4/ \mu^2.$$ So for $N(\mu)$ we can take the smallest 
integer $\ge 4/\mu^2$.
\qed

\medskip 
For the proof of Theorem~\ref{thm:simulunif} we are interested in the case where 
the sets $K_1,\dots, K_n$ are pairwise disjoint closed disks in $\oC$. Then $\mu =1/4$ and the previous proof gives the bound $n\le 64$. It is not hard to see that if  $A$ is sufficiently thick, say $w_A\ge  100$, then actually  $n\le 2$.

\begin{lemma} \label{lem:subann}
Suppose that the collection $\{K_i:i\in I\}$
consists of pairwise 
 disjoint compact and $\mu$-fat sets in $(\oC, \sigma, \Sigma)$. Let $N=N(\mu)\in \N$ be a number as in 
 Lemma~\ref{lem:ringfat}.
 
 If $A=A(x;r,R)$ is an arbitrary annulus in $\oC$ with 
 $w_A\ge 1$, then there exists 
 a subannulus $A'=A(x;r',R')\sub A$ and a set $I_0\sub I$ with the following properties:
 
 \begin{itemize}
 
 \smallskip
 \item[(i)] $\#I_0\le N$,
 
 \smallskip
 \item[(ii)] $w_{A'}\ge w_A^{1/3^N}$,
 
 \smallskip
 \item[(iii)] $w_{A'}(K_i)\le w_{A'}^{1/3}$ for all $i\in I\setminus I_0.$
 \end{itemize}
\end{lemma}

This lemma will be applied when the width of $A$ is every large. It then says that by removing a controlled number of compact sets  in   the given collection, we can find a subannulus $A'$ of $A$ whose width is 
not much smaller than the width of $A$ and is much  larger then the width relative to $A'$ of the remaining sets in the collection. 

\medskip 
\no
{\em Proof.} If $w_A(K_i)\le w_A^{1/3}$ for all 
$i\in I$, we can choose $I_0=\emptyset$ and $A'=A$.

Otherwise, there exists $i_1\in I$ such that 
$w_A(K_{i_1})\ge w_A^{1/3}$.
Let $$A_1=A(x; r_A(K_{i_1}), R_A(K_{i_1})).$$
Then $K_{i_1}$ meets both complementary 
parts  of  $A_1$. This follows from the definitions of $r_A(K_{i_1})$ and  $R_A(K_{i_1})$, and the facts that $K_{i_1}$ is compact while $A_1$ is open. 
We also have  $w_{A_1}= w_A(K_{i_1})\ge w_A^{1/3}.$

If $w_{A_1}(K_i)\le w_{A_1}^{1/3}$ for all $i\in I\setminus 
\{i_1\}$, we choose $I_0=\{i_1\}$ and $A'=A_1$. Otherwise,
there exists $i_2\in I$, $i_2\ne i_1$ such that 
$w_{A_1}(K_{i_2})\ge w_{A_1}^{1/3}$. 
Define $$A_2=A(x; r_{A_1}(K_{i_2}), R_{A_1}(K_{i_2})).$$
Then $A_2$ is a subannulus of $A_1$ with 
$w_{A_2}\ge w_A^{1/3^2}$ and the sets $K_{i_1}$ and
$K_{i_2}$ meet both complementary parts  of 
$A_2$. 

Continuing in this manner we obtain a sequence of annuli 
$A_1,\dots, A_k$, and indices $i_1, \dots, i_k$. The process must stop after $k\le N$ 
steps, because otherwise we would obtain 
$N+1$ distinct $\mu$-fat sets 
 $K_{i_1}, , \dots, K_{i_{N+1}}$ 
that meet both complementary parts of the annulus  
$ A_{N+1}$. This is impossible by Lemma~\ref{lem:ringfat}, since  $w_{A_{N+1}}\ge 
w_A^{1/3^{N+1}} \ge 1$. 

The annulus  $A'=A_k$ and the set $I_0=\{i_1,\dots, i_k\}$ have the desired properties.
\qed
\medskip

\begin{proposition} \label{thm:uptrans} 
Let   $\mathcal{K}=\{ K_i:i\in I\}$ be a finite  collection of pairwise disjoint  
continua     in $\oC$. 
Suppose that the sets $K_i$ are 
 $\mu$-fat sets  in $(\Sph,\sigma, \Sigma)$ for $i\in I$, and let $N=N(\mu)\in \N$ be a number as in 
 Lemma~\ref{lem:ringfat}.
 
  Then there exists  a function
 $\psi\:(0,\infty)
\ra (0,\infty)$  with $$\lim_{t\to \infty}\psi(t)=0$$ that can be chosen only depending  on $\mu$ and satisfies the following property: 
 if $E$ and $F$ are arbitrary disjoint  continua
in $\Sph\setminus   \bigcup_{i\in I}\inte(K_i)$ with $\Delta(E,F)\ge 12$, then 
there exists a set $I_0\sub I$ with $\#I_0\le N$ 
such that for the transboundary modulus 
in the open set  $\Omega'=\Sph\setminus \bigcup_{i\in I_0}
K_i$ with respect to the collection $\mathcal{K}'=\{K_i:i\in I\setminus I_0\}$ we have
$$M_{\Om', \mathcal{K}'}(\Gamma(E,F;\Om'))\le \psi(\Delta(E,F)). $$ 
\end{proposition} 

In general, $\Omega'$ will only be an open subset of 
$\oC$ and not necessarily a region. The definitions of the path family $\Gamma(E,F;\Om')$ and of the transboundary modulus $M_{\Om', \mathcal{K}'}(\Gamma(E,F;\Om'))$   for an  open set $\Om'$  are exactly the same as for  regions. 
 Note that 
$$E,F\sub \oC\setminus\bigcup_{i\in I}\inte(K_i)=\overline 
{\oC\setminus\bigcup_{i\in I}K_i}\sub \overline 
{\oC\setminus\bigcup_{i\in I_0}K_i}=\overline {\Om}'.$$

To explain what the proposition  means suppose that 
$E$ and $F$ are continua in $\Sph\setminus 
\bigcup_{i\in I}\inte(K_i)$ whose relative distance is 
large. Consider the family $\Gamma$  of all paths  in $\Sph$ that connect $E$ and $F$. Then in general the transboundary modulus  of $\Gamma$ in $\Sph$ with respect to the collection $\{K_i:i\in I\}$ need not be small. 
The reason is that there could be some sets $K_i$ that are very close to both $E$ and $F$ and serve as a 
``bridge" between $E$ and $F$.
If there are many such bridges, the transboundary modulus of $\Gamma$ can  be large even if $E$ and $F$ have large relative distance.  The proposition  says that if we impose a uniform fatness condition on the sets $K_i$, and  remove  some elements $K_i$ from our collection, then the transboundary modulus of the family of paths 
connecting $E$ and $F$  in the complementary region of the discarded sets behaves in the expected way; namely, it is uniformly small if the relative separation of $E$ and $F$ is large. The sets $K_i$ that we have to  remove from the collection may depend on $E$ and $F$, but their number is uniformly bounded  
only depending on  the fatness parameter $\mu$. 

Using the remark following Lemma~\ref{lem:ringfat} one can show that if  the sets $K_i$ are round disks and $\Delta(E,F)$ is large enough,  one has to discard at most two disks in order to get a modulus bound of the desired type.
  
The restriction $\Delta (E,F)\ge 12$ in Proposition~\ref{thm:uptrans}
is not very essential and one can prove a more general version.  For this one has to find an appropriate   bound for $M_{\Om', \mathcal{K}'}(\Gamma(E,F; \Om'))$ also  for small  $\Delta(E,F)>0$.  This can be done  by an argument  very similar to  the proof   of Proposition~\ref{thm:weakuptrans}.  The present   version of Proposition~\ref{thm:uptrans} will be  sufficient for our purpose.

  In the proof of this proposition we need a variant of inequality 
  \eqref{proiectionest}. To formulate it,    
  let $(X,d)$ be a metric space, 
 $x\in X$, $\pi \: X\setminus\{x\}\ra \R $ be the  map defined by 
$u\in X\setminus\{x\} \mapsto \pi(u)=\log d(u,x)$, and  $\alpha$ be a locally rectifiable path in $X\setminus\{x\}$. Then  
\begin{equation}\label{proiest2}
m_1(\pi(\alpha))\le \int_\alpha \frac{ ds}{d(x,\cdot)},
\end{equation} 
where integration is with respect to arclength and $m_1$ again denotes $1$-dimensional Lebesgue measure. 

One can easily reduce this statement  to the case when 
$\alpha\:[0,L]\ra X\setminus\{x\} $ is  a rectifiable path in arclength parametrization, where $L=\length (\alpha)$.  By considering a suitable subpath and reversing orientation of the path if necessary  one can further assume that 
$p=\alpha(0)$ is a point on $\alpha$ with minimal distance to $x$,  
and $q=\alpha(L)$ a point with maximal distance.
Then $\pi(\alpha)=[\log d(x,p), \log d(x,q)]$, and so 
$$\int_\alpha \frac{ ds}{d(x,\cdot)}\ge \int_0^L \frac{ds}{d(x,p)+s}=\log\bigg(1+\frac{L}{d(x,p)}\biggr)\ge \log\bigg(\frac{d(x,q)}{d(x,p)}\biggr)=m_1 (\pi(\alpha))$$
as desired.

 \medskip\no
 {\em Proof of Proposition~\ref{thm:uptrans}.} Let $E$ and $F$ be disjoint continua in $\OC\setminus \bigcup_{i\in I}\inte(K_i)$ with $\Delta(E,F)=t\ge 12$. 
 We may assume that $\diam(E)\le \diam(F)$. 
 Pick a point $x\in E$, and define $r=2\diam (E)$
 and $R=\dist(E,F)/2$.  Then $R/r=t/4\ge 3$. Consider the annulus 
 $A=A(x;r,R)$. Then $E$ is contained in $B(x,r)$ and 
 $F$ in the complement of $B(x,R)$. So the annulus $A$ separates the sets $E$ and $F$. Moreover, 
  $$w_A=\log(R/r)= \log (t/4)\ge 1.$$
Then we can find a subannulus  $A'=A(x;r',R')$ of $A$ and a set $I_0\sub I$ as in Lemma~\ref{lem:subann}.

We define $\Omega'=\Sph\setminus\bigcup_{i\in I_0}K_i$,
and consider the transboundary modulus of the path
family $\Gamma(E,F;\Omega')$  in $\Omega'$ with respect to the
collection $\mathcal{K}'=\{K_i:i\in I\setminus I_0\}$. We set  
$K':=\bigcup_{i\in I\setminus I_0}K_i$.

 We have to find a bound for 
 $M_{\Omega', \mathcal{K}'}(\Gamma(E,F;\Omega'))$  depending on $t$ and $\mu$ that is  small  if $t$ is large. 
 We define a transboundary mass distribution as follows. 
 We let 
 $$ \rho(u)=\frac 1{w_{A'}\sig(u,x)} \forr u\in A'\cap (\Omega'\setminus  K')$$
 and $\rho(u) =0 $  elsewhere.
 Moreover, we let $$\rho_i=w_{A'}(K_i)/w_{A'} \forr 
 i\in I\setminus I_0 \text{ with }K_i\cap A'\ne \emptyset,$$
 and $\rho_i= 0$ for all other $i\in I\setminus I_0$.
 
 We claim that this transboundary mass distribution is admissible 
 for $\Gamma(E,F;\Omega')$. To see this let $\ga \in
 \Gamma(E,F;\Omega')$ be an arbitrary path that is locally rectifiable in $\Om'\setminus K'$. Since $A'$ is a 
 subannulus of $A$, it also separates $E$ and $F$. 
  Hence $\gamma$ meets 
 both complementary parts of $A'$, and so there exists an open  subpath $\alpha$ of $\gamma$ that lies in 
 $A'$   and connects the components of 
 the boundary of $A'$. 
 Obviously,
 $$ \int_{\ga\cap (\Omega'\setminus  K')}\rho\,ds+\sum_{i\in I\setminus I_0,\, K_i\cap \ga \ne \emptyset}\rho_i \ge 
 \int_{\alpha\cap  (\Omega'\setminus  K')}\rho\,ds+\sum_{i\in I\setminus I_0,\, K_i\cap \alpha \ne \emptyset} \rho_i .$$
 We want  to show that the right hand side of this inequality  is bounded below by  $1$.
 
 Let $\pi$ be the map on  $\overline A'$ to the interval 
 $[\log r', \log R']$ defined by $u\mapsto \pi(u):=\log \sig(u,x)$.
 Then 
 \begin{equation}\label{eq:proiect}
 (\log r', \log R')= \pi(\alpha)\sub \pi (\alpha\cap (\Omega'\setminus K'))
 \cup \bigcup_{i\in I\setminus I_0, \, \alpha\cap K_i\ne \emptyset}\pi({A'}\cap K_i).
 \end{equation}
By using \eqref{proiest2} for $d=\sigma$ and the pieces of $\alpha$ in $\Omega'\setminus K'$, we see that  
$$\int_{\alpha\cap (\Omega'\setminus K')}\rho\, ds  \ge \frac 1 {w_{A'}}
m_1 (\pi(\alpha\cap (\Omega'\setminus K'))).$$
We also have 
$$\rho_i\ge  \frac 1 {w_{A'}}
m_1(\pi({A'}\cap K_i)) \text{ for all } i\in I\setminus I_0, $$ 
and so  \eqref{eq:proiect} implies 
\begin{eqnarray*}
 \int_{\alpha\cap(\Omega'\setminus K')}\rho\, ds+\sum_{i\in I\setminus I_0,\, K_i\cap \alpha \ne \emptyset} \rho_i&\ge &\\
  \frac 1 {w_{A'}}  m_1(\pi(\alpha\cap (\Omega'\setminus K')))
\!\!\!&+&  \!\!\!\frac 1 {w_{A'}} \sum_{i\in I\setminus I _0,\, \alpha \cap K_i\ne \emptyset}  m_1(\pi({A'}\cap K_i))\\
&\ge & \frac 1 {w_{A'}} m_1((\log r', \log R'))=1.
\end{eqnarray*}
The admissibility of our transboundary mass distribution follows.

To obtain mass bounds for our transboundary mass distribution, we first 
note that 
\begin{equation} \label{integral}
\int_{A'}\frac {d\Sigma(u)}{\sig(u,x)^2}=\int_{r'}^{R'}\frac{d(\Sigma(B(x,v))}{v^2}=2\pi\int_{r'}^{R'}\frac{dv}{v}=2\pi \log(R'/r')=
2\pi w_{A'}.
\end{equation}

Hence for the density  part  of the mass we get 
\begin{equation}\label{eq:contpart}
\int_{\Omega'\setminus  K'}\rho^2\, d\Sigma \le \frac 1{w_{A'}^2}\int_{A'}
\frac {d\Sigma(u)}{\sig(u,x)^2}=
\frac{2\pi}{w_{A'}}.
\end{equation}

To estimate the mass of the discrete part, we 
consider two subsets $I_1$ and $I_2$ of $I\setminus I_0$.
Namely, let $I_1$ be the set of all $i\in I\setminus I_0$ such that $A'\cap K_i\ne\emptyset$ and $w_{A'}(K_i)\le \log 2$,
and  $I_2$ be the set of all $i\in I\setminus I_0$ such that $A'\cap K_i\ne\emptyset $ and $w_{A'}(K_i)>\log 2$.
Since $\rho_i=0$ for all $i\in I\setminus I_0$ with $A'\cap K_i=\emptyset$, we have 
\begin{equation}\label{eq:split}
\sum_{i\in I\setminus I_0} \rho_i^2=
\sum_{i\in I_1} \rho_i^2 +\sum_{i\in I_2} \rho_i^2.
\end{equation}

For $i\in I_1\cup I_2$ let $r_i:=r_{A'}(K_i)$ and 
$R_i:=r_{A'}(K_i)$. For these $i$ we then have
$$ \rho_i=\frac 1{w_{A'}}\log(R_i/r_i) $$
and  $\diam(K_i)\ge (R_i-r_i)$. 

Since $K_i$ is connected, we can find a point $a_i\in A'\cap K_i$ with 
$\sig(a_i,x)=\tfrac 12(r_i+R_i).$
Let $B_i:=B(a_i,\tfrac12 (R_i-r_i))\sub A'$.
The disk $B_i$ is centered at a point in $K_i$ and has a radius not exceeding the diameter of $K_i$. 
Hence the $\mu$-fatness of $K_i$ gives 
\begin{equation} \label{eq:massAintK}
\Sigma(A'\cap K_i)\ge\Sigma(B_i\cap K_i)
\ge \mu \Sigma(B_i)=\frac \pi 4 \mu (R_i-r_i)^2,
\end{equation}
and so
\begin{equation}\label{eq:integr}
\int_{A'\cap K_i}
\frac {d\Sigma(u)}{\sig(u,x)^2}\ge 
\frac {\Sigma(A'\cap K_i)}{R_i^2}
\ge\pi \mu \frac {(R_i-r_i)^2}{4R_i^2}.
\end{equation}
Now if $i\in I_1$, then $R_i\le 2r_i$ and so
$$\log(R_i/r_i)=\log\biggl(1+\frac{R_i-r_i}{r_i}\biggr)\le 
\frac{R_i-r_i}{r_i}\le 2\frac{R_i-r_i}{R_i}.$$

These inequalities imply
\begin{eqnarray}\label{eq:first}
\sum_{i\in I_1}\rho_i^2&=&\frac1{w_{A'}^2}
\sum_{i\in I_1}\log(R_i/r_i)^2 \nonumber\\
&\le &\frac{16  }{\pi \mu w_{A'}^2}
\sum_{i\in I_1} \int_{A'\cap K_i}
\frac {d\Sigma(u)}{\sig(u,x)^2}\\
&\le &\frac{16 }{\pi \mu w_{A'}^2}\int_{A'}\frac 
{d\Sigma(u)}{\sig(u,x)^2}\le 
\frac{32}{\mu w_{A'}}. \nonumber
\end{eqnarray}

If $i\in I_2$, then $R_i>2r_i$ and so \eqref{eq:integr} shows that
$$ 
\int_{A'\cap K_i}
\frac {d\Sigma(u)}{\sig(u,x)^2}\ge \frac\pi{16}\mu.
$$
This implies that
$$\frac\pi {16} \mu\cdot\#I_2\le\sum_{i\in I_2} \int_{A'\cap K_i}
\frac {d\Sigma(u)}{\sig(u,x)^2}\le 
\int_{A'}\frac {d\Sigma(u)}{\sig(u,x)^2}\le 2\pi w_{A'}.
$$
Hence 
$$\#I_2\le \frac{32}{\mu} w_{A'}.$$
By choice of $I_0$ according to Lemma~\ref{lem:subann}, we have
$$ w_{A'}(K_i)<w_{A'}^{1/3}$$
for all $i\in I\setminus I_0$. Using this with the upper bound on $\#I_2$ 
 we conclude
\begin{eqnarray}\label{eq:scnd}
\sum_{i\in I_2}\rho_i^2 &=& \frac1{w_{A'}^2}
\sum_{i\in I_2} w_{A'}(K_i)^2 
\le \frac1{w_{A'}^2}\cdot \#I_2\cdot w_{A'}^{2/3} \\
&\le &  \frac{32}{\mu w_{A'}^{1/3}}.\nonumber
\end{eqnarray}
By choice of $A$  and $A'$ we have 
$w_{A'}\ge w_A^{1/3^N}= \log (t/4)^{1/3^N}\ge 1.$
Combining this with \eqref{eq:contpart},  \eqref{eq:first}, and \eqref{eq:scnd}, we arrive at the bound

\begin{eqnarray*}M_{\Om', \mathcal{K}'}(\Gamma(E,F;\Om'))&\le&
 \int_{\Omega'\setminus  K'}\rho^2\, d\Sigma +\sum_{i\in I\setminus I_0}\rho_i^2  \le 
\frac {2\pi }{w_{A'}}+ \frac{32}{\mu w_{A'}}+   \frac{32}{\mu w_{A'}^{1/3}}\\ &\le& \frac{C(\mu)}{w_{A'}^{1/3}}\le 
\frac{C(\mu)}{\log(t/4)^{1/3^{N+1}}}.
\end{eqnarray*}
Since $N=N(\mu)$ this gives the desired uniform bound in $\mu$ and $t$ that becomes small
if $t$ becomes large. 
\qed \medskip

 \begin{remark} \label{rem:uptrans} \textnormal{ The previous proposition holds in greater generality.
 Namely, suppose that we have a region $U\sub \oC$ equipped with a path metric $d$ induced by a conformal length element $ds_\lambda=\lambda ds$ and an associated measure $\nu$  such that $d\nu=\lambda^2d\Sigma$ as in Remark~\ref{rem:chb}. 
 Suppose also that there exists a constant $C_0\ge 1 $ such that 
 \begin{equation}\label{eq:uppmass}
 \frac{1}{C_0}r^2 \le  \nu(B_d(a,r))\le C_0r^2
  \end{equation}  
 whenever $a\in U$ and $0<r\le \diam_d(U)$.}
 \textnormal{Then an analog  of Proposition~\ref{thm:uptrans} holds
 in the metric measure space $(U, d, \nu)$ (instead of $(\OC,\sigma, \Sigma)$)
 with a  constant $N=N(\mu, C_0)$ and a function $\psi=\psi_{\mu, C_0}$.}
 
 \textnormal{  Indeed, it is clear that  versions  of Lemmas~\ref{lem:ringfat} and \ref{lem:subann}  are true    in this greater generality with a constant $N=N(\mu, C_0)$.  Based on this, the proof of Proposition~\ref{thm:uptrans} can easily be adapted by  changing  the metric $\sigma$ to  $d$ and the measure $\Sigma$ to  $\nu$. 
 All inequalities will  remain valid up to an adjustment of the multiplicative constants.  The upper mass bound in 
 \eqref{eq:uppmass}  is used to derive  an    inequality for the analog of the integral on the left hand side  in
 \eqref{integral}   for sufficiently thick annuli $A'$.  The bound will be a  multiple of $w_{A'}$ with 
 a suitable constant depending on $C_0$. The lower mass bound  \eqref{eq:uppmass} is used in the proof  of Lemma~\ref{lem:ringfat} and in \eqref{eq:massAintK}. Actually, in both cases we only need the lower mass bound for disks $B_d(a,r)$ with $r\le 
 \sup_{i\in I} \diam(K_i)$.} 
 
\textnormal{  We will   later  formulate a specific case explicitly in Proposition~\ref{prop:uptrans2}.  
} \end{remark}

 \section{Classical uniformization }
\label{s:unifcyl} 
\no
In this section we discuss some facts related to  classical uniformization. 
The main result is  Theorem~\ref{thm:squareunif}. It  can be derived from  the remark in \cite[p.~412]{oS96} on periodic uniformization.  We will give a different  proof based on   the methods developed in \cite{oS96}.  We will use some standard facts from complex analysis such as Montel's Theorem, the Argument Principle, etc. See, for example, 
 \cite{Ru} for precise statements and general background. 

  We consider {\em finitely connected} regions in $U\sub \OC$, i.e., regions  with finitely many  complementary components. The region is called {\em labeled} if its complementary  components are labeled by the numbers $0, \dots, n$, i.e., if a bijection between the set of complementary  components and the set $\{0, \dots, n\}$ has been  specified.  Here we assume that there are $n+1\ge 2$ complementary components. Two labeled regions are considered  equal if the underlying sets are the same and the labels on complementary  components agree. If $U$ is a labeled region, then  
we denote the component of the complement   with label $i$ by 
$\widehat \partial_i U$ and the boundary of this component by
$\partial_iU$. Then we have 
$$\partial U= \partial_0U\cup \dots \cup \partial_nU.$$

Let  $f\: U \ra V$ be a conformal map between labeled regions
$U$ and $V$.  Then  the number of complementary components of $U$ and $V$  is   the same and the map  $f$ induces a bijection $\phi$ on $\{0,\dots, n\}$ with the following property:  for all $i=0, \dots, n$ and all  sequences  $(z_k)$ in $U$ with 
 $z_k\to   \partial_iU$ we have  $f(z_k)\to  \partial_{\phi(i)}V$ (see
  \cite[Sect.~15.3]{Co}, and in particular \cite[p.~81, Prop.~15.3.2]{Co}). 
 The map $f$ is {\em label-preserving} if $\phi$ is the identity on 
 $\{0,\dots, n\}$.

 A complementary or boundary component of a region $U$ is called {\em degenerate} or {\em non-degenerate} depending on whether 
 it consists of one or of more than one point. If $f\:U\ra V$ is  a 
 label-preserving conformal map between labeled regions, then 
for each $i=0, \dots, n$ the component  $\widehat \partial U_i$ 
is degenerate if and only if  $\widehat \partial V_i$ is degenerate; this easily follows from the fact that points are removable singularities for bounded analytic functions.

\begin{lemma}\label {lem:Iordanext} Let $f\:U\ra V$ be a
label-preserving conformal map between labeled regions $U$ and $V$ in $\oC$ with  finitely   boundary components. If $\partial_0U$ and 
$\partial_0V$ are Jordan  curves, then there exists a unique  extension of 
$f$ to a homeomorphism from $U\cup\partial_0U$ onto
$V\cup\partial_0V$. 
\end{lemma}

This follows from \cite[p.~83, Thm.~15.3.6 (b)]{Co}  (note that the  points on   $\partial_0U$ and 
$\partial_0V$ are {\em simple} boundary points of $U$ and $V$, respectively; see \cite[p.~52, Def.~14.5.9]{Co}  and \cite[p.~53, Cor.~14.5.11]{Co}). See  also  Remark~\ref{rem:indproof} below where an outline of the proof will be given. 
Lemma~\ref{lem:Iordanext}  implies that  if all boundary components of $U$ and $V$ are Jordan  curves or degenerate, then $f$ extends to a homeomorphism 
from $\overline U$ onto $\overline V$ (see also \cite[p.~82, Thm.~15.3.4]{Co}). 

We need a statement similar in spirit to Lemma~\ref{lem:Iordanext} on  uniform convergence  of sequences of conformal maps ``up to the boundary".  It relies on  some equicontinuity result for boundary maps which will be derived from the  
following well-known fact.  

\begin{lemma}[Wolff's Lemma]
\label{lem:wolf}  Let $U\sub \C$ be open,  $z_0\in \C$, 
  and
$f\:U\ra \C$ be a conformal map with  $f(U)\sub \D$. For $r>0$ let 
$\gamma_r=U\cap\{z\in \C:|z-z_0|=r\}.$ 

Then for all 
$0<t<1$ there exists $s\in (t,\sqrt t)$ such that 
$$ \textnormal{length}_\C(f(\ga_s))\le \frac{2\pi}{\sqrt {\log(1/t)}}.$$ 
\end{lemma}

Note that $\gamma_r=U\cap\{z\in \C:|z-z_0|=r\}$ is a circle or consists of a countable collection of open circular arcs. In the statement $\textnormal{length}_\C(f(\ga_s))$ denotes the total Euclidean length of the images under $f$ (recall from Section~\ref{s:nt} that the subscript 
$\C$ refers to the Euclidean metric on $\C$). 

For the proof of Lemma~\ref{lem:wolf} see    \cite[p.~20, Prop.~2.2]{Pom}.  

A {\em crosscut} in an open Jordan  region $D\sub \oC$ is an arc $\alpha$ whose interior points 
lie in $D$ and whose endpoints lie on $\partial D$. 
A  crosscut  $\alpha$ in $D$ {\em separates} two points $p\in D$ and $q\in \partial D$  if every path $\phi\:[0,1]\ra \OC$ with  $p=\phi(0)$, $q=\phi(1)$, and $\phi([0,1))\sub D$   meets $\alpha$. 

The set of points on $\partial D$ separated by a crosscut $\alpha$ in $D$ from a given point  $p\in D\setminus \alpha$     is equal to one of the subarcs $\gamma$ of $\partial D$ with the same endpoints 
as $\alpha$.  
In the special case $D=\D$ and $p=0$ a simple argument shows 
 there is a constant $c_0>0$ such that  if $\diam_\C(\alpha)\le c_0$, then $\gamma$ is the smaller arc on 
$\partial \D$ with the same endpoints as $\alpha$, and so 
$\diam_\C(\gamma)\le \diam_\C(\alpha)$.  
 It is not hard to see that $c_0=1$   is the sharp constant in this statement.   

Based on this and the Sch\"onflies Theorem one can show  that if  $D\sub \oC$ is an arbitrary open Jordan  region,  and $p\in D$,   
then for every $\eps>0$ there exists $\delta>0$ such that for every crosscut $\alpha$ in $D$ with $p\notin \alpha$ and  $\diam(\alpha)<\delta$ we have  $\diam(\gamma)<\eps$ for the arc $\gamma$ of points on $\partial D$ separated by $\alpha$ from $p$.

\begin{lemma}[Equicontinuity of boundary maps]
\label{lem:unifcont}  Let $r\in (0,1)$, 
$$A=\{z\in \C: r<|z|\le 1\},$$  and let 
$\mathcal{F}$ be the family of all homeomorphism $f$ on $A$ that are conformal on $\inte(A)$ so that  $f(A)\sub \overline \D$,   $f(\partial \D)=\partial \D$, and $0$ is contained in the bounded component of $\C\setminus f(\inte(A))$. 

Then the family $\{f|\partial \D: f\in   \mathcal{F}\}$ of boundary maps is equicontinuous with respect to the Euclidean metric.
Moreover, every sequence in $\mathcal{F}$ has a subsequence that converges to a function in $\mathcal{F}$ uniformly  on compact subsets of  $A$. 
\end{lemma}

The existence of a subsequence that converges uniformly  on compact subsets of $\inte(A)$ immediately follows from Montel's Theorem. The point here is that we get uniform convergence ``up to the boundary" $\partial \D$, i.e., on compact subsets of $A$. 

Note that for $f\in \mathcal{F}$ the set $f(\inte(A))$ has two complementary components.  One is equal to the complement of $\D$ while the other is a compact subset of $\D$.  

\medskip
\no{\em Proof.} Let $\eps>0$, $z_0,z_1\in \partial \D$, and 
$f\in \mathcal{F}$ be arbitrary. We may assume that $\eps<1/10$. 

Suppose that  $\delta>0$ and $|z_1-z_0|<\delta$. Lemma~\ref{lem:wolf} implies that 
if  $\delta>0$ is small enough only depending on $\eps$, then there exists  a crosscut $\alpha$ in $\D$ that lies in $A$,  separates the 
 points $z_0$ and $z_1$ from each point on the circle $\{z\in \C:|z|=r\}\sub \partial A$ and satisfies $\length_\C (f(\alpha))<\eps$. Then  $\beta=f(\alpha)$ is also 
a crosscut in $\D$, and it separates $f(z_0), f(z_1)\in \partial \D$ from each point in the bounded  component of $\C\setminus f(A)$, and hence from $0$ by our hypotheses.  Since $\length_\C(\beta)<\eps<1/10$, this implies that $f(z_0)$ and $f(z_1)$ lie on the smaller subarc 
of $\partial \D$ determined by the endpoints of $\beta$.
This arc has diameter bounded by the diameter of $\beta$. 
Hence
$$|f(z_0)-f(z_1)|\le \text{length}_\C(\beta)<\eps. $$
The equicontinuity of the family of boundary maps follows. 

Let $(f_n)$ be an arbitrary sequence in $\mathcal{F}$. 
Since the sequence is  uniformly bounded, Montel's Theorem implies that there exists a subsequence that converges uniformly on compact subsets of $\inte(A)$. By the first part 
of the proof, we know that the maps $f_n|\partial \D$ are equicontinuous. Hence by passing to a further subsequence, we may assume that our subsequence converges uniformly 
on $\partial \D$.  

Replacing our original sequence by such a subsequence,  we may assume that   $(f_n)$ converges uniformly on compact subsets of $\inte(A)$ 
and uniformly on $\partial \D$. We claim that this implies that 
$(f_n)$ converges uniformly on compact subsets of $A$. 
Indeed, if $K\sub A$ is an arbitrary compact set, then there exists $r<r'<1$ such that 
$$K\sub A'= \{z\in \C: r'\le |z|\le 1\}. $$ 
The circle $\{z\in \C: |z|=r'\}$ is a compact subset of $\inte(A)$, so the convergence of our sequence $(f_n)$ is uniform on this set. Morever, $(f_n)$ converges uniformly on $\partial \D$. 
The Maximum Principle implies that  the sequence converges 
uniformly on $A'$ and hence on $K$.  

Let $f$ be the limit function of the sequence $(f_n)$. Then $f$ is continuous on $A$. Moreover, by Hurwitz's Theorem   $f$ is either constant or a conformal map on $\inte(A)$. Here the former case is impossible,
because we have $f_n(\partial \D)=\partial \D$ and so 
$f(\partial \D)=\partial \D$; indeed, if $y\in \partial \D$ is arbitrary, and $x$ is any sublimit of a sequence $(x_n)$ in $\partial \D$ with $f_n(x_n)=y$ for each $n$, then $f(x)=y$. 

The set $f(\inte(A))$ is a region in $\D$. It follows from the Argument Principle that every point in $\D$ that is sufficiently close to $\partial \D$ lies in $f(\inte(A))$. Hence 
one of the boundary components of $f(\inte(A))$ is $\partial \D$, and so  
 Lemma~\ref{lem:Iordanext}  implies that $f$ is a 
homeomorphism  on $A$.  Since the value $0$ is not attained by any of the functions $f_n$, the function $f$ does not attain  $0$ either. This implies that $0$ is contained in the bounded component of $\C\setminus f(\inte(A))$. We conclude that $f\in \mathcal{F}$.
\qed 
\medskip

\begin{remark}\label{rem:indproof}\textnormal{We referred to Lemma~\ref{lem:Iordanext} in the proof of the previous lemma.  By using similar  ideas  as in the previous proof based on Lemma~\ref{lem:wolf} and the fact on crosscuts in Jordan regions  mentioned before Lemma~\ref{lem:unifcont},   one can actually easily give a proof of Lemma~\ref{lem:Iordanext}.   One first shows the uniform continuity of the map $f$ in Lemma~\ref{lem:Iordanext} on a
(topological)  annulus  $A\sub U$ 
with $\partial_1U\sub \partial A$. This implies that  $f$ has a continuous extension to $\partial_1U$. This extension 
satisfies $f(\partial_1U)\sub \partial_1V$. The map $f|\partial _1U$
is a homeomorphism of $\partial_1U$ onto $ \partial_1V$, because an inverse map
can be obtained by applying the same argument to $f^{-1}$ on $V$.
}
\end{remark}

A region $V\sub \OC$ is called a {\em circle domain} if
its complementary components are round, possibly degenerate, disks.  The following theorem is one of the landmarks of classical uniformization theory. 

\begin{theorem}[Koebe's  Uniformization Theorem]\label{thm:Koebunif}
Let $U\sub \OC$ be a region with finitely many complementary components. Then there exists a conformal map $f\:U\ra V$ of $U$ onto a circle domain $V$. The map $f$ is unique up to post-composition with an orientation-preserving  M\"obius transformation.
\end{theorem}

See \cite[p.~106, Thm.~15.7.9]{Co} for the existence, and 
\cite[p.~102, Prop.~15.7.5]{Co} for the uniqueness statement (note that in \cite{Co} this is only formulated for regions $U$ whose complementary components are non-degenerate, but our more general 
version can easily be derived from this).
\medskip 

If we equip $\C$  with the Euclidean metric, then  the cyclic 
group $\Ga$ generated by the translation $z\mapsto z+2\pi \iu$, where $\iu$ is the imaginary unit,  acts on $\C$ by isometries.   The Riemannian quotient
$\C/\Gamma$  is isometric to the 
infinite cylinder $Z=\{(x,y,z)\in \R^3: x^2+y^2=1\}$  with the Riemannian metric induced from $\R^3$. The exponential function induces an isometry of $\C/\Gamma\cong
Z$ 
with $\C^{*}=\C\setminus\{0\}$. Here $\C^{*}$ is equipped with 
the {\em flat metric} $d_{\C^*}$ induced    by the length element
$$ds_{\C^*}=\frac{|dz|}{|z|}. $$
We use terminology  for sets in $\C^{*}$ that is  suggested 
by this identification of $\C^{*}$ with the cylinder $Z$.

A {\em finite $\C^{*}$-cylinder}   is a set $A$ of the form
$$A=\{z\in \C: r<|z|<R\}, $$
where $0<r<R$. We denote by $\partial_i A=\{z\in \C: |z|=r\}$ its {\em inner}, and by $\partial_o A=\{z\in \C: |z|=R\}$ its {\em outer} boundary component. The {\em height} $h_A$ of the finite $\C^*$-cylinder $A$ is the 
quantity $h_A=\log(R/r)$. A {\em $\C^*$-square} $Q$ 
is a set of the form 
$$Q=\{\rho e^{\iu t}: \alpha\le t\le\beta \text { and } 
r\le \rho\le R\}, $$
where $\alpha<\beta$, $\beta-\alpha<2\pi$, $0<r<R$,
and $\beta-\alpha=\log(R/r)$. 
The {\em side length} $\ell(Q)$  of $Q$ is defined as $$\ell(Q)=\beta-\alpha=\log(R/r).$$
Note that $0< \ell(Q)<2\pi$.  We call the point 
$ p_Q=\sqrt {rR}e^{\iu (\alpha+\beta)/2}$
the {\em center} of $Q$.   Sometimes it is useful to allow the case of  {\em degenerate} $\C^*$-squares, where $r=R$, $\alpha=\beta$, and $\ell(Q)=0$. Then $Q$ only  consists   of the  point $p_Q$. \medskip

In the following we fix $n\in \N$.  We denote by $\mathcal{S}$  the  set 
of  labeled regions $U \sub \C^*$  that can be written 
as 
\begin{equation}\label{sqdom}
U=\D\setminus (Q_1\cup\dots \cup Q_n) 
\end{equation}
where the sets 
$Q_1, \dots, Q_{n}$ are pairwise disjoint subsets of $\D$ 
such that  $Q_1, \dots , Q_{n-1}$ are $\C^*$-squares, and $Q_n$ is a closed Euclidean disk centered at $0$.  Here we allow degenerate squares and disks,  i.e., sets  consisting of only one point.   The complementary  components 
of $U$ as in \eqref{sqdom} are the sets $$\OC\setminus \D, 
 Q_1, \dots, Q_{n},  $$ 
We assume they are labeled by $0, \dots, n$ in this order. 

Similarly, we denote by $\mathcal{C}$ the set 
of all labeled regions $V \sub \D$ that can be written 
as 
\begin{equation}\label{circdom}
V=\D\setminus  (C_1\cup \dots \cup C_n), 
\end{equation} 
where 
$C_1, \dots, C_{n}$ are pairwise disjoint closed Euclidean disks contained in $\D$ such that $C_n$ has center $0$.  Again we allow degenerate  disks $C_i$ consisting  of only one point.  
Moreover, we assume that the complementary  components  
$$\OC\setminus  \D,  C_1, \dots, C_{n}$$ of $V$ are labeled by $0, \dots, n$, respectively. 

There are natural identifications  of the  spaces $\mathcal{S}$
and $\mathcal{C}$ with certain (relatively) open and connected subsets 
 $S$ and $C$ of $\D^{n-1}\times [0,\infty)^n$, respectively. 
Indeed, if a region $U\in \mathcal{S}$ is written as in 
\eqref{sqdom}, we let it correspond to the 
point 
$$x=(p_1, \dots, p_{n-1}, r_1, \dots, r_n)\in\D^{n-1}\times  [0,\infty)^n ,$$ 
where $p_i\in \D$ is the center and $r_i\ge 0$ is the sidelength
of the $\C^*$-square $Q_i$ for $i=1, \dots, n-1$, and $r_n\ge 0$ is the radius of $Q_n$. 
It is clear that the correspondence $U\leftrightarrow x$ gives 
a bijection of the space $\mathcal{S}$
and an   open  subset $S$ of $\D^{n-1}\times [0,\infty)^n$. The set $S$ is path-connected and hence connected. 
Indeed, if $x\in S$ is arbitrary, then we can get a path in $S$ 
connecting $x$ to a basepoint in $S$  by performing the following procedure on the region $U$ corresponding to $x$:  we shrink the complementary components of $U$ to points, and then 
move these points in $\D$ to prescribed positions   while avoiding collisions of the points. 

Similarly, if $V\in \mathcal{C}$ is written as in \eqref{circdom}, we let it correspond to the point
$$y=(q_1, \dots, q_{n-1}, s_1, \dots, s_n)\in\D^{n-1}\times [0,\infty)^n,$$ 
where $q_i\in \D$ is the center and $s_i\ge 0$ is the radius 
of the disk $C_i$ for $i=1, \dots, n-1$, and $s_n\ge 0 $ is the radius of $C_n$.
Again we get  
a bijection of the space $\mathcal{S}$
and a  open and connected subset $S$ of $\D^{n-1}\times [0,\infty)^n$.

We need a criterion when a sequence $(x_k)$ in one of the sets $S$ or $C$ has a convergent subsequence with a limit in the set.   For this the only obstacle for sequences in $C$  is when the  corresponding  regions have complementary components that get close to each other. For sequences in 
$S$ there is the additional obstacle that  some of the  complementary $\C^*$-squares of the regions may  ``wrap around" the cylinder $\C^*$ and have sidelengths approaching $2\pi$. The following lemma 
gives a simple condition that prevents these phenomena.  

In the proof we use  {\em Hausdorff convergence}
of sets. We remind the reader of the definition of this concept. Let   $(A_k)$  be  a sequence of closed subsets of a metric space $(X,d)$.  We say that the sequence $(A_k)$ {\em Hausdorff converges}  to another closed set 
$A\sub X$, written as $A_k\to A$ as $k\to \infty$, 
if for all $\eps>0$ there exists $N\in \N$ such that 
$A\sub N_\eps(A_k)$ and $A_k\sub N_\eps(A)$ whenever  $k>N$.  
We will use this for subsets of $X=\oC$. Unless otherwise specified,  
$d$ will then be the chordal metric on $\oC$.  If the sets under consideration are contained in a compact subset of $\C$,  one can  alternatively use the Euclidean metric.

\begin{lemma}[Subconvergence criterion] \label{lem:subconvcrit}Let   $(x_k)$ be a 
sequence in $S$ (or $C$), and $U_k$ be the labeled 
region in $\mathcal{S}$ (or $\mathcal{C}$) corresponding to 
$x_k$ for $k\in \N$. Suppose that there exist pairwise disjoint closed 
Jordan  regions 
$D_1, \dots, D_n \sub \D$ such that $\widehat \partial_i U_k
\sub D_i$ for all    $i=1,\dots, n$ 
and all $k\in \N$. Then the sequence $(x_k)$ has a subsequence that converges to a point in $S$ (or $C$).  
\end{lemma}

\no {\em Proof.} We will only prove the statement if $n\ge 2$ and $(x_k)$ is a sequence in $S$.
The cases when $n=1$ or when  $(x_k)$ is a sequence in $C$ are  similar and easier.  

Note that $0$ is contained in each of the sets $\widehat \partial_nU_k$, and so $0\in D_n$.   
For each $k\in \N$ the boundary component $\widehat \partial_1 U_k$ of $U_k$ is a $\C^*$-square  contained in 
$D_1\sub \D$. By passing to a subsequence if necessary, we may assume that for $k\to \infty$ the centers of the $\C^*$-squares $\partial_1 U_k$  converge to a point $c_1\in D_1$ and their sidelengths converge to a number $l_1\in [0,2\pi]$. We claim that $l_1<2\pi$. 

For otherwise, $l_1=2\pi$. Then a limiting argument shows that 
the circle $\{z\in \C: |z|=|c_1|\}$ is contained in $D_1$. Since $D_1$ is a Jordan  region, this implies that 
$\{z\in \C: |z|\le |c_1|\}\sub D_1$ and so $0\in D_1$. On the other hand, $0\in D_n$. Since $n\ne 1$,  and $D_1$ and $D_n$ are disjoint by hypothesis we get a contradiction. 
So  $l_1<2\pi$.

Let $Q_1$ be
the (possibly degenerate) $\C^*$-square  with center $c_1$ and sidelength $l_1$. Then 
$\widehat \partial_1U_k\to  Q_1$ as $k\to \infty$ in the sense of Hausdorff convergence, and $Q_1\sub D_1$. 

A similar argument  shows that by passing to successive subsequences
if necessary,  we may assume that $\partial_iU_k\to Q_i$ as $k\to \infty$, where $Q_i\sub D_i$ is a $\C^*$-square for $i=1, \dots, n-1$, and a closed disk centered at $0$ for $i=n$.  
Since the sets $D_1, \dots, D_n$ are pairwise  disjoint subsets of $\D$, the same is true for the sets $Q_1, \dots, Q_n$.
It follows that   $U=\D\setminus (Q_1\cup \dots\cup Q_n)$ 
is a region in $\mathcal{S}$, where the complementary  components
$$\OC\setminus  \D,  Q_1, \dots,  Q_n$$  of $U$ 
are labeled by the numbers $0, \dots, n$ in this order. 
If $x\in S$ is the point corresponding to $U$, then it is clear that 
$(x_k)$ subconverges to $x$. \qed
\medskip

We  define a map $\eta\: S\ra C$ as follows. 
Let $x\in S$ be arbitrary, and $U\in \mathcal{S}$ be the labeled region corresponding to $x$.  By Koebe's  Uniformization 
Theorem there exists a conformal map $f$ of $U$ onto a circle domain $ V$,  unique up to post-composition with a M\"obius transformation. 
We label the complementary  components of $V$ so that $f$ is 
label-preserving.  By post-composing $f$ by a  M\"obius transformation we may assume that 
the boundary component of $V$ with label $0$ is the unit circle 
$\partial \D$ and that the one with label $n$ is of  the form
$\{z\in \C: |z|=s\}$ with $0\le s<1$. By the remark following 
Lemma~\ref{lem:Iordanext} the map  $f$ has an  extension, also denoted $f$,  to a homeomorphism from $\overline U$ to 
$\overline V$. By 
post-composing $f$ by a suitable rotation, we may also assume that $f(1)=1$.  So $f$ is normalized such 
that 
$$f(\partial \D)=\partial \D, \ f(\partial_nU)=\partial_nV=\{w\in \C:|w|=s\},  \text{ where } 0\le s<1, \text{ and } f(1)=1. $$  Note that with these normalizations the map $f$ 
is  uniquely determined, and the labeled region  $V=f(U)$ lies in $\mathcal{C}$. We let $y\in C$ be the point corresponding to $V$, and set $\eta(x):=y$. 

Our goal is to show that $\eta$ is surjective. 
We need some preparation.

\begin{lemma}\label{lem:sepalim} For $k\in \N\cup\{\infty\}$ let  
 $\varphi_k\: \partial \D\ra J_k:=\varphi_k(\partial\D)\sub \oC$ be   homeomorphisms such that $\varphi_k\ra \varphi_\infty$ uniformly on $\partial \D$ as $k\to \infty$. 
For $k\in \N$ let $M_k\sub \OC\setminus J_k$ be a set whose points are separated by $J_k$ from a basepoint $p\in \OC\setminus 
\bigcup_{k\in \N\cup\{\infty\}} J_k$. 
If $$\delta:=\liminf_{k\to \infty} \dist(M_k,J_\infty)>0,$$ then 
$J_\infty$ separates the points in $M_k$ from $p$ for all large enough $k$. 
\end{lemma}

Here we say that a Jordan curve $J\sub \oC$ separates two points $a,b\in  \oC$  if 
$a$ and $b$ lie in different complementary components of $J$. 

\begin{proof} We may assume that $p=\infty$. Then $J_k$ is a 
Jordan curve in $\C$ for all $k\in \N\cup \{\infty\}$. One of the two closed Jordan regions in $\oC$ bounded by $J_k$ contains $\infty$. Let 
 $D_k\sub \C $ be the other one, and let $\alpha_k$ be the loop defined by 
  $\alpha_k(t)=\varphi_k(e^{\iu t})$ for $t\in [0,2\pi]$. 
Then  a point $a\in \C\setminus J_k$ is separated from $p=\infty$ 
 by $J_k$ if and only if the winding number of the loop 
 $\alpha_k$
 around $a$ is non-zero; indeed, this winding number 
is $\pm1$ for points in $\inte(D_k)$ depending on the orientation
of $\alpha_k$, and $0$ for points in $\C\setminus D_k$. 

By our hypotheses we have $M_k\sub \C\setminus N_{\delta/2}(J_\infty)$ for large enough $k$, say for $k\ge k_1$. Moreover,
since $\varphi_k\to \varphi_\infty$ as $k\to \infty$ uniformly on $\partial \D$, for all large enough $k$, say for $k\ge k_2$, the loop $\alpha_k$ lies in $N_{\delta/2}(J_\infty)$ and is homotopic to $\alpha_\infty$  in $ N_{\delta/2}(J_\infty)$. Then for  $k\ge k_2$ the winding numbers of $\alpha_k$ and $\alpha_\infty$ around any point in $\C\setminus 
N_{\delta/2}(J_\infty)$ are the same. By our hypotheses  the winding number of $\alpha_k$ around any  point in $M_k$ is $\pm1$. Hence 
for $k\ge k_1\vee k_2$ the winding number of $\alpha_\infty$ around any point $a\in M_k$ is $\pm1$, and so  $J_\infty$ separates $a$ from $p$. 
 \end{proof}

\begin{lemma}\label{lem:etacont}
The map $\eta$ is continuous.
\end{lemma}

\no{\em Proof.}
 Let $(x_k)$ be an arbitrary sequence in $S$ with 
$x_k\to x_\infty\in S$ as $k\to \infty$. Define $y_k=\eta(x_k)$ and 
$y_\infty=\eta(x_\infty)$. We will show that there exist a subsequence 
$(y_{k_l})$ of $(y_k)$ such that $y_{k_l}\to y_\infty$ as $l\to \infty$. 

Since $(x_k)$ is arbitrary, this fact can then also be applied to any subsequence of $(x_k)$. Hence for every subsequence of $(y_k)$ there exists a ``subsubsequence" that converges to $y_\infty$. 
This implies that the sequence $(y_k)$ itself converges to $y_\infty$, and   the continuity of $\eta$ follows. 

It remains to produce the subsequence $(y_{k_l})$ with 
$y_{k_l}\to y_\infty$.  For  $k\in \N\cup\{\infty\}$ let $ U_k\in \mathcal{S}$ be the  labeled region corresponding to 
$x_k\in S$,  $ V_k\in \mathcal{C}$ be  the labeled region corresponding to 
$y_k\in S$,  and  for $k\in \N$ let $f_k\:\overline U_k\ra \overline V_k$ be the homeomorphism
as in the definition of $\eta$. 

Since $x_k\to x_\infty$, every compact subset of $U_\infty\cup \partial \D$ lies in  $ U_k\cup \partial \D$ for sufficiently large $k$. 
In particular, the map  $f_k$ is  defined on each such set for sufficiently large $k$. 
 
Using the second part of Lemma~\ref{lem:unifcont} together 
with Montel's theorem, we can find 
 a subsequence  of the sequence  $(f_k)$ that converges uniformly on compact subsets of $U_\infty\cup \partial \D$.  By replacing our original sequence by this subsequence, we  may assume that   $(f_k)$ itself  has this convergence property, and that the limit map $f$ is a 
 homeomorphism     on $U_\infty\cup\partial \D$ that  is  conformal on $U_\infty$, and satisfies $f(\partial\D)=\partial \D$ and  
 $f(1)=1$.

Let  $D_1, \dots, D_n\sub \D$  be pairwise disjoint closed Jordan  regions with 
$$ \widehat \partial_i U_\infty \sub  \inte(D_i)$$ 
for $i=1, \dots, n$. Such Jordan  regions can be found by slightly  enlarging the complementary components $\widehat \partial_1 U_\infty ,\dots,\widehat \partial_n U_\infty$ of $U_\infty$. 
For each $i=1, \dots, n$ and all    large enough $k\in \N$ we then have $\widehat \partial_iU_k\sub \inte (D_i)$, and so   $\partial D_i$  separates the points in $\widehat \partial_iU_k$ from the point $1$.  

Let   
 $$
 U=\D\setminus (D_1\cup \dots \cup D_n).$$ 
Then $ \overline U\sub U_\infty\cup \partial\D$. 
The image 
 $V=f(U)$ can be written as 
$$V=\D\setminus (E_1\cup \dots \cup E_n),$$ 
where $E_1, \dots, E_n\sub \D$ are pairwise disjoint closed Jordan  regions  with
$f(\partial D_i)=\partial E_i$ for  $i=1, \dots,n$. 

Since $f_k\ra f$ uniformly on compact subsets of $U_\infty$, the Argument Principle implies that if $K\sub U_\infty$ is compact, then $f(K)\sub f_k(U_k)$  for all large enough $k$.  In particular,  a small neighborhood of $\partial E_i=f(\partial D_i)$ will lie in $V_k=f_k(U_k)$ if $k$ is large enough. Hence we can choose  $\delta>0$  such that   
$\dist(\partial E_i, \widehat \partial_iV_k)\ge \delta $ for all $i=1, \dots, n$ and all large enough $k$.  We also have   $\widehat \partial_iU_k\sub \inte(D_i)$ and $\partial D_i\sub U_k$ for large enough $k$. Then    $f_k(\partial D_i)$ separates the points in $\widehat  \partial_iV_k$ from the point
$1$. Since   $f_k\to f$ uniformly on $\partial D_i$, it follows from Lemma~\ref{lem:sepalim} that $\partial E_i=f(\partial D_i)$ separates 
the points in  $ \widehat \partial_iV_k$ from $1$ for all $i=1, \dots, n$ and all large enough $k$.  
For such $k$ we then have $\widehat \partial_iV_k\sub E_i$ for all $i=1,\dots, n$. 
By Lemma~\ref{lem:subconvcrit} we may  
pass to another subsequence if necessary and  assume that $y_k\to \tilde y_\infty\in C$.  

The proof will be complete if we can show that $\tilde y_\infty=y_\infty$. Let $ \widetilde V_\infty \in \mathcal{C}$ be the labeled region corresponding to $\tilde y_\infty$.  By definition of the map $\eta$ it is enough to show that $f$ is a label-preserving 
conformal map of $U_\infty$ onto $  \widetilde V_\infty$. For then $f$ has the right  normalization and so $\tilde y_\infty=\eta(x_\infty)= y_\infty$. 

It is clear that $\widehat \partial_i\widetilde V_\infty \sub E_i$ for $i=1, \dots n$.  If we let $D_i$ shrink to $\widehat \partial_i U_\infty$, then the corresponding Jordan  region $E_i$ shrinks to $\widehat \partial_i
f(U_\infty)$. Hence   $\widehat \partial_i V_\infty \sub \widehat \partial_i
f(U_\infty)$ for $i=1, \dots n$, and so   $f(U_\infty)\sub  \widetilde V_\infty$.  These inclusions  show that if in addition $f(U_\infty)\supseteq\widetilde V_\infty$, 
 then $f$ is a label-preserving conformal map between $U_\infty$ and $ \widetilde V_\infty$ as desired.

In order to establish $f(U_\infty)\supseteq \widetilde V_\infty$, we  repeat  the argument in the first part of the proof for the sequence $(f_k^{-1})$.  Passing to a subsequence if necessary, we may assume that     it converges uniformly on compact subsets of 
  $\widetilde V_\infty\cup \partial\D$ to a conformal map $g$.  
  The same proof as above then shows that 
  $g(\widetilde V_\infty)\sub U_\infty$. 
  Hence the map $f\circ g$ is well-defined on $\widetilde V_\infty$ and by uniform convergence we have  
  $$ f(g(w))=\lim_{k\to \infty} f_k(f_k^{ -1}(w))=w$$
  for $w\in \widetilde V_\infty$. This implies that $f(U_\infty)\supseteq  \widetilde V_\infty$
  as desired.  
\qed \medskip

A map between topological  spaces $X$ and $Y$  is called {\em proper} if the preimage of each compact subset of $Y$ is a compact subset of $X$. 

\begin{lemma}\label{lem:etaprop}
The map $\eta$ is proper. 
\end{lemma} 

\no{\em Proof.} We claim that  every point $y\in C$ has 
a neighborhood $N\sub C$ such that $\eta^{-1}(N)$ is
relatively compact  in $S$. Given this  claim every compact set $K\sub C$ can be covered by finitely many 
such neighborhoods $N_1, \dots, N_m$. Then 
$$ \eta^{-1}(K)\sub \eta^{-1}(N_1)\cup \dots \cup \eta^{-1}(N_m)$$
is relatively compact. The set $\eta^{-1}(K)$ is also closed, since $K$ is closed and $\eta$ is continuous.  Hence $\eta^{-1}(K)$ is compact as desired.

It remains to prove the claim.  Let $y\in C$ be arbitrary, and let $V\in \mathcal{C}$ be the labeled region corresponding 
to $y$. By enlarging the complementary   
 components of $V$ with labels $1, \dots, n$ slightly,  we can 
find a neighborhood $N$ of $y$ in $C$ and  closed Jordan  
regions $ D'_i\sub  \inte (D_i)\sub  D_i\sub \D$   for $i=1, \dots, n$ with the following 
properties: the regions $D_1, \dots, D_n$ are pairwise disjoint,
 and 
 if $\widetilde V\in \mathcal{C}$ is any region corresponding to a point in $N$, then 
$\widehat \partial_i \widetilde V\sub D_i'$ and for all $i=1, \dots,n$.  In particular,  if  $\Om=\D\setminus(D_1'\cup\dots\cup D_n')$, then 
 $\partial D_1\cup \dots \cup \partial D_n\sub   \Om \sub \widetilde  V$.    
  
Now let $(x_k)$ be an arbitrary sequence in  $ \eta^{-1}(N)$, let   $U_k\in \mathcal {S}$ be the labeled region corresponding to $x_k$ and $f_k$ be the  
map  on $\overline U_k$ as in the definition of $\eta$ for $k\in \N$. Then $g_k=f_k^{-1}$ is defined on $\Om\cup \partial \D$ by choice of $N$. 

  Lemma~\ref{lem:subconvcrit} and Montel's Theorem imply that by passing to a subsequence if necessary, we may assume that the sequence $(g_k)$ converges 
to a map $g$ uniformly on compact subsets of $\Om\cup \partial \D$. The map $g$ is a homeomorphism on $\Om\cup \partial \D$, is conformal on $\Om$, and we have $g(\partial \D)=\partial \D$ and $g(1)=1$.  Let $E_i$ for $i=1, \dots, n$ be the closed Jordan  region in $\D$ bounded by the Jordan  curve
$g(\partial D_i)\sub \D$. Then  the regions $E_1, \dots,  E_n$ are pairwise disjoint. 

Similarly as in the proof of Lemma~\ref{lem:etacont} 
one can show that   $\widehat \partial_i U_k\sub E_i$ for all $i=1, \dots , n$ and all large enough $k$. By  Lemma~\ref{lem:subconvcrit}  the sequence $(x_k)$ subconverges to a point in $\mathcal{S}$. Hence $\eta^{-1}(N)$ is precompact. 
 \qed \medskip

\begin{lemma} \label{lem:degree} Let $A$ and $B$ be (relatively) open and connected subsets of $\C^m\times [0,\infty)^n$, where $m,n\in \N$.  Assume that 
$$A_0=\{(p,0)\in A: p\in \C^m, 0\in \R^n\}=A\cap(\C^m\times\{(0,\dots, 0)\})\ne \emptyset$$ and that $\eta\: A\ra B$ is  a proper and continuous map satisfying the following conditions:
for  arbitrary $x=(p_1, \dots, p_m, r_1,\dots, r_n)\in A$ and 
$\eta(x)=(q_1, \dots , q_m, s_1, \dots, s_n)\in B$ we require that
\begin{itemize} 
\smallskip
\item[\textnormal{(i)}] if $r_i=0$ for some  $i=1, \dots, n$, then $s_i=0$,  

\smallskip
\item[\textnormal{(ii)}] if $r_1=\dots=r_n=0$, then 
$p_1=q_1$, \dots, $p_m=q_m$, and $s_1=\dots=s_n=0$.

\end{itemize} 
Then the map $\eta$ is surjective.
\end{lemma}
This is a special case of \cite[3.4 Degree Lemma, p.~407]{oS96}.

\begin{lemma} \label{lem:suri} The map $\eta\: S\ra C$ is surjective. 
\end{lemma}

\no {\em Proof.} If $n=1$ this is clear, because then $S=C$ and $\eta$ is the identity map. If $n\ge 2$, we apply  Lemma~\ref{lem:degree} with 
$A=S$, $B=C$, and $m=n-1\ge 1$. Then obviously $A_0\ne \emptyset$, as the points in this  set correspond to the regions $U$ in $\mathcal{S}$ with degenerate  complementary components $\widehat \partial_i U$, $i=1, \dots, n$.  The map  $\eta$ is continuous by Lemma~\ref{lem:etacont} and proper by 
Lemma~\ref{lem:etaprop}. Condition (i) in  Lemma~\ref{lem:degree}  follows from that fact that the map $f$ as in the definition of $\eta$ will send   a degenerate  complementary component of a region in $\mathcal{S}$
to a  degenerate  complementary components of a region in $\mathcal{C}$. Moreover,  if all the complementary components except the one with label $0$ are degenerate, then $f$ is the identity map.  This implies 
condition (ii) in Lemma~\ref{lem:degree}. 
\qed  \medskip

\begin{theorem}\label{thm:squareunif} Let $\Om\sub \OC$ 
be a region with $n+1\ge 2$ complementary components, one of which is non-degenerate. Then there exists
a conformal map of $\Om$ onto a region $U$ of the form 
$$U=\D\setminus (Q_1\cup \dots \cup Q_n),$$ where 
 $Q_1, \dots, Q_n$ are pairwise disjoint subsets of 
$\D$ such that  $Q_1,\dots, Q_{n-1}$ are  (possibly degenerate)
$\C^*$-squares, and $Q_n$ is a closed (possibly degenerate) Euclidean disk  centered at $0$. 
 \end{theorem}
 
\no {\em Proof.}  Koebe's  Uniformization Theorem implies that  there is a conformal map of  $\Om$ onto a circle domain 
$V\in \mathcal{C}$. Since the map $\eta\:S \ra C$ introduced above is surjective, $V$,  and hence also $\Om$, is conformally equivalent to a region $U\in \mathcal{S}$. \qed  
 \medskip 
 
 If all the complementary components of $\Om$ are non-degenerate, the same is true for the region $U$ in the previous theorem. Combining this with Lemma~\ref{lem:Iordanext}, we get the following statement.  
 \begin{corollary}\label{thm:cylunif} Let $n\ge 1$, and 
suppose that $D_0,\dots, D_n$ are pairwise disjoint 
closed Jordan  regions in $\OC$. Then there 
exist a finite $\C^*$-cylinder $A$, pairwise disjoint $\C^*$-squares 
$Q_1, \dots, Q_{n-1}\sub A$,
and a homeomorphism $f\: \overline \Om\ra \overline U$, where 
$$\Om=\OC\setminus (D_0\cup \dots \cup D_n)
\quad \text{and}\quad U=A\setminus (Q_1\cup \dots \cup Q_{n-1}), $$
that is conformal on $\Om$ and maps 
$\partial D_0$ to $\partial_iA$ and  $\partial D_n$ to $\partial_oA$.
\end{corollary}

As we will see in Corollary~\ref{cor:uniqsquare}, the map $f$ in the previous statement is unique up to post-composition with a Euclidean similarity fixing the origin.

\section{Proof of the main result}
\label{s:proof1} 

\no We start with a definition. 
A set  $\Om\sub \oC$  is called $\la$-$LLC$ for $\la\ge 1$ ($LLC$ stands for
  {\em linearly locally connected}) 
if the following two conditions are satisfied:

\smallskip \noindent 
($\la$-$LLC_1$):$\quad$
 If $a\in \Om$, $r>0$, and  $x,y\in \Om \cap B(a,r)$, $x\ne y$, then there 
exists a continuum $E\sub \Om\cap  B(a,\la r)$ with $x,y\in E$.   

\smallskip \noindent 
($\la$-$LLC_2$):$\quad$
 If $a\in \Om$, $r>0$, and $x,y\in \Om \cap (\oC\setminus B(a,r)) $, $x\ne y$, 
 then there exists a con\-tinuum $E\sub \Om \cap (\oC\setminus
  B(a,r/\lambda))$
 with $x,y\in E$. \smallskip
 
 \begin{lemma} \label{lem:circ1LLC}
 Every  finitely connected circle domain $V\sub \oC$  is  $1$-$LLC$. 
 \end{lemma}

 \medskip 
\no{\em Proof.} Let $V$ be as in the statement and  $B\sub \oC$ be an 
arbitrary open or  closed  disk.   It suffices to show that $B\cap V$ is path-connected. Indeed, if this is true, then it follows that $V$ is $1$-$LLC_1$.  Noting  that the complement of an open disk  in $\oC$ (as in 
the $LLC_2$-condition) is a closed disk (centered at the point antipodal to the center of the original disk), we will also have that $V$ is $1$-$LLC_2$.  

We first  show that $B\cap \overline V$ is path-connected. Let  $x,y\in B\cap \overline V$ be  arbitrary. We  can connect $x$ and $y$ by a path 
$\alpha$ in $B$.   If $D$ is  
one of the disks which form the complementary components of $V$, then $B\cap \partial D$ is a connected set  (this is an elementary geometric fact where it is important that  $B$ and $D$ are round disks).  So if $\alpha$ meets $\inte(D)$, then we can replace a subpath 
of $\alpha$ by a path in $B\cap \partial D$, so that the new path lies in $B$, connects $x$ and $y$, but is disjoint from 
$\inte(D)$. By repeating this procedure for  the other complementary components of $V$, we finally obtain a path $\beta$ in $B$ that 
connects  $x$ and $y$ and avoids the interior of each complementary component of $V$. Then $\beta$ is a path in  $B\cap \overline V$ 
connecting $x$ and $y$. 

To show that  $B\cap  V$ is also path-connected, let again   $x,y\in B\cap  V$ be  arbitrary. By slightly enlarging the radii of the complementary components of $V$, we can find a finitely connected circle domain 
$V'\sub \oC$ with $x,y\in V'$ and $\overline V'\sub V$. Then by the first part 
of the proof there exists a path   $\beta$  in  $B\cap \overline V'\sub 
B\cap  V$ 
connecting $x$ and $y$.  The path-connectedness of $B\cap V$ follows.
  \qed 
\medskip 

 A {\em Schottky set} is a set $T\sub \oC$ that can be written 
as 
$$ T=\oC\setminus \bigcup_{i\in I} B_i,$$ 
where $\{B_i: i\in I\}$   is a collection of pairwise disjoint round open disks. 
One can define the notion of a Schottky set similarly for subsets of higher-dimensional spheres.  This concept was introduced in  \cite{BKM} (with the additional requirement $\# I\ge 3$). 
By  \cite[Prop.~2.2]{BKM}  every Schottky set is $1$-$LLC$
(the condition $\#I\ge 3$ is irrelevant for this conclusion). Since the  closure of every circle domain is a Schottky set, one can derive 
Lemma~\ref{lem:circ1LLC} from this result. We included a complete proof 
of this lemma for the convenience of the reader. 
 See also the related Lemma~\ref{lem:LLCsq} below. \medskip

The following theorem is the main ingredient in the proof of Theorem~\ref{thm:simulunif}. It is  of independent interest.

\begin{theorem} \label{thm:unifqs} 
Let  $U=\oC\setminus \bigcup_{i\in I}D_i$ be  a finitely connected region whose complementary components $D_i$ 
are closed  Jordan  regions that are 
$s$-relatively separated and whose boundaries $\partial D_i$   are $k$-quasicircles for $i\in I$. Assume  that 
$0,1,\infty\in U$.

 If $f\:U\ra V$  is a conformal map of $U$ onto a circle domain $V$ with $f(0)=0$, $f(1)=1$, and $f(\infty)=\infty$, then $f$ is $\eta$-quasisymmetric with 
$\eta$ only depending on $s$ and $k$.
  \end{theorem} 

\no{\em Proof.} The map $f$ extends uniquely to a homeomorphism between $\overline U$ and $\overline V$ (see the discussion after Lemma~\ref{lem:Iordanext}). Moreover, 
we can further extend this  map  (non-uniquely) to a homeomorphism on  $\Sph$ (this  follows from Remark~\ref{rem:ext} as in the proof of Lemma~\ref{lem:misc}~(iii)). We keep denoting this homeomorphism on $\oC$ by $f$, and  use a prime to denote image points under $f$,
 i.e., $a'=f(a)$ for $a\in \Sph$.

Since all subsets of $\C$ are $N_0$-doubling with a universal constant 
$N_0$, by Proposition~\ref{prop:wkqs} it suffices to show that $f|U$ is  $H$-weakly
quasisymmetric with $H=H(s,k)$. So we have to show that there exists a constant $H=H(s,k)$ with the following property: if  $x,y,z\in U$ are arbitrary, then
$\sig(x,y)\le \sig(x,z)$ implies that 
$\sig(x',y')\le H\sig(x',z')$. 

Assume on the contrary that for some  points 
 $x,y,z\in U$  with $\sig(x,y)\le \sig(x,z)$
 we have 
 $\sig(x',y')> H\sig(x',z')$ for some large $H>>1$.  We will show that 
 this leads to a contradiction if $H$ is chosen large enough
 depending only on $s$ and $k$.  
 
 Note that under our assumption  the points $x,y,z$ must be  distinct.
 Since $V$ is a finitely connected circle domain, this set  $1$-$LLC$ by Lemma~\ref{lem:circ1LLC}.  So we can find a continuum $E'\sub V$ with $x',z'\in E'$  such that 
 $$\diam(E')\le 3\sig(x',z').$$ 
 
The points $0,1,\infty$ have mutual distance bounded below by $\sqrt 2\ge 1$. So by Lemma~\ref{lem:3norm} 
 we can find a point $u=u'\in \{0,1,\infty\}$ such that 
 $\sig(u,y)\ge 1/2$ and 
 $\sigma(u',x')\ge 1/2$. Since  $\sig(x',y')\le \diam(\Sph)= 2$, we then have $\sigma(u',x')\ge \frac14\sig(x',y')$, and so 
 $u'\not \in  B(x', \frac14\sig(x',y')) $.  Again using the $1$-$LLC$-property of $V$, this allows us to find a continuum $F'\sub V$ with 
 $y',u'\in F'$ such that 
 $$F'\cap B(x', \tfrac1{4}\sig(x',y'))=\emptyset.$$
 
 Then assuming that $H\ge 24$ we have
 $$ \dist(E',F')\ge\frac1{4}\sig(x',y')-3\sig(x',z')\ge  \frac1{8}\sig(x',y')\ge \frac H{8}
 \sig(x',z')
 $$
 and $$\diam(E')\wedge \diam(F')\le 3\sig(x',z').$$
 Therefore,
 $$ \Delta(E',F')\ge \frac H{24}.$$ 

Let $K_i=f(D_i)$ for $i\in I$. Then $\{K_i:i\in I\}$ is the  collection of complementary components of $V$. 
 Since $V$ is a circle domain, every set $K_i$ is a closed round disk.  Round disks  are $\mu$-fat in $(\oC, \sigma, \Sigma)$ with $\mu=1/4$ (see the discussion before Lemma~\ref{lem:ringfat}).   If $N=N(1/4)$ the corresponding integer 
 as provided by Lemma~\ref{lem:ringfat} ($N$ is a universal constant;  as we have seen,  one can take $N=64$, or even $N=2$), then 
 Proposition~\ref{thm:uptrans}  allows us the following 
 conclusion. There exists a universal non-increasing function $\psi\:(0,\infty)\ra (0,\infty)$ such that $\lim_{t\to \infty}\psi(t)=0$ with the following property:  for  some set 
 $I_0\sub I$ with $\#I_0\le N$ we have that 
 $$M_{\Om', \mathcal{K}'}(\Gamma(E',F';\Om'))\le \psi(\Delta(E',F'))\le \psi(H/24),$$ where 
$\Om'=\Sph\setminus\bigcup_{i\in I_0} K_i$
and transboundary modulus is with respect to the collection 
$\mathcal{K}'=\{K_i:i\in I\setminus I_0\}$.  

Define  $E=f^{-1}(E')$ and $F=f^{-1}(F')$. Then $E$  and $F$ are  continua in $U$ containing the sets $\{x,z\}$ and 
$\{y,u\}$, respectively.
Then $$\diam(F)\ge \sig(y,u)\ge \tfrac 12\ge \tfrac 14 \diam(E),$$ 
and 
$$\dist(E,F)\le \sig(x,y)\le \sig(x,z)\le \diam(E)
\le 4 (\diam(E)\wedge \diam(F)). $$
It follows that 
\begin{equation}\label {sep}
\Delta(E,F)\le 4.
\end{equation}

Since $\#I_0$ can be bounded by the  universal constant 
$N$, it follows from Proposition~\ref{thm:lowmod} that the region $\Om=
\Sph\setminus \bigcup_{i\in I_0}D_i$ is $\phi$-Loewner with 
$\phi$ only depending on $s$ and $k$.
Combining this with Proposition~\ref{prop:qround}, Proposition~\ref{prop:lowtrans}, and  \eqref{sep}, we see that there is a positive constant 
$m=m(s,k)>0$ such that 
$$ M_{\Om,\mathcal{K}}(\Gamma(E,F; \Om))\ge m,$$
where the transboundary modulus in $\Om$ is with respect to the collection $\mathcal{K}=\{D_i:i\in I\setminus I_0\}.$

Now $f(\Gamma(E,F; \Om))=\Gamma(E',F'; \Om')$
and $f$ is conformal on the set 
$\Om\setminus \bigcup_{i\in I\setminus I_0}D_i=U$.
Hence $$M_{\Om, \mathcal{K}}(\Gamma(E,F; \Om))=
M_{\Om',\mathcal{K}'}(\Gamma(E',F'; \Om'))$$
by invariance of transboundary modulus (see the discussion after Lemma~\ref{lem:invtrans})
and our estimates give
$$ m\le \psi(H/24).$$
Since $\psi$ is a fixed function with $\psi(t)\to 0$ as 
$t\to \infty$, this leads to a contradiction if $H$ is larger than a constant depending on $s$ and $k$.
\qed
\medskip  

Note that the homeomorphic extension $f\:\overline U\ra \overline V$ of the map in the previous theorem is also an $\eta$-quasisymmetry with the same function $\eta$ as for the 
map $f|U$. This follows  from  the distortion estimates for $f$ on $U$ by  a simple limiting argument. 

The previous proof is somewhat technical and it is worthwhile to summarize the main ideas of the argument. If the map $f$ does not have the desired quasisymmetry property, then, as we have seen, one can 
find continua $E$ and $F$ in $U$ with controlled relative distance such that 
the relative distance of the image continua $E'$ and $F'$ in $V$ is large. 
To get a contradiction one wants to consider a suitable family $\Gamma$ of paths
 connecting $E$ and $F$, and its  image family $\Gamma'$. 
Since $\Delta(E,F)\lesssim 1$, one  hopes to find  $\Gamma$ so that the modulus of this family is  not too small, while $\Delta(E',F')>>1$ should imply that  the modulus of $\Gamma'$ is small.  Conformal invariance of modulus will then give the desired contradiction. 

The  obvious first choice $\Gamma=\Gamma(E,F;U)$ cannot serve this 
purpose. Even though the  complementary components of $U$  are uniformly relatively separated, by restricting oneself to paths {\em in} $U$,  it is possible to    obtain a very sparse family 
whose modulus is not uniformly bounded below by a constant only depending on the relevant parameters $s$ and $k$. Using this  family 
$\Gamma(E,F;U)$ in combination with 
 Proposition~\ref{thm:lowmod}, one can actually show that $f$ is $\eta$-quasisymmetric, where $\eta$ will depend  on $s$ and $k$,  but  also on the number of complementary components of $U$ 
 (for which we have no control). 

To get a larger 
path family  
one should allow the paths to run through the ``holes" (i.e., the complementary components) of  $U$ 
and $\Gamma=\Gamma(E,F; \oC)$ seems like a better choice. 
It is clear that then one has to use transboundary modulus to get 
the necessary modulus invariance.  Proposition~\ref{prop:lowtrans} applied to 
the Loewner domain $\Om=\oC$ and the family of all complementary components of $U$, then actually gives a uniform lower  bound for the transboundary boundary modulus of 
$\Ga=\Gamma(E,F; \oC)$. Unfortunately, the corresponding transboundary modulus of the image family $\Ga'=\Gamma(E',F'; \oC)$
need not be small due to possible complementary  components 
of $V$ that serve as ``bridges" between $E'$ and $F'$. As discussed after  Proposition~\ref{thm:uptrans}, one can remedy this problem 
by disallowing the paths to run through certain holes that have to be selected depending on $E'$ and $F'$, but whose number is bounded by 
a universal constant $N$.   Accordingly, in the previous proof we considered the family $\Gamma=\Gamma(E,F; \Om)$, where 
$\Om=\oC\setminus \bigcup_{i\in I_0} D_i$,  and its image family 
$\Gamma'$. The paths  in these families are allowed to pass through  holes except through those labeled by $i\in I_0$. By Proposition~\ref{thm:uptrans}
the family $\Gamma'$ has small transboundary modulus. 
Even though we have no control which elements are in $I_0$, we have 
a uniform upper bound $\#I_0\le N$.  So  Proposition~\ref{thm:lowmod} allows us to conclude that $\Om$ is $\phi$-Loewner with a function 
$\phi$ only depending on $s$ and $k$ (the number $n=\#I_0$ of complementary components of $\Om$ does not enter as it is uniformly bounded). Together with Proposition~\ref{prop:lowtrans} this leads to a lower bound for the transboundary modulus of $\Gamma$. The crucial point  in this argument is that all quantities that are relevant in  the upper and lower estimates can be controlled by the parameters $s$ and $k$.
Hence  $f$ will be an $\eta$-quasisymmetry with $\eta=\eta_{s,k}$. 

Another subtlety in the previous proof is the initial choice of the continua $E'$ and $F'$. For the family $\Gamma(E',F';\Om')$ to be defined, we need 
$E',F'\sub \overline {\Om'}$. On the other hand, we do not know in advance 
which set  
$\Om'=\oC\setminus \bigcup_{i\in I_0} K_i$ will be, because this region depends on the choice of $I_0$. Hence we choose $E'$ and $F'$ as subsets of $V$, because this set, and hence also $E'$ and $F'$,   
 are  contained in {\em all} regions $\Om'$ that can possibly appear.

\medskip
\no
{\em Proof of Theorem~\ref{thm:simulunif}.}
Suppose that $\mathcal{S}=\{ S_i:i\in I\}$ is a collection  of  $s$-relatively separated  $k$-quasicircles 
bounding pairwise disjoint closed Jordan  regions  $D_i$. Note that by the remark following Lemma~\ref{lem:curvsepp} the regions $D_i$, $i\in I$, are also $s$-relatively separated.  
 
It is clear that the index set $I$ is at most countable.  So if $I$ is infinite, we may assume that $I=\N$.  
Let $T=\Sph\setminus\bigcup_{i\in I} \inte(D_i)$.
We first want to show there there exists  an $\eta$-quasi-M\"obius map 
$f\: T\ra T'$ of $T$  onto a Schottky set $T'$, i.e., 
$T'$ is the 
the complement of a collection of pairwise disjoint round 
open disks. Here we can choose $f$ so that its distortion function $\eta$ only depends on $s$ and $k$.

The set $T$ contains three distinct points that do not lie on any of the quasicircles $S_i=\partial D_i$, $i\in I$. This  follows from Lemma~\ref{lem:misc}
(iv). Note that if $I=\N$ then we can apply this lemma, since $\diam (D_i)\to 0$ as $i\to \infty$.  This was shown in the first part of the  proof of Proposition~\ref{prop:extend} and was derived from the fact that 
the sets  $D_i$ are $\lambda$-quasi-round with $\lambda=\lambda(k)\ge 1$
(Proposition~\ref{prop:qround}).

If we apply any M\"obius transformation
to  our collection $\mathcal{S}$, then the new collection will consist of $s'$-relatively separated $k'$-quasicircles, where 
$s'$ only depends on $s$ and $k'$ only depends on $k$ 
(Corollary~\ref{cor:qmobinvcond}).
In this way we may reduce ourselves to the case where  
$T$ contains  the points $0,1,\infty$ and  none of these points lies on any quasicircle $S_i$.

If $I$ is finite, then there exists a conformal map 
$f$ of the finitely connected region $U=\Sph\setminus\bigcup_{i\in I}D_i$ onto a circle domain  $V$
such that $f(0)=0$, $f(1)=1$, $f(\infty)=\infty$.
By Theorem~\ref{thm:unifqs} the map $f$ is $\eta$-quasisymmetric with $ \eta$ only depending on $s$ and $k$. The map $f$ extends uniquely to a homeomorphism of $\overline U=T$ onto the Schottky set 
 $T':=\overline V$ (this was pointed out in the proof of Theorem~\ref{thm:unifqs}), and the extended map $f$ 
 is    an 
$ \eta$-quasisymmetry on $\overline U=T$ (see the remark after the proof of
Theorem~\ref{thm:unifqs}). 

If $I$ is infinite, and so $I=\N$,  then 
for each $n\in \N$ we consider the finitely connected region 
$U_n=\Sph\setminus \bigcup_{i=1}^nD_i$. Then $0,1,\infty\in U_n$ for all $n\in \N$, and so  
again there exist an $\eta$-quasisymmetric map 
$f_n$ of   $\overline U_n$ onto the closure $\overline V_n$ of a circle domain $V_n$
such that $f_n(0)=0$, $f_n(1)=1$, $f_n(\infty)=\infty$.
Here $\eta$ depends only on $s$ and $k$, but not on $n$.
Note that $\bigcap_{n\in \N}\overline U_n=T$. 
Since the maps $f_n|T$, $n\in \N$, are normalized and $\eta$-quasisymmetric, the sequence $(f_n)$ subconverges to an $\eta$-quasisymmetric embedding
$f\:T\ra \Sph$, i.e., there exists a subsequence 
$(f_{n_k})$ of $(f_n)$ that converges to $f$ uniformly 
on $T$ (Lemma~\ref{lem:subcon}). 

We claim that   $f(T)$ is a Schottky set. To see this note that if $i\in \N$ is arbitrary, then $S_i$ is the boundary of the complementary component $D_i$ of 
$U_n$ for $n\ge i$. Hence $f_n(S_i)$ is a circle 
for $n\ge i$. Since $f_{n_k}\to f$ uniformly  on $T$, it follows that 
the Jordan  curve $f(S_i)$ is Hausdorff limit of a sequence of circles. Therefore, $f(S_i)$ must be a circle itself. By Lemma~\ref{lem:misc}  (iii), the circles  $f(S_i)$, $i\in I$,  bound
pairwise disjoint closed disks $D'_i$ such that $T'=f(T)=\Sph\setminus \bigcup_{i\in I} D'_i$. Hence  $T'$ is a Schottky set. 

So both when  $I$ is finite or infinite we showed that there exists an $\eta$-quasisymmetric map  of $T$ onto a Schottky set $T'$
where $\eta=\eta_{s,k}$.
Then 
 $f$ is  $\tilde \eta$-quasi-M\"obius with $\tilde \eta$
only depending on $\eta$ and hence only on $s$ and $k$ (Proposition~\ref{prop:inter}~(ii)). By Proposition~\ref{prop:extend}
we can extend   $f$  to an
$H$-quasiconformal homeomorphism on $\Sph$ with $H$ only depending on $\tilde \eta$ and $k$ and hence only on $s$ and $k$.  The map $f\:\Sph\ra \Sph$ is the desired quasiconformal map that sends each quasicircle $S_i$ to a round circle.
 \qed
 \medskip 
 
 The next example shows that we cannot omit the condition of uniform relative separation in Theorem~\ref{thm:simulunif}. 
 
 \begin{example}\label{ex:nonunif} For a set $M\sub \C$ and $a\in \C$, $b>0$, let $a+bM:=\{a+bz:z\in M\}$. Define $Q=[-1,1]\times  [-1,1]\sub \R^2\cong \C$ and
 $$Q_{2i-1}=2^{-i} +8^{-i}Q \text{ and }Q_{2i}=2^{-i}+(2+1/i)8^{-i}+8^{-i}Q$$
 for $i\in \N$. Then the sequence  $Q_i$, $i\in \N$, consists of pairwise 
 disjoint  squares lined up on the positive real axis. The main point is that for $i\to \infty$ the distance 
 $(1/i)8^{-i}$ 
 of $Q_{2i-1}$ and  $Q_{2i}$ goes to $0$ faster than the sidelength 
 $2\cdot 8^{-i}$
 of these squares. 
 
 The sets $S_i=\partial Q_i$, $i\in I$,  are uniform
 quasicircles. We claim that there is no quasiconformal  map $f\:\OC\ra \OC$ such that $f(S_i)$ is a round circle
 for each $i\in \N$. Indeed suppose there is such a map.
  By precomposing $f$ by a Euclidean similarity that maps $Q$ to  $Q_{2i-1}$ 
   and post-composing $f$ by a M\"obius transformation
  we obtain a sequence of $(f_i)$ of $H$-quasiconformal 
  maps on $\OC$ such that 
  $f_i(\partial Q)=\partial \D$, $f_i(1)=1$, $f_i(\iu)=\iu$, $f_i(-1)=-1$, 
  and such that $f_i(\partial D_i)$ is a circle where
  $D_i=(2+1/i)+Q $ for $i\in \N$. Here $H$ does not depend on $i$
  and so  the sequence is {\em uniformly quasiconformal}.  
  Hence  $(f_i)$ has a convergent subsequence that converges uniformly to a quasiconformal map $g$ on $\oC$ (a suitably normalized sequence of uniformly quasiconformal maps on $\oC$ 
  has subsequence that converges uniformly to 
   a quasiconformal map as a sublimit; see \cite[Sect.~19--Sect.~21 and Sect.~37]{Va}. The existence of the quasiconformal  sublimit $g$ can also easily be derived from Proposition~\ref{prop:inter} and Lemma~\ref{lem:subcon}).

  Then we have $g(\partial Q)=\partial \D$. Moreover, 
  since $\partial D_i\to 2+\partial Q$ in the Hausdorff sense and $f_i(D_i)$ is a circle for each $i\in \C$, the set $g(2+\partial Q)$ 
  is also a circle. Since the squares $Q$ and $2+Q$ share a common side,  and $g(\partial Q)=\partial \D$, we conclude 
 $g(2+\partial Q)=\partial \D$. This is impossible, since $g$ is a homeomorphism and so two distinct sets in $\OC$ cannot have the same image. 
 \end{example}

Finally, we give an  example showing that  in Theorem~\ref{thm:simulunif} one cannot drop the assumption that  the quasicircles bound pairwise disjoint Jordan  regions. 
 
\begin{example}\label{ex:nested}  We define a collection $S_i$, $i\in \N_0$ of uniformly relatively  separated   uniform quasicircles 
as follows. For $S_0$ we pick any quasicircle in $\oC$ that is not a round circle; to be specific, let $S_0=\partial Q$, where 
$Q=[-1,1]\times [-1,1]\sub \R^2\cong \C$. All other curves $S_i$, $i\in \N$, will be round circles that bound closed disks $D_i$ that 
are pairwise disjoint and disjoint from $S_0$. 
We can choose the circles $S_i$, $i\in \N$, such that the family 
$S_i$, $i\in \N_0$ (including $S_0$), is uniformly relatively separated and such that 
the set $T=\OC\setminus \bigcup_{i\in \N}\inte(D_i)$ has spherical measure zero.  One can obtain such disks $D_i$ and circles $S_i=\partial 
D_i$ by a    procedure that successively scoops out disks  from  the two complementary components of $S_0$. This is essentially identical to the construction in the proof of Thm.~1.3 in 
\cite[p.~435]{BKM}, so we omit the details.

 Now suppose that there was a quasiconformal
(and hence quasisymmetric) map $f\:\oC\ra \oC$ such that 
$f(S_i)$ is a round circle for each $i\in \N_0$. Then $f|T$ 
is a quasisymmetric map of the Schottky set $T$ onto the Schottky set 
$$T'=\OC \setminus \bigcup_{i\in \N}\inte(f(D_i)).$$
Since $T$ has spherical measure zero, the map $f|T$ is identical to the restriction of a M\"obius transformation \cite[Thm.~1.1]{BKM}. Since $S_0\sub T$ and $f(S_0)$ is a round circle, this implies
that $S_0$ must be a round circle itself.  This is a contradiction showing that a map $f$ as stipulated does not exist. 
\end{example}

\section{Extremal metrics for transboundary modulus}
\label{s:square} 

\no 
In this section we solve an extremal problem 
for  transboundary modulus and prove  a related variant of 
Theorem~\ref{thm:simulunif}.  
We employ the terminology 
for sets in the cylinder $\C^*\cong Z$ introduced in the beginning of Section~\ref{s:unifcyl}.
We denote by $d_{\C^*}$ the  flat metric on $\C^*$ induced by the length element 
$$ ds_{\C^*}=\frac{|dz|}{|z|},$$  and  by $A_{\C^*}$ the corresponding measure on ${\C^*}$ induced by the volume element
$$dA_{\C^*}(z)=\frac{dm_2(z)}{|z|^2}.$$ 
Here $m_2$ is $2$-dimensional Lebesgue measure. 
Note that if $Q$ is a $\C^*$-square with sidelength $\ell(Q)$, then 
$A_{\C^*}(Q)=\ell(Q)^2$.  
The length of a locally rectifiable path $\alpha$ in $\C^*$ with respect to the
 metric $d_{\C^*}$ is given  by
 $$\text{length}_{\C^*}(\alpha):=\int_\alpha \frac{|dz|}{|z|}.$$ 
  For  $z_0\in \C^*$ and $r>0$ we define 
$$B_{\C^*}(z_0,r)=\{z\in \C^*: d_{\C^*}(z_0,z)<r\}.$$  
Recall that the height $h_A$ of a finite $\C^*$-cylinder $A=\{z\in \C: r<|z|<R\}$ is given by $h_A=\log(R/r)$. 

To motivate our next result  we will discuss some background on extremal problems for classical modulus
 (for more details see \cite[Ch.~I]{LV}, for example).   
Suppose $Q\sub \C$ is a quadrilateral, i.e., a closed Jordan  region with four distinguished points 
on its boundary. These points divide $\partial Q$ into four arcs. Let $E$ and $F$ be two of these arcs that are ``opposite to each other" on $\partial Q$ (i.e., non-adjacent and separated by the other two arcs) and consider the path family $\Gamma=\Gamma(E,F;\inte(Q))$ of all paths 
in $\inte(Q)$ connecting $E$ and $F$.   It is well-known how to compute 
$\Mod (\Gamma)$ (at least in principle); namely, map $Q$ by a conformal map to a Euclidean rectangle such that the 
vertices of $Q$ and $R$ correspond to each other under the map.
If $E$ and $F$ corresponds to sides of $R$ with length $a$, and the other two arcs on $\partial Q$ to sides of $R$ with length $b$, then 
$$ \Mod (\Gamma)=a/b.$$ 
Moreover, if $f$ is the  conformal map of $Q$ onto $R$, then the 
 unique (up to changes on sets of measure zero)
extremal density $\rho$ for $\Mod (\Gamma)$ 
is $\rho(z)=|f'(z)|/b$ if we use the Euclidean metric as base metric on 
$Q$. 
This easily follows from conformal invariance of modulus, 
and the fact that if $Q=R$, then $\rho\equiv 1/b$ is the  extremal density.
The main point here is that   rectangles are ``extremal regions" for this type of modulus problem. 

To give another example, suppose $D_0,D_1\sub \oC$ are disjoint closed Jordan  regions. Consider the (topological) annulus $V=\oC\setminus(D_0\cup D_1)$, and the family 
$\Gamma=\Gamma(\partial D_0, \partial D_1; V)$ of all paths in $V$ connecting the boundary components $\partial D_0$ and $\partial D_1$  of the annulus. In this case  the extremal regions for 
$\Mod(\Gamma)$ are finite $\C^*$-cylinders. One can find a 
$\C^*$-cylinder 
$A$ that is conformally equivalent to $V$. Then
$$ \Mod(\Gamma)=\frac{2\pi}{h_A}.$$ Moreover, the essentially unique extremal density for $\Mod(\Gamma(\partial_iA, \partial_oA; A))$ is $\rho\equiv 1/h_A$ (with the flat metric on $\C^*$ as base metric), and using a  conformal map  between $V$ and $A$ one can easily identify the extremal 
density for  $\Mod(\Gamma)$. 

In this section we are interested in similar results for transboundary modulus.    Suppose  
 $D_0, \dots, D_{n+1}\sub \oC$ are pairwise disjoint closed Jordan
 regions, and $V=\oC\setminus(D_0\cup D_{n+1})$. We consider  the transboundary modulus  
 $M_{V, \mathcal{K}}(\Gamma)$, where $\mathcal{K}=\{D_1, \dots, D_{n}\}$ and $\Gamma$ is the family of all paths in $V$ connecting  $\partial D_0$ and $\partial D_{n+1}$. 
 So the paths have their endpoints on $\partial D_0$ and $\partial D_{n+1}$, but they do not meet $D_0$ and $D_{n+1}$ otherwise, and  they may pass through all
 the other Jordan  regions  $D_1, \dots, D_{n}\sub V$.  For the transboundary mass distribution we may put weights on the elements in  
 $\mathcal{K}$, i.e., on $D_1, \dots , D_{n}$, but not on $D_0$ and $D_{n+1}$. 
The density of the transboundary mass distribution  will be defined 
 on 
 \begin{equation}\label{regOm}
 \Om=V\setminus (D_1\cup \dots \cup D_n)=\oC\setminus (D_0\cup \dots \cup D_{n+1}). 
 \end{equation} 
 As we will see (cf.\ the discussion after  the proof of 
Proposition~\ref{prop:transcyl})  the extremal region (corresponding to $\Om)$ is of the form 
\begin{equation}\label{extrreg}
U= A\setminus (Q_1\cup \dots \cup Q_{n}),
\end{equation} where $A$ is a finite 
$\C^*$-cylinder and $Q_1, \dots,  Q_{n}$ are pairwise disjoint $\C^*$-squares in $A$.   

We first require a lemma.

\begin{lemma}\label{lem:LLCsq} Let $A$ be a finite $\C^*$-cylinder, and $Q_1, \dots, Q_n$ be   pairwise 
disjoint $\C^*$-squares in $ A$. Then any two points
$$x,y\in T:=\overline A\setminus\bigcup_{i=1}^n\inte(Q_i)$$
can be joined by a rectifiable path $\beta$ in $T$ with
$$\textnormal{length}_{\C^*}(\beta)\le 2d_{\C^*}(x,y). $$

Suppose in addition that 
$$\ell(Q_i)\le 2\pi-\eps_0 \forr i=1,\dots, n,$$
where $0<\eps_0<2\pi. $
Then for all $z_0\in T$ and  all $r_0\in (0,\eps_0/2]$
 the set
$T\setminus B_{\C^*}(z_0,r_0)$ is path-connected. 
\end{lemma}
\medskip
\no{\em Proof.} Let $x,y\in T$ be arbitrary. We connect $x$ and $y$ by a geodesic segment $\alpha$ in $\overline A$ with respect to the metric $d_{\C^*}$. If $\alpha_i:=\alpha \cap \inte(Q_i)\ne \emptyset$, then an elementary geometric argument shows that one of the two subarcs
of $\partial Q_i$ with the same endpoints as $\alpha_i$ 
has length bounded by $ 2 \text{length}_{\C^*}  (\alpha_i)$. Denote this arc by 
$\tilde \alpha_i$. If we replace each $\alpha_i\ne \emptyset$
by $\tilde\alpha_i$, then we obtain a path $\beta$ in $T$  
with  endpoints $x$ and $y$, and with $\length_{\C^*}(\beta)\le 
 2\length_{\C^*}(\alpha)= 2 d_{\C^*}(x,y).$ 
 
 For the  the proof  of the second part of the statement we may assume 
 that $z_0$ lies on the positive real axis. Since $r_0< \pi$,  the disk $B_0=B_{\C^*}(z_0,r_0)$ does not contain any point on the negative real axis. Then we can connect 
 $x$ and $y$ by a path $\alpha$ in $\overline A\setminus B_0$ as follows: there is an arc $\alpha_1$ on the 
 circle $\{z\in \C: |z|=|x|\}$ that does not meet $B_0$ 
and connects $x$ to a point $x_1$ on the negative real axis.
Similarly, $y$ can be connected by an arc $\alpha_2$ on the 
 circle $\{z\in \C: |z|=|y|\}$ that does not meet $B_0$ to a point $y_1$ on the negative real axis. Finally, 
let $\alpha_3$ be a segment on the negative real axis 
with endpoints $x_1$ and $y_1$. Then running through the paths
$\alpha_1, \alpha_2, \alpha_3$ in suitable order gives  the desired path in $\overline A\setminus B_0$ that connects $x$ and $y$. Similarly  as in the first part of the proof  we want to modify $\alpha$ to obtain a path $\beta$ in $T\setminus B_0$ that connects $x$ and $y$.

We leave it to the reader to verify the following elementary fact of Euclidean geometry: if $Q$ is a  square  and $B=B_\C(z,r)$ a disk in $\C$, and 
 $z\not\in \inte(Q)$, then the set $\partial Q\setminus B$ is 
 connected. A similar statement is true for $\C^*$-squares $Q$ and disks $B=B_{\C^*}(z,r)$ 
 in $\C^*$ with $z\not\in \inte(Q)$ if we impose the additional condition
 $\ell(Q)+2r\le 2\pi$. Indeed, this follows from the Euclidean 
 fact if we lift by the exponential function to $\C$ and note that the condition $\ell(Q)+2r\le 2\pi$ implies that 
 every lift of $Q$ can meet at most one lift of $B$.

 The center $z_0$ of $B_0$ is contained in $T$ and hence lies outside the interior of each $\C^*$-square $Q_i$ for  $i=1, \dots, n$.  Moreover, we have $\ell(Q_i)+2r_0\le(2\pi-\eps_0)+\eps_0=2\pi$. It
 follows that $\partial Q_i\setminus B_0$ is connected for all
 $i=1, \dots, n$.
 
 We use this fact to modify the path $\alpha$ obtained above as follows: if $\alpha$ meets $\inte(Q_1)$, then 
 there is a first point point $x'$ and a last point $y'$ 
 on $\partial Q_1$ as we travel from $x$ to $y$ along $\alpha$. Since $x'$ and $y'$ lie on $\alpha$ and hence outside $B_0$,
 we can connect these  points on $\partial Q_1$ by a path $\tilde \alpha$ in 
 $\partial Q_1\setminus B_0$. If we replace the subpath of  $\alpha$ between $x'$ and $y'$ by 
 $\tilde \alpha$, we obtain a new path connecting 
 $x$ and $y$ in $\overline A\setminus (B_0\cup \inte(Q_1))$.
 Continuing this procedure with the other $\C^*$-squares we finally obtain a path 
 $\beta$ connecting $x$ and $y$ in 
 $$\overline A\setminus (B_0\cup \inte(Q_1) \cup \dots \cup \inte(Q_n))=T\setminus B_0. $$ The proof is complete.   \qed
\medskip

\begin{proposition}\label{prop:transcyl} Let $A$ be a finite 
$\C^*$-cylinder, and $\mathcal{K}=\{Q_i:i=1, \dots,n\}$ be a finite (possibly empty) family of pairwise disjoint (possibly degenerate) $\C^*$-squares in $A$. Define $\Gamma=
\Ga(\partial_iA,\partial_oA; A)$  and $K=Q_1\cup\dots\cup Q_n$.
Then
$$ M_{A,\mathcal{K}}(\Gamma)=\frac{2\pi}{h_A}. $$
Moreover, the essentially unique extremal  admissible transboundary mass distribution for $\Gamma$ consisting of a Borel function $\rho$ on $A\setminus K$,
and discrete weights $\rho_i\ge 0$ for $i=1,\dots,n$ such that 
\begin{equation} \label{totmass}\int_{A\setminus  K} \rho^2\, dA_{\C^{*}}+
\sum_{i=1} ^n \rho_i^2=M_{A,\mathcal{K}}(\Gamma)= \frac{2\pi}{h_A}, \end{equation}
is given by  $\rho(z)= 1/h_A$ for  $z\in A\setminus K$, and 
$\rho_i=\ell(Q_i)/h_A$ for $i=1, \dots, n$. 
\end{proposition}

The underlying base  metric here (see Remark~\ref{rem:chbase}) is the flat metric on $\C^*$. 
Essential  uniqueness   means that if we have another admissible transboundary mass distribution for $\Gamma$ with \eqref{totmass}, then 
$\rho(z)= 1/h_A$ for {\em almost every } $z\in A\setminus K$, and 
$\rho_i=\ell(Q_i)/h_A$ for $i=1, \dots, n$. 

\medskip 
\no{\em Proof}. Suppose that $A=\{z\in \C: r<|z|<R\}$, where
$0<r<R$. Then $h_A=\log(R/r)$. 
Let $\rho(z)=1/h_A$ for $z\in A\setminus K$ and 
$\rho_i=\ell(Q_i)/h_A$ for $i=1, \dots, n$.
We claim that this transboundary mass distribution is admissible
for the family $\Gamma$. Let $\ga \in \Gamma$ be an arbitrary 
path that is locally rectifiable in $A\setminus K$. We may assume that 
$\ga$ is parametrized by the interval $[0,1]$ and that 
$\ga(0)\in \partial 
A_i$ and   $\ga(1)\in \partial 
A_o$.  By definition of  $\Gamma$ we have $\ga((0,1))\sub A$.

 Let $\pi\:\overline {A}\ra  [\log r, \log R] \ $ be the map 
  $z\mapsto \pi(z):=\log|z| $.
 Then 
 \begin{equation}\label{eq:proiect2}
 (\log r, \log R)\sub  \pi(\ga  \cap (A\setminus K) ) \cup \bigcup_{\ga \cap Q_i\ne \emptyset}\pi(\ga \cap Q_i).
 \end{equation}
Note that   \eqref{proiest2}   implies that 
$$\int_{\ga \cap (A\setminus K)}\rho\, ds_{\C^{*}} =
 \frac 1 {h_A}\int_{\ga \cap (A\setminus K)}
\frac{|dz|}{|z|}\ge  \frac 1 {h_A} m_1(\pi(\ga \cap (A\setminus K))).$$
We also have 
$\rho_i=  \frac 1 {h_A} m_1(\pi(Q_i))$ for $i=1,\dots, n$, and so  
 by \eqref{eq:proiect2} we obtain  
\begin{eqnarray*}
 \int_{\ga \cap (A\setminus K)} \rho\, ds_{\C^*}+\sum_{\alpha\cap Q_i \ne \emptyset} \rho_i&\ge & \frac 1 {h_A} m_1(\pi(\ga  \cap (A\setminus K) ))
+ \frac 1 {h_A}  \sum_{\alpha \cap Q_i\ne \emptyset} m_1(Q_i)\\
&\ge& \frac 1 {h_A}   m_1 ((\log r, \log R))=1.
\end{eqnarray*}
The admissibility of our transboundary mass distribution follows.

We conclude that  
\begin{eqnarray*}
M_{A,\mathcal{K}}(\Ga)&\le &
\int_{A\setminus K} \rho^2\,dA_{\C^*}+\sum_{i=1}^n \rho_i^2\\
&=&\frac1{h_A^2}A_{\C^*}(A\setminus K)+\frac1{h_A^2}
\sum_{i=1}^n A_{\C^*}(Q_i)\\
&=&\frac1{h_A^2}A_{\C^*}(A)=\frac{2\pi}{h_A}.
\end{eqnarray*}

To get an inequality in the other direction,  suppose that we have an  admissible  transboundary mass distribution 
for the family $\Gamma$ 
consisting of a density  $\rho$ on $A\setminus K$,
and discrete weights $\rho_i\ge 0$ for $i=1,\dots,n$.
For each $\varphi\in [0,\pi]$ the path $\alpha_{\varphi}\:[\log r, \log R]\ra \overline A $ defined by    $\alpha_{\varphi}(t):=
t e^{\iu\varphi}$ for $t\in[\log r, \log R]$  belongs to $\Ga$. 
Hence for each $\varphi\in [0,2\pi]$ we have
$$\int_{\alpha_\varphi\cap (A\setminus K)} \rho\, ds_{\C^{*}}
+\sum_{\alpha_\varphi \cap Q_i \ne \emptyset} \rho_i\ge 1. $$
Integrating this over $\varphi$, using Fubini's theorem, and the Cauchy-Schwarz inequality, we arrive at
\begin{eqnarray}
2\pi &\le &
\int_{A\setminus  K} \rho\, dA_{\C^{*}}+
\sum_{i=1} ^n \ell(Q_i)\rho_i  \label{CS1}\\
&\le & A_{\C^{*}}(A\setminus  K)^{1/2} \biggl(\int_{A\setminus  K} \rho^2\, dA_{\C^{*}}\biggr)^{1/2}+\sum_{i=1} ^n \ell(Q_i)\rho_i  \nonumber\\
&\le & \biggl(A_{\C^{*}}(A\setminus  K)+\sum_{i=1} ^n\ell(Q_i)^2\biggr)^{1/2}
\biggl(\int_{A\setminus  K} \rho^2\, dA_{\C^{*}}+
\sum_{i=1} ^n \rho_i^2\biggr)^{1/2}\label{CS2}\\
&=&(2\pi h_A)^{1/2}\biggl(\int_{A\setminus  K} \rho^2\, dA_{\C^{*}}+
\sum_{i=1} ^n \rho_i^2\biggr)^{1/2}.\nonumber
\end{eqnarray}
Hence 
$$\int_{A\setminus  K} \rho^2\, dA_{\C^{*}}+
\sum_{i=1} ^n \rho_i^2\ge \frac{2\pi}{h_A} $$
for every transboundary mass distribution that is admissible for $\Ga$.
This implies the other desired inequality
$M_{A,\mathcal{K}}(\Ga)\ge {2\pi}/{h_A}. $

If we have an admissible transboundary mass distribution satisfying \eqref{totmass}, then we must have equality in \eqref{CS1} and \eqref{CS2}. Equality in \eqref{CS2}
implies  that there 
exists $\lambda>0$ such that $\int_{A\setminus  K} \rho^2\, dA_{\C^{*}}=\la^2 A_{\C^*}({A\setminus  K}) $ and $\rho_i=\la \ell(Q_i)$ for 
$i=1, \dots, n$.  Hence $\la =1/h_A$ by \eqref{totmass}, and  so  $\int_{A\setminus  K} \rho^2\, dA_{\C^{*}}=A_{\C^*}({A\setminus  K})/h_A^2 $ and $\rho_i= \ell(Q_i)/h_A$ for 
$i=1, \dots, n$.  This and equality in \eqref{CS1} give 
$$  \int_{A\setminus  K} \rho^2\, dA_{\C^{*}}=\frac{1}{h_A^2}A_{\C^*}({A\setminus  K}) =\frac{1}{h_A}\int_{A\setminus  K} \rho\, dA_{\C^{*}}, $$
and so $\rho=1/h_A$ almost everywhere on $A\setminus K$.
\qed \medskip

A general criterion for a density to be  extremal for  the modulus of a given path family is due to Beurling (see \cite[Thm.~4.4, p.~61]{Ah2}). It is easy to extend this 
condition to a criterion for the extremality of a transboundary mass distribution. 
Based on this one   can give a  proof of 
Proposition~\ref{prop:transcyl} that is slighly more streamlined (but uses essentially the same ideas). 

 Combining Proposition~\ref{prop:transcyl} with Corollary~\ref{thm:cylunif} and invariance of transboundary modulus, one  can immediately give a solution 
to the problem discussed in the beginning of this section. If  the setup is 
as before Lemma~\ref{lem:LLCsq}, then we map the region $\Om$
in \eqref{regOm} to a region of the form $U$ as in \eqref{extrreg} by a conformal map $f$. The map $f$ has a unique extension to a homeomorphism from $\overline\Omega $ onto $\overline U$, and a further (non-unique) extension as a homeomorphism on $\oC$. We assume 
 that $f(\partial D_0)=\partial_iA$, $f(\partial D_{n+1})=\partial_oA$, and that the labeling of the other complementary components is such that 
 $f(D_i)=Q_i$ for $i=1, \dots, n$.  
 Then $f(V)=A$  and $f(\Gamma)=\Gamma(\partial_iA, \partial_oA; A)$.   
 Hence 
 $$ M_{V,\mathcal{K}}(\Gamma)=\frac {2\pi}{h_A}. $$
 Moreover, based on the last part of Proposition~\ref{prop:transcyl} one can easily identify the essentially unique  extremal 
 transboundary mass distribution for $M_{V,\mathcal{K}}(\Gamma)$ (we leave this to the reader). 

We record another application of Proposition~\ref{prop:transcyl}. 

\begin{corollary} \label{cor:uniqsquare} The map $f$  in Corollary~\ref{thm:cylunif} is unique up to a post-composition by a map of the form $z\mapsto az$, $a\in \C^{*}$. 
\end{corollary}

\no {\em Proof.} Let $n\in \N_0$,  $A$ and $A'$ be finite cylinders, 
$Q_1, \dots, Q_n$ pairwise disjoint $\C^*$-squares in $A$, and 
$Q'_1, \dots, Q'_n$ pairwise disjoint $\C^*$-squares in $A'$. Let $U=A\setminus (Q_1\cup \dots \cup Q_n)$, $V=A'\setminus (Q_1\cup \dots \cup Q_n)$, and suppose that $g\: \overline U\ra  \overline V$ is a homeomorphism that 
is a conformal map on $U$ with $g(U)=V$, and satisfies $g(\partial_i A)=\partial_i A'$
and $f(\partial_o A)=\partial_o A'$.  
It suffices to show that there exists $a\in \C^*$ such that 
$g(z)=az$ for all $z\in U$.   We extend $g$ (non-uniquely) 
to a homeomorphism from $\overline A$ onto $\overline A'$, which we  also denote by $g$.

Let $\Gamma=\Gamma(\partial_i A,\partial_oA;A)$, and 
$\Gamma'=\Gamma(\partial_i A',\partial_oA';A')$.
Then $\Gamma'=g(\Gamma)$. By invariance of transboundary 
modulus and Proposition~\ref{prop:transcyl} we have
$$\frac{2\pi}{h_A}= M_{A,\mathcal{K}}(\Gamma)=
M_{A',\mathcal{K'}}(\Gamma')=\frac{2\pi}{h_{A'}}, $$
where $\mathcal{K}=\{Q_1, \dots, Q_n\}$ and 
$\mathcal{K'}=\{Q'_1, \dots, Q'_n\}$. 

As we have seen in the proof of Proposition~\ref{prop:transcyl}, 
the transboundary mass distribution consisting of the density $\rho'=1/h_{A'}$ on $V$ and the weights $\rho'_i=\ell(Q'_i)/h_{A'}$ is admissible  for the modulus $M_{A',\mathcal{K'}}(\Gamma')$ and has  minimal total mass. As in the proof of Lemma~\ref{lem:invtrans}  (using the flat metric $d_{\C^*}$ on $\C^*$ instead of the spherical metric) one sees  that  
the transboundary  mass distribution consisting of the density  
$$\rho(z)=\frac{|zg'(z)|}{h_{A'}|g(z)|}\forr z\in U$$ and the weights $\rho_i=\ell(Q'_i)/h_{A'}$ is  admissible  for the modulus $M_{A,\mathcal{K}}(\Gamma)$. Since  
$M_{A,\mathcal{K}}(\Gamma)=
M_{A',\mathcal{K'}}(\Gamma')$ by invariance of transboundary modulus,  this implies that that this transboundary mass distribution
is also extremal   for the modulus  $M_{A,\mathcal{K}}(\Gamma)$. The uniqueness statement in Proposition~\ref{prop:transcyl}  implies that  $z\mapsto |zg'(z)|/|g(z)|$ is a constant function on $U$. Hence the function $z\mapsto zg'(z)/g(z)$ is also constant on $U$, say 
$z g'(z)/g(z) \equiv c$ on $U$, where $c\in \C$. Suppose that 
$\partial_iA=\{z\in \C: |z|=r\}$, where $r>0$. Since $g$ maps the circle 
$\partial_iA$ to the circle $\partial_i A'$, the map $g$ has an analytic extension to a neighborhood of  $\partial_iA$ by the Schwarz reflection principle  and it follows that 
$ z g'(z)/g(z) = c$ for $z\in\partial_iA$.
Let $\alpha$ be the path $t\in[0,2\pi]\mapsto  \alpha(t):=re^{\iu t}$.
Then we have 
$$ \frac 1{2\pi \iu} \int_{g\circ \alpha} \frac{dw}{w}= \frac 1{2\pi \iu} \int_\alpha \frac{g'(z)}{g(z)}\, dz =   \frac c{2\pi \iu} 
\int_\alpha  \frac{dz }{z}=c. $$
On the other hand, the expression of the left-hand side represents the winding number of the path 
 $g\circ \alpha$ around $0$. Note  that $g\circ \alpha$ is a  parametrization of the circle $\partial_iA'$, the map  $g|\partial_iA$ is injective,  and $0$ lies ``on the left" of the oriented 
 path $g\circ \alpha$ since 
$g$ is orientation-preserving.  Thus,  this winding number is equal to $1$ and so  $c=1$. This implies that 
 the function $z \mapsto g(z)/z$ has vanishing derivative on $U$, and so  there exists a constant $a\in \C^*$ 
with 
$g(z)=az$ for $z\in U$ as desired. 
\qed  \medskip

\begin{lemma}\label{lem:cylqs} In Corollary~\ref{thm:cylunif} suppose in addition that 
the Jordan  curves $\partial D_0, \dots, \partial D_n$  are 
$s$-relatively separated $k$-quasi\-circles, and that 
$$\diam(\partial D_0)\wedge  \diam(\partial D_n)\ge d>0. $$

Then there exist  constants $C_1=C_1(s,k)>0$, $C_2=C_2(s,k,d)>0$, and $\eps_0=\eps_0(s,k,d)>0$ such that
\begin{equation}\label{hbdd}
C_1\le h_A\le C_2,
\end{equation}
 and 
 \begin{equation}\label{lQbdd}
\ell(Q_i)\le 2\pi-\eps_0 \foral i=1,\dots, n-1. 
\end{equation}
\end{lemma}

\no{\em Proof.} Let $V=\OC\setminus(D_0\cup D_n)$, 
 ${\mathcal K}=\{D_i:i=1, \dots, n-1\}$, and 
 $\Ga=\Ga(\partial D_0,\partial D_n; V)$.  We can extend the map $f$ in  Corollary~\ref{thm:cylunif} (non-uniquely) 
 to a homeomorphism from  $\overline V$ onto $\overline A$.   By the properties of the map $f$ we then have 
 $$f(\Ga)=\Ga(\partial_iA,  \partial_oA; A).$$
 Hence by invariance of transboundary modulus and 
 Proposition~\ref{prop:transcyl} we get 
 $$M_{V,{\mathcal K}}(\Ga)=2\pi /h_A. $$
 This shows that in order to establish  inequality \eqref{hbdd}, it suffices to show that
  $M_{V,{\mathcal K}}(\Ga)$ is bounded below by a positive constant only depending on $s$, $k$ and $d$, and bounded above a constant  only depending on $s$ and $k$. 

To produce the first bound note that   by  Lemma~\ref{lem:curvsepp}
the regions $D_0, \dots, D_n$ are also $s$-relatively separated.     
So by Proposition~\ref{thm:lowmod} the region $V=\OC\setminus (D_0\cup D_n)$ is $\phi$-Loewner, where $\phi=\phi_{s,k}$. Moreover, for the continua $\partial D_0$ and $\partial D_n$ we have 
$$ \Delta(\partial D_0,\partial D_n)\le 2/d. $$

Since the continua in ${\mathcal K}$ are 
$s$-relatively separated, and also $\la$-quasi-round with $\la=\la(k)$ by 
Proposition~\ref{prop:qround}, it follows from Proposition~\ref{prop:lowtrans} that 
$M_{V,{\mathcal K}}(\Ga)\ge C(s,k,d)>0$
as desired.  

To produce an inequality in the opposite direction,
note that 
$$ \Delta(\partial D_0,\partial D_n)\ge  s,$$
since $\partial D_0$ and $\partial D_n$ are 
$s$-relatively separated.  Hence by Proposition~\ref{thm:weakuptrans},
$$M_{V,{\mathcal K}}(\Ga)\le C(s,k). $$
The first part of the theorem  follows. 

To prove the second part of the proposition consider one of the 
$\C^*$-squares $Q_1,\dots, Q_{n-1}$, say $Q_1$. Under the map $f$ it corresponds to one of the Jordan  regions $D_1, \dots, D_{n-1}$,
say to $D_1$. Let $V'=\OC\setminus (D_0\cup D_1\cup
D_n)$. Then again by Proposition~\ref{thm:lowmod} the region $V'$ is $\phi$-Loewner with $\phi=\phi_{s,k}$. 
We can again invoke Proposition~\ref{prop:lowtrans} and 
the invariance of transboundary modulus to conclude that
$$M_{U, \mathcal{L}}(\Ga(\partial_iA,\partial_oA;U))=
M_{V',\mathcal{K}'}(\Ga(\partial D_0,\partial D_n;V'))\ge C(s,k,d)>0.$$
Here $U=A\setminus Q_1$,  $\mathcal{L}=\{Q_2,\dots, 
Q_{n-1}\}$, and $\mathcal{K}'=\{D_2, \dots, D_{n-1}\}$. 

On the other hand, suppose that $A=\{z\in C: r<|z|<R\}$.
Without loss of generality we may assume that 
$$Q_1=\{se^{\iu t}: r'\le s\le R', t\in [\alpha, 2\pi-\alpha]\},$$ where
$r<r'<R'<R$ and $\alpha\in (0,\pi)$.  Then  $\ell(Q_1)=2(\pi-\alpha)
=\log(R'/r')$. We have to show that $\alpha$ cannot be smaller than a positive constant only depending on $s$, $k$, 
and $d$. 

Note that every path $\ga\in \Ga=\Gamma(\partial_iA, \partial A_o; U)$ lies in the complement of $Q_1$ and meets both circles $\{z\in \C: |z|=r'\} $ and $\{z\in \C: |z|=R'\}$. Hence 
$\ga$  passes  through the channel
$$M=\{ se^{\iu t}: r'< s< R', t\in (-\alpha, \alpha)\} $$
meeting  ``bottom" and  ``top".  We use this fact to produce 
a transboundary mass distribution for $M_{U, \mathcal{L}}(\Ga(\partial_iA,\partial_oA;U))$ that has small mass if $\alpha$ is small.

 We use the flat metric on $\C^*$ as base metric and set  
$$\rho(u)=1/\ell(Q_1) \forr u\in M\cap U',$$ 
and $\rho=0$ elsewhere on $U'$, where 
$$ U'=U\setminus (Q_2\cup \dots \cup Q_{n-1})=A\setminus (Q_1 \cup \dots \cup Q_{n-1}). $$ 
Moreover, for $i\in \{2, \dots, n-1\}$  we set 
$$\rho_i=\ell(Q_i)/\ell(Q_1) \quad \text{if}\quad Q_i\cap M\ne \emptyset$$ and $\rho_i=0$ otherwise. By considerations very similar to the ones in the proof of Proposition~\ref{prop:transcyl} one can show that this transboundary mass distribution is admissible for  $\Ga$. 

A $\C^*$-square $Q$ that meets $M$ and is disjoint from 
$Q_1$ must satisfy $\ell(Q)<2\alpha$. This   implies 
$$Q\sub \widetilde M:=\{se^{\iu t}: r'e^{-2\alpha}<s<R'e^{2\alpha},
-\alpha<t<\alpha. \}$$
Hence 
\begin{eqnarray*} 
\int_{U'} \rho^2\, dA_{\C^*}+\sum_{i=2}^{n-1} \rho_i^2&\le &
\frac 1{\ell(Q_1)^2} \biggl(A_{\C^*}(M\cap U')+
\sum_{Q_i\cap M\ne \emptyset} A_{\C^*}(Q_i)\biggr)\\
&\le &\frac 1{\ell(Q_1)^2}A_{\C^*}(\widetilde M)=
\frac {\alpha(\pi+\alpha)}{(\pi-\alpha)^2}, 
\end{eqnarray*}
and so 
$$0<C(s,k,d)\le M_{U,\mathcal{L}}(\Ga)\le \frac {\alpha(\pi+\alpha)}{(\pi-\alpha)^2}
. $$
This shows that $\alpha\ge c(s,k,d)>0$ as desired.
\qed \medskip

\begin{proposition} \label{prop:uptrans2} There exists a number $N\in \N$,  and  a function
 $\psi\:[0,\infty)
\ra (0,\infty)$  with $$\lim_{t\to \infty}\psi(t)=0$$ satisfying   the following property:  if  
$\mathcal{K}=\{Q_i:i\in I\}$ is a collection of  pairwise disjoint $\C^*$-squares $Q_i\sub \C^*$, and   
if  $E$ and $F$ are arbitrary disjoint  continua
in $\C^*\setminus   \bigcup_{i\in I}\inte(Q_i)$ with $\Delta_{\C^*}(E,F)\ge 12$, then 
there exists a set $I_0\sub I$ with $\#I_0\le N$ 
such that for the transboundary modulus  of the path family $\Gamma(E,F;\Om')$
in the region $\Omega'=\C^*\setminus \bigcup_{i\in I_0}
Q_i$ with respect to the collection $\mathcal{K}'=\{Q_i:i\in I\setminus I_0\}$ we have
$$M_{\Om', \mathcal{K}'}(\Gamma(E,F;\Om'))\le \psi(\Delta_{\C^*}(E,F)). $$ 
\end{proposition} 

Here $\Delta_{\C^*}(E,F)$ denotes (in accordance with our convention from Section~\ref{s:nt}) the relative distance of $E$ and $F$ with respect to the flat metric $d_{\C^*}$ on $\C^*$.  
Note that  if $E$ and $F$ are as in the statement, then 
$$E,F\sub \C^*\setminus   \bigcup_{i\in I}\inte(Q_i)\sub
\C^*\setminus   \bigcup_{i\in I_0}\inte(Q_i)=\overline \Om'.$$

\begin{proof}  The proposition immediately follows from   Remark~\ref{rem:uptrans}.
We have to check the relevant conditions   in this remark.  For the mass  bounds in the metric measure space 
 $(\C^*, d_{\C^*}, A_{\C^*})$ note that if 
 $a\in \C^*$, then  
we have 
$$A_{\C^*}(B_{\C^*}(a,r))\le \pi r^2$$ for all $r>0$,  and 
$$ A_{\C^*}(B_{\C^*}(a,r))= \pi r^2$$ for all $r\le  \pi$. The last equality  implies that 
$$ A_{\C^*}(B_{\C^*}(a,r))\ge \frac{\pi}{5} r^2$$  for all 
$r\le \sup\{ \diam_{\C^*}(Q): Q \text { is a $\C^*$-square}\}= \pi \sqrt 5  $. So we get the relevant upper and lower mass bounds.

Moreover, it is clear that  a $\C^*$-square  $Q$ in $(\C^*, d_{\C^*}, A_{\C^*})$ is  $\mu$-fat for some  universal constant $\mu>0$. 
To produce an explicit (non-sharp) constant $\mu$ let $x\in Q$
and $0<r\le \diam_{\C^*}(Q)\le \sqrt 2 \ell(Q)$ be arbitrary.  If $0\le s\le \ell(Q)/2$, then $Q\cap B_{\C^*}(x,s)$ contains at least a ``quarter" of the disk   $ B_{\C^*}(x,s)$. 
If  we apply this for 
$s=r/(2\sqrt 2)\le \ell(Q)/2\le  \pi$, we obtain 
\begin{eqnarray*} A_{\C^*}(Q\cap  B_{\C^*}(x,r))&\ge&
 A_{\C^*}(Q\cap  B_{\C^*}(x,s))\ge  \frac 14  A_{\C^*}( B_{\C^*}(x,s))\\ &=&\frac \pi 4s^2=
\frac{\pi}{32}r^2\ge \frac 1{32}A_{\C^*}(  B_{\C^*}(x,r)) . 
\end{eqnarray*}
So we can take $\mu=1/32$. 
\end{proof}

\begin{proposition}\label{thm:cylqs} In Corollary~\ref{thm:cylunif} suppose in addition that 
the Jordan  curves $\partial D_0, \dots, \partial D_n$  are 
$s$-relatively separated $k$-quasi\-circles, and that 
$$\diam(\partial D_0)\wedge  \diam(\partial D_n)\ge d>0. $$

Then  $f$ is an $\eta$-quasisymmetric map from $\overline \Om$ equipped with the chordal  metric to $\overline U$ 
equipped with flat metric on $\C^*$. Here $\eta$ only depends
 on $s$,  $k$, and   $d$. 
  \end{proposition}

\no{\em Proof.}
The proof of the  theorem is very similar to the proof of Theorem~\ref{thm:unifqs}. Note that both metric spaces 
$(\oC, \sig)$ and $(\C^*, d_{\C^*})$ are  doubling, and so 
 every subset of one of these spaces is  $N_0$-doubling, where $N_0$ is a universal constant.  So by Proposition~\ref{prop:wkqs} it is enough 
 to show that on $\overline \Om$ the map $f$ is weakly $H$-quasi\-symmetric with $H=H(s,k,d)$. We can extend $f$ (non-uniquely) to a homeomorphism from $\oC\setminus (\inte(D_0)\cup \inte(D_n))$ onto 
$\overline A$. 
We will use the notation $u':=f(u)$ for $u\in \oC\setminus (\inte(D_0)\cup \inte(D_n))$. 

To reach a contradiction assume that for some points 
 $x,y,z\in \overline \Om$  with $\sig(x,y)\le \sig(x,z)$
 we have 
 $d_{\C^*}(x',y')> H d_{\C^*}(x',z')$ for some large $H>>1$. 
 Then the points $x,y,z$ are distinct. 
 We want to find continua $E'$ and $F'$ in 
 $$\overline U=\overline A\setminus (\inte(Q_1)\cup \dots \cup \inte(Q_{n-1}))$$ whose relative distance is large, but for which the relative distance of the preimages $E$ and $F$ is controlled.

 By Lemma~\ref{lem:LLCsq} we can find a continuum $E'\sub \overline U$ connecting $x'$ and $z'$ such that 
 $$\diam_{\C^*}(E')\le 2d_{\C^*}(x',z').$$ 

 The choice of $F'$ is more involved.  
Since the sets $\partial D_0$ and $\partial D_n$ are $s$-relatively separated, we have  
$$ \dist(\partial D_0,\partial D_n)\ge s(\diam(\partial D_0)\wedge \diam(\partial D_n))\ge sd. $$
Hence $y$ must have distance $\ge sd/2$ to one of the sets 
$\partial D_0$ and $\partial D_n$, say to $\partial D_0$.
Then 
$$ \dist(y, \partial D_0)\ge sd/2. $$
Let $\eps_0=\eps_0(s,k,d)\in (0,2\pi)$ be as 
\eqref{lQbdd}.  Then 
$$ \diam_{\C^*}( B_{\C^*}(x, \eps_0/4))\le \eps_0/2<\pi=\diam_{\C^*}(\partial_i A). $$
Hence there exists a point $u\in \partial D_0$ such that 
for its image point we have 
$u'\in \partial_i A\setminus B_{\C^*}(x,\eps_0/4)$. 
Note that then 
\begin{equation}\label{bbbsss}
\sigma(u,y)\ge sd/2.
\end{equation}

By \eqref{hbdd}  we  have 
$$d_{\C^*}(x',y')\le \diam_{\C^*}(\overline A)\le (\pi+h_A)\le(\pi+ C_2(s,k,d))=:C_3(s,k,d).$$
If we define $c_4:=c_4(s,k,d)=\eps_0/(4C_3)<1$, then
$c_4d_{\C^*}(x',y'))\le \eps_0/4$, and so   both points $u'$ and $y'$ lie outside the 
ball $B_{\C^*}(x', c_4d_{\C^*}(x',y'))$. 
By Lemma~\ref{lem:LLCsq} 
can find a continuum $F'\sub \overline U$
 connecting $y'$ and $u'$ such that 
 $$F'\cap B_{\C^*}(x', c_4d_{\C^*}(x',y'))=\emptyset.$$

Combining this with the diameter bound for $E'$, we see  (as in the proof of Theorem~\ref{thm:unifqs})  that
 if $H\ge C_{5}(s,k,d)$, then for the relative distance of $E'$ and $F'$ with respect to the metric $d_{\C^*}$ we have 
  $$ \Delta_{\C^*}(E',F')\ge H/C_{6}\ge 12,$$ 
  where $C_{6}=C_{6}(s,k,d)$. 
  
  Define  $E=f^{-1}(E')$ and $F=f^{-1}(F')$. Then $E$  and $F$ are  continua in $\overline \Om$ containing the sets $\{x,z\}$ and 
$\{y,u\}$, respectively.
Then $ \dist(E,F)\le \sigma(x,y)$. Using 
\eqref{bbbsss} we get, 
$$\diam(E)\wedge \diam(F)\ge \sigma(x,z)\wedge \sigma(y,u)\ge 
\sigma(x,z)\wedge (sd/2). $$
Hence 
 \begin{equation}\label{sep5} 
 \Delta(E,F)\le \frac  {\sigma(x,y)} {\sigma(x,z)\wedge (sd/2)}
\le 1\vee (4/(sd))=:C_{7}(s,d).
\end{equation}

 Let $N\in \N$ and $\psi\:(0,\infty)\ra (0,\infty)$ with  $\lim_{t\to \infty}\psi(t)=0$ be as in Proposition~\ref{prop:uptrans2}.  Then  for  some set 
 $I_0\sub I:=\{1,\dots,n-1\}$ with $\#I_0\le N$ we have that 
 $$M_{W,\mathcal{K}'}(\Gamma(E',F';W))\le \psi(\Delta_{\C^*}(E',F'))\le
  \psi( H/C_{6}),$$ where 
$W=\C^*\setminus \bigcup_{i\in I_0}{Q_i}$
and transboundary modulus is with respect to the collection 
$\mathcal{K}'=\{Q_i:i\in I\setminus I_0\}$.   If $V':=A\setminus \bigcup_{i\in I_0}{Q_i}$, then $U\sub V'\sub W$, and  $\Gamma(E',F';V')\sub \Gamma(E',F';W)$, and so 
$$ M_{V',\mathcal{K}'}(\Gamma(E',F';V'))\le M_{W,\mathcal{K}'}
(\Gamma(E',F';W))\le \psi( H/C_{6}). $$

Define $V=\Sph\setminus (D_0\cup D_n\cup \bigcup_{i\in I_0}D_i)$. 
Note that by Lemma~\ref{lem:curvsepp} the complementary components $D_i$, $i\in I_0\cup\{0,n\}$, of $V$ are $s$-relatively separated. Since $\#I_0$ can be bounded by the  universal constant 
$N$, it follows from Proposition~\ref{thm:lowmod} that the region $V=\Sph\setminus (D_0\cup D_n\cup \bigcup_{i\in I_0}D_i)$  is $\phi$-Loewner with 
$\phi$ only depending on $s$ and $k$.
Combining this with \eqref{sep5} and Proposition~\ref{prop:lowtrans}, we see that there  is a positive constant 
$C_{8}=C_{8}(s,k,d)>0$ such that 
$$ M_{V, \mathcal{K}}(\Gamma(E,F; V))\ge C_{8},$$
where the transboundary modulus in $V$ is with respect to the collection $\mathcal {K}=\{D_i:i\in I\setminus I_0\}.$

 Our (extended) map $f$  is a homeomorphism from $\overline V$ onto $\overline V'$, and a conformal map from $V\setminus \bigcup_{i\in I\setminus I_0}D_i=\Om$ onto $V'\setminus \bigcup_{i\in I\setminus I_0}Q_i
 =U$. Moreover, $f(\Gamma(E,F;V))=\Gamma(E',F'; V')$,  and so invariance of transboundary modulus gives 
 $$M_{V,\mathcal{K}}(\Gamma(E,F; V))=
M_{V',\mathcal{K}'}(\Gamma(E',F'; V')).$$
Hence our estimates lead to the inequality
$$ C_{8}\le \psi( H/C_{6}).$$
Since $\psi$ is a fixed function with $\psi(t)\to 0$ as 
$t\to \infty$, this leads to a contradiction if $H$ is larger than a constant only depending on $s$,  $k$, and $d$.
\qed \medskip

\begin{theorem}\label{thm:cylunif1a} Let $I=\{0,\dots, n\}$, where 
$n\ge 1$, or $I=\N_0$. Suppose that $\{D_i: i\in I\}$ is a collection of pairwise disjoint closed Jordan  regions whose boundaries $\partial D_i$, $i\in I$, form a family of uniformly 
relatively separated   uniform quasicircles.  Then there 
exists a finite $\C^*$-cylinder $A$, pairwise disjoint $\C^*$-squares 
$Q_i\sub A$ for  $i\in I\setminus\{0,1\}$, 
and a quasisymmetric homeomorphism $f\:  T\ra  T'$, where 
\begin{equation}\label{T'}
T=\OC\setminus \bigcup_{i\in I}\inte(D_i)
\quad \text{and}\quad T'=\overline A\setminus \bigcup_{i\in I\setminus\{0,1\}}\inte(Q_i),
\end{equation} that 
 maps 
$\partial D_0$ to $\partial_iA$ and  $\partial D_1$ to $\partial_oA$. Here $T$ and $T'$ are equipped with the restriction of the chordal metric and the flat metric on $\C^*$, respectively.  \end{theorem}

\no{\em Proof.}
If $I$ is finite, then the statement follows from Proposition~\ref{thm:cylqs}.

If $I=\N_0$,  for each $n\in \N$ we  consider  the finitely connected region 
$\Om_n=\Sph\setminus \bigcup_{i=0}^nD_i$.
Then  $\bigcap_{n\in \N}\overline \Omega_n=T$. 
By  Proposition~\ref{thm:cylqs} there exists an $\eta$-quasisymmetric embedding 
$f_n$ of   $\overline \Om_n$ into the closure $\overline  A_n$ of a finite $\C^*$-cylinder $A_n$ auch that $f_n(\partial D_0)=\partial_oA_n$, $f_n(\partial D_1)=\partial_iA_n$, and such that the complementary components of $f_n(\Om_n)$ in $A_n$ are $\C^*$-squares.  
Here the distortion function $\eta$ does not depend  on $n$. 
Postcomposing $f_n$ with a suitable dilation $z\mapsto \lambda z$, 
$\lambda\ne 0$, which does not affect $\eta$, we may in addition assume that $\partial_o A_n=
\partial \D$ for all $n\in \N$.

By Lemma~\ref{lem:subcon} the sequence   $(f_n)$ subconverges on $T$ to an $\eta$-quasisymmetric embedding
$f\:T\ra \C^*$, i.e., there exists a subsequence 
$(f_{n_l})$ of $(f_n)$ that converges to $f$ uniformly 
on $T$. Since   $\partial D_i$ is the boundary of the complementary component $D_i$ of 
$\Om_n$ for $n\ge i$ and  $f_{n_k}\to f$ uniformly, it follows that 
for fixed $i\in \N_0$ the Jordan  curve $f(\partial D_i)$ is the Hausdorff limit  of the sets $f_{n_l}(\partial D_i)$ as $l\to \infty$. Therefore, $f(\partial D_0)=\partial \D$. Since $f_n(\partial D_1)=
\{z\in \C:|z|=r_n\} $ with $r_n\in (0,1)$ for $n\ge 1$, it follows that    $f(\partial D_1)=\{z\in \C:|z|=r\}$  for some $0< r\le 1$. Since $f$ is an embedding, we have $0<r<1$.

 By a similar consideration it follows  $f(\partial D_i)=\partial Q_i$ for $i\ge 2$, where $Q_i$
is a (non-degenerate) $\C^*$-square.    Here  $Q_i\cap Q_{j}=\emptyset$ for $i\ne j$. Indeed, it  is clear that 
$\inte(Q_i)\cap  \inte(Q_{j})=\emptyset$, because
$Q_i$ and $Q_{j}$ can be written as Hausdorff limits of sequences of $\C^*$-squares, where corresponding $\C^*$-squares in the sequences have empty intersection. Moreover, 
$\partial Q_i\cap  \partial Q_{j}=f(\partial D_i)\cap 
f(\partial D_{j})=\emptyset$ for $i\ne j$, because $f$ is an embedding.

Let $A\sub \C^*$ be the finite cyclinder with $\partial_oA=\partial \D$ 
and 
$$\partial_{i}A=f(\partial D_1)=\{z\in \C:|z|=r\}.$$  Since 
$f_n(\overline \Om_n)\sub \overline A_n$ for all $n\ge 1$, and 
$\overline A_{n_l}\ra  \overline A$ as $l\to \infty$, we have 
$T'=f(T)\sub \overline A$. Since $f$ is an embedding,  this implies 
that the $\C^*$-squares $Q_i$, $i\ge 2$, lie in $A$.   
As follows from Lemma~\ref{lem:misc} (iii), we  have  $f(T)\sub Q_i$
or $f(T)\sub \overline A\setminus Q_i$. Here the former case is impossible as $f(\partial D_{j})=\partial Q_{j}$ has empty intersection with $Q_i$ for $j\ne i$. Putting this all together, 
Lemma~\ref{lem:misc} (iii) shows that $T'=f(T)$ can be written as in \eqref{T'}. 
 \qed\medskip
 
 As follows from Proposition~\ref{thm:cylqs} and the previous proof, the statement in Theorem~\ref{thm:cylunif1a} is quantitative in the following 
 sense: if the collection $\partial D_i$, $i\in I$, consists of $s$-relatively separated $k$-quasicircles, and $\diam(\partial D_0)\wedge  \diam(\partial D_1)\ge d>0$, then 
 one can find an $\eta$-quasisymmetric map $f$ with $\eta=\eta_{s,k,d}$. 
 The dependence on $d$ here is unavoidable.
 This can be seen as follows (in the ensuing argument we leave some details to the reader). 
 
  Suppose we could always choose $\eta=\eta_{s,k}$. Then  for each 
 $n\in \N$ we can  find an $\eta$-quasisymmetric map $f_n$  (with $\eta$
 independent of $n$)
mapping  the closure of the  finite $\C^*$-cylinder $A_n=\{z\in \C:1/n<|z|<1\}$ equipped with the chordal metric to  the closure of a finite 
$\C^*$-cylinder $A'_n=\{z\in \C:r_n<|z|<1\}$
  equipped with the   flat metric such that $f_n(\partial \D)=\partial \D$. 
 One can then pass to a sublimit  (this does not  follow directly from Lemma~\ref{lem:subcon}, but from the methods of its proof) 
which produces a quasisymmetric embedding $f$ of $\overline \D\setminus \{0\}$
 equipped with the chordal metric into $\overline  \D\setminus\{0\}$ equipped with the flat metric. This map $f$  also  satisfies  $f(\partial \D)=\partial \D$.   
  
Since $\overline \D\setminus \{0\}$  has  finite  diameter in the chordal metric, its image set $f(\overline \D\setminus \{0\})$  must have  finite diameter in the flat metric. Since $f$ is a quasisymmetry, this implies that 
$f$ is uniformly continuous and so it  has  a continuous extension as a map from 
$\overline \D$ to $\overline \D\setminus \{0\}$ (note that $0$ is ``infinitely far away" in the flat metric). This is impossible for topological reasons. Namely, since  the Jordan curve  $\partial \D$ is contractible in $\D$, its image  $\partial \D=f(\partial \D) $ is contractible in $f(\overline  \D)\sub \overline \D\setminus \{0\}$,  and hence in $\C^*$. This is absurd.

\section{Sierpi\'nski carpets and carpet modulus} \label{s:carpet}

\no
The {\em standard Sierpi\'nski carpet} $T$ is defined as follows.
Let $T_0=[0,1]\times [0,1]\sub \R^2\cong \C$ be the unit square in $\C$. We subdivide $T_0$ into nine subsquares of equal sidelength and remove the interior of the ``middle" square. The resulting set $T_1$ is the union of eight non-overlapping closed squares of Euclidean sidelength $1/3$.  On each of these squares we perform an operation similar to the one that was used to construct $T_1$ from $T_0$. Continuing successively in this manner, we obtain a nested sequence of compact sets 
$T_0\supset T_1\supset T_2 \supset \dots$ such that $T_n$ consists of $8^n$ non-overlapping squares of sidelength $1/3^n$.  Now $T$ is defined as $T=\bigcap_{n\in \N_0}T_n$. 

A {\em (Sierpi\'nski) carpet} is a topological space homeomorphic to the  standard Sierpi\'nski carpet. A metric space $X$ is a carpet if and only if it a locally connected continuum that is  planar, has topological dimension $1$, and has no  local cut-points \cite{Why}.  Here  $X$ is called {\em planar} if it is homeomorphic to a subset of $\oC$. A {\em local cut point} in $X$ is a point $p\in X$ such that for all sufficiently small neighborhoods $U$ of $p$ the set $U\setminus\{p\}$ is not connected.  

A set $T\sub \Sph$ is a carpet if and only if $\inte(T)=\emptyset$ 
and it can be written as 
\begin{equation}\label{carprep}
T=\Sph\setminus \bigcup_{i\in \N_0} \inte(D_i),
\end{equation}
where the sets $D_i$, $i\in \N_0$, form a collection of pairwise disjoint closed Jordan  regions in $\Sph$ with $\diam(D_i)\to 0$ as $i\to \infty$ \cite{Why}. 

A Jordan  curve $S$ in a carpet $T$ is called a {\em peripheral 
circle} if $T\setminus S$ is a connected set. The peripheral 
circles of a carpet as in \eqref{carprep} are precisely the Jordan  curves $\partial D_i$, $i\in \N_0$. In particular, the collection of the peripheral circles of the standard Sierpi\'nski carpet $T$ consists 
of the boundary $\partial T_0$ of the unit square and the boundaries of the squares that were successively  removed from $T_0$ in the construction of $T$. 

A carpet $T\sub \Sph$ is called {\em round} if its peripheral circles are round circles. This is true if and only if the Jordan  regions  $D_i$ in the representation of $T$ as in \eqref{carprep} are round disks. So every round carpet is a Schottky set (see Section~\ref{s:proof1}). Hence it  follows from  \cite[Thm.~1.1]{BKM} that round carpets of spherical measure zero are {\em rigid} in the following sense: if $T\sub \oC$ is a round carpet of spherical measure zero and $f\:T\ra T'$ is a quasisymmetric map of $T$ onto another 
round carpet $T'\sub \oC$, then $f$ is the restriction of a M\"obius transformation to $T$. 

\medskip 
\no{\em Proof of Corollary~\ref{cor:carpetunif}.}  Let  $T$ be  a carpet as in the statement.  Then $T$ can be written as in \eqref{carprep}. By Theorem~\ref{thm:simulunif} there exists  a quasiconformal map $f\:\Sph \ra \Sph$ such that $f(\partial D_i)$ is a round circle for each $i\in \N_0$. Hence we can write 
$T'=f(T)$ as 
$$ T'=\Sph\setminus \bigcup_{n\in \N_0} \inte(D'_i), $$ 
where the sets $D'_i=f(D_i)$ are pairwise disjoint closed disks. 
Since $\inte(T)=\emptyset$, we also have $\inte(T')=\emptyset$, and so $T'$ is a round carpet. By Proposition~\ref{prop:inter}
the map $f$ is a quasisymmetry, and hence also its restriction 
$f|T\: T\ra T'$. The existence part of the statement follows.

Suppose in addition that $T$ has spherical measure zero. Since quasiconformal maps on $\Sph$ preserve such sets (see  \cite[Def.~24.6 and Thm.~33.2]{Va}), the round carpet $T'$ is also a set of spherical measure zero. 
Let  $g\: T\ra \widetilde T$ be  another quasisymmetry onto a round carpet $\widetilde T\sub \Sph$. Then $g\circ f^{-1}$ is a quasisymmetry of $T'$ onto $\widetilde T$. Since round carpets 
of measure zero are rigid, the map $g\circ f^{-1}$ is the restriction of a M\"obius transformation, and so $g$ is equal to $f$ post-composed with 
a M\"obius transformation. So we also have the uniqueness part of the statement,   and the proof is complete. \qed \medskip

Let $T\sub \Sph$ be a carpet, and $f\: T\ra \Sph$ be an embedding. Then $T'=f(T)$ is also a carpet, and $f$ induces a bijection between the peripheral circles of $T$ and $T'$. It was shown in the proof of Lemma~\ref{lem:misc} (iii) that there exists a homeomorphism 
$F\: \Sph\ra \Sph$ with $F|T=f$.  We call $f$ {\em 
orientation-preserving} if $F$ is orientation-preserving (with respect to the standard orientation on $\oC$). This does not depend on the choice of the homeomorphic extension $F$ of $f$. In more intuitive terms, $f$ is orientation-preserving if the following condition is true: if we orient each  peripheral circle $S$ 
of $T$ so that $T$ lies ``on the left" of $S$, then the induced orientation on the  peripheral circle $S'=f(S)$ of $T'=f(T)$ is such that $T'$ lies  ``on the left" of $S'$. 

\medskip
\no {\em Proof of Theorem~\ref{thm:3circ}.} Let $T\sub \C$ be a carpet as in the statement. As we have seen in the proof of Corollary~\ref{cor:carpetunif}, there exists a quasiconformal map $g\: \Sph \ra \Sph$ such that $T'=g(T)$ is a round carpet of measure zero. If $f\: T\ra T$ is a quasisymmetry, then $f'=g\circ f \circ g^{-1}|T'$ is also a quasisymmetry. Since $T'$ is rigid, it follows that $f'=F'|T'$ is the restriction of a M\"obius transformation $F'\: \Sph \ra \Sph$.  Suppose in addition that  $f$ is orientation-preserving. Then  the same is true for $f'$ and hence for $F'$.

If $f$ has three distinct fixed points, the same is true 
for $f'$ and  for $F'$. So $F'$ is the identity map on $\Sph$, which implies that $f$ is the identity of $T$.

Similarly, if $f$ fixes three distinct peripheral circles of $T$ setwise, then $f'$ fixes three distinct peripheral circles of $T'$ setwise. Since the peripheral circles of $T'$ are round circles,
it follows that $F'$ fixes three disjoint round circles setwise. Moreover, these circles bound pairwise disjoint disks. Since $F'$ is an orientation-preserving M\"obius transformation, $F'$ must be the identity map on $\Sph$. Hence $f$ is the identity on $T$. \qed  \medskip

\no {\em Proof of Theorem~\ref{thm:cylunif0}.} This is a special case of Theorem~\ref{thm:cylunif1a}. \qed \medskip

Let $T\sub \Sph$ be a carpet represented as in \eqref{carprep}, and $\Ga$ be a collection of paths in $\OC$. We 
define the {\em carpet modulus} of $\Ga$ {with respect to $T$}, denoted by $\M_T(\Ga)$, as follows.
Let $\rho_i\ge 0$ for $i\in \N_0$. We call the weight sequence 
$(\rho_i)_{i\in \N_0}$ {\em admissible} for $\Ga$ (with given $T$) if 
there exists a family $\Ga_0\sub\Ga$ with 
$\Mod(\Ga_0)=0$ such that 
$$ \sum_{\ga\cap D_i\ne \emptyset} \rho_i\ge 1 \foral \ga\in \Ga\setminus \Ga_0. $$
Then 
$$ \M_T(\Ga):= \inf_{(\rho_i)} \sum_{i\in \N_0}\, \rho_i^2, $$
where the infimum is taken over  all weight sequences $(\rho_i)$
that are admissible for $\Ga$.  An admissible weight sequence for which this infimum is attained is called {\em extremal}. 

It is  essential  here to allow the exceptional 
family $\Ga_0$ with  vanishing modulus in the classical sense.   Of course, one could define carpet modulus by requiring  the inequality in the admissibility condition for {\em all} $\ga\in \Ga$, but  this leads  to a notion of carpet modulus that is not very interesting.
Our notion of carpet modulus is useful for  
studying the quasiconformal geometry of carpets, since it is related to the geometry of the carpet and is   invariant under quasiconformal maps.

\begin{proposition} [Quasiconformal invariance of carpet modulus]  \label{prop:invcarmod}
Let $T\sub \OC$ be a carpet,   $\Ga$  a family of paths in 
$\OC$, and $f\: \OC\ra \OC$  a quasiconformal map.
Then 
$$ \M_T(\Ga)=\M_{f(T)}(f(\Ga)). $$
\end{proposition}

\no{\em Proof.} Note that $T'=f(T)$ is also a carpet. If 
$T$ is represented as in \eqref{carprep}, then
$$T'=\oC\setminus \bigcup_{i\in \N_0} \inte(D_i'),$$ where 
$D'_i=f(D_i)$ for $i\in \N_0$. 
Moreover, we have $\ga\cap D_i\ne \emptyset $ for $\ga\in \Ga$ if and only if $f(\ga)\cap f(D_i)\ne \emptyset$.

A quasiconformal map preserves the modulus  of a path family up to a fixed multiplicative constant (Proposition~\ref{prop:moddist}). So if $\Ga_0\sub \Ga$ and $\Mod(\Ga_0)=0$, then $\Mod(f(\Ga_0))=0$.   
This implies that if $(\rho_i)$ is an admissible weight sequence for $\Ga$ with 
respect to the carpet $T$, then it is also admissible for 
$\Ga'=f(\Ga)$ with respect to the carpet $T'$. 
Hence $\M_{T'}(\Ga')\le \M_T(\Ga)$. Applying the same argument to the quasiconformal map  $f^{-1}$, we get an inequality in the other direction. Hence 
$\M_{T'}(\Ga')= \M_T(\Ga)$ as desired. \qed \medskip

The crucial  point in the previous proof was that while 
 quasiconformal maps only preserve the moduli of general path families  up to a multiplicative constant, they  preserve the modulus of a path family  with vanishing modulus. 
 
 Suppose $T$ is a carpet as in \eqref{carprep}. Consider the path family 
 $$\Gamma=\Gamma_o(\partial D_0, \partial D_1; \Sph\setminus 
(D_0\cup D_1))$$ of all open paths in the topological annulus $\Sph\setminus 
(D_0\cup D_1)$ connecting its boundary components $\partial D_0$ and $\partial D_1$. We are interested in finding  $\M_T(\Gamma)$. The next statement shows that with suitable assumptions on $T$ the answer is very similar to   the answer to the corresponding question for transboundary modulus 
studied in Section~\ref{s:square}.   A subtlety here is that it is better  to consider the family of {\em open} paths  
$\Gamma$ instead of the family of closed paths 
$\Gamma'=\Gamma(\partial D_0, \partial D_1; \Sph\setminus 
(D_0\cup D_1)).$  In contrast to the paths in $\Gamma$, the paths in $\Gamma'$ meet $D_0$ and $D_1$, so one obtains more admissible sequences by putting  non-zero weights on $D_0$ and $D_1$. By choosing 
 the weights $1/2$ on $D_0$ and $D_1$, and all other weights equal to $0$, for example, one gets  trivial  
 inequalities such as 
  $\M_T(\Gamma)\le 1/2$ which do not reflect the geometry of $T$.

\begin{corollary} \label{cor:annmod}Let $T\sub \C$ be a carpet of spherical measure zero whose peripheral circles are uniformly separated uniform quasicircles. Suppose $T$ is represented as in \eqref{carprep} and $f\: T\ra T'$ is a quasisymmetric map as in Theorem~\ref{thm:cylunif0} with  
$$T'=\overline A\setminus \bigcup_{i\ge 2}\inte(Q_i), $$
where  $A$  a finite $\C^*$-cylinder and the sets $Q_i$, $i\ge 2$, are pairwise disjoint $\C^*$-squares in $A$, and we have 
$f(\partial D_0)=\partial_iA$ and $f(\partial D_1)=\partial_oA$.  
Let $$\Gamma=\Gamma_o(\partial D_0, \partial D_1; \Sph\setminus 
(D_0\cup D_1)). $$ Then 
$$\M_T(\Gamma)=\frac {2\pi}{h_A}. $$ 

Moreover, a   unique extremal weight sequence $(\rho_i)_{i\in \N_0}$ for 
$\M_T(\Gamma)$ exists and is given by 
\begin{equation}\label{extdens}
\rho_0=\rho_1=0\quad \text{and}\quad  \rho_i=\ell(Q_i)/h_A\forr i\ge 2.\end{equation} 
\end{corollary}

\no {\em Proof.} Since the metric $d_{\C^*}$  and the spherical metric are comparable on $\overline A$, the map $f$ is a quasisymmetric and hence also a 
quasi-M\"obius embedding from $T$ into $\Sph$ (equipped with the chordal metric). By
Proposition~\ref{prop:extend} it has an extension to quasiconformal map  $F\:\Sph \ra \Sph$.  Since $T$ has measure zero and  quasiconformal maps preserve such sets,  the set $T'=f(T)=F(T)$ has spherical measure zero. 
Hence also $A_{\C^*}(T')=0$ and so  we have 
\begin{equation}\label{measallQ}\sum_{i\ge 2}\ell(Q_i)^2=\sum_{i\ge 2}A_{\C^*}(Q_i)=A_{\C^*}(A)=
2\pi h_A.
\end{equation} 

Note that $\Gamma':=F(\Gamma)=\Gamma_o(\partial_iA, \partial_oA; A)$.  So by quasiconformal invariance of carpet modulus (Proposition~\ref{prop:invcarmod}) for the first part of the statement it suffices to show that 
$$ \M_{T'}(\Gamma')=\frac{2\pi}{h_A}. $$
Now the argument   is very similar to the proof of Proposition~\ref{prop:transcyl}. We can write $A=\{z\in \C: r<|z|<R\}$, where 
$0<r<R$.  Then the closures of the  complementary components of the carpet $T'$ in $\Sph$ are the sets $\overline{B}_{\C}(0,r)$, $\Sph\setminus {B}_{\C}(0,R)$, and $Q_i$, $i\ge 2$.
They are labeled by $0$, $1$, and $i$, respectively. We  define 
a corresponding weight sequence $(\rho_i)_{i\in \N_0}$ by $\rho_0=\rho_1=0$ and 
$\rho_i=\ell(Q_i)/h_A$ for $i\ge 2$.

We claim that this weight  sequence  is admissible
for the modulus  $M_{T'}(\Gamma')$. To see this let $\Gamma_0\sub \Gamma'$ be the family of all paths $\alpha\in \Gamma'$ that are not locally rectifiable or are locally rectifiable and satisfy
$$\text{length}(\alpha\cap T'):=\int_\alpha \chi_{T'}\, ds>0. $$ Since $T'$ has measure zero, we have $\Mod(\Gamma_0)=0$. Indeed, the function 
$\rho$ defined by $\rho(z)=\infty$ for $z\in  T'$ and $\rho(z)=0$ for $z\in \Sph\setminus T'$ is an admissible density for 
$\Gamma_0$ with $\int\rho^2\,d\Sigma=0$.

Now let $\alpha\in \Ga\setminus \Gamma_0$ be  arbitrary.  If $z\mapsto \pi(z):=\log|z| $
 is the map  of $ {A}$ to the interval 
 $(\log r, \log R)$ defined by $z\mapsto \pi(z):=\log|z| $, 
 then 
 \begin{equation}
 (\log r, \log R)= \pi(\alpha )\sub \pi (\alpha\cap T')
 \cup \bigcup_{\alpha\cap Q_i\ne \emptyset}\pi(\alpha\cap Q_i).
 \end{equation}

All subsets of $\R$ appearing in the last inclusion are Borel sets, and hence measurable. Since $\alpha\not\in \Gamma_0$, this path is locally rectifiable and we have  $\length(\alpha\cap T')=0$. Since 
$\pi $ is Lipschitz (it is $1$-Lipschitz if $ A$ is equipped with flat metric, and hence also Lipschitz with respect to the chordal metric) this implies 
that $m_1(\pi (\alpha\cap T'))=0$. Hence 
\begin{eqnarray*}\sum_{\alpha\cap Q_i \ne \emptyset} \rho_i&=& 
 \frac 1{h_A}\sum_{\alpha \cap Q_i\ne \emptyset} \ell(Q_i)= \frac 1{h_A}\sum_{\alpha \cap Q_i\ne \emptyset} m_1 (\pi(Q_i)) \\
&\ge&  \frac 1{h_A} m_1(\pi (\alpha\cap T'))+ \frac 1{h_A}
\sum_{\alpha \cap Q_i\ne \emptyset} 
m_1(\pi(\alpha\cap Q_i))\ge  \frac 1{h_A}  m_1((\log r, \log R))=1.
\end{eqnarray*}
The admissibility of our  weight sequence  follows, and 
we conclude that  
\begin{eqnarray*}
\M_{T'}(\Ga')&\le &
\sum_{i\in \N_0} \rho_i^2 \, =\, \frac1{h_A^2}
\sum_{i\ge 2} \ell(Q_i)^2\\
&=&\frac1{h_A^2}
\sum_{i\ge 2} A_{\C^*}(Q_i)\, = \, \frac1{h_A^2}A_{\C^*}(A)=\frac{2\pi}{h_A}.
\end{eqnarray*}

To get an inequality in the other direction,  suppose that we have an  admissible  weight sequence  $(\rho_i)_{i\in \N_0}$ 
for the family $\Gamma'$.  

For each $\varphi\in [0,2\pi]$ the path  $\alpha_{\varphi}\:(\log r, \log R)\ra A $ defined by  $\alpha_{\varphi}(t):=
t e^{\iu \varphi}$ for   $t\in(\log r, \log R)$ belongs to $\Ga'$.  There exists a family $\Gamma_0\sub \Gamma'$ with $\Mod(\Gamma_0)=0$ such that 
\begin{equation}\label{aephi}
\sum_{\alpha_\varphi \cap Q_i \ne \emptyset} \rho_i\ge 1 
\end{equation} 
for all $\varphi\in [0,2\pi]$ with $ \alpha_{\varphi}\not\in \Gamma_0$. 
The set $E$ of all $\varphi\in [0,2\pi]$ for which  this inequality fails is a Borel set 
($E$ is the  preimage of $[0,1)$ under   the  Borel function on $[0,2\pi]$ given 
by  $\sum_{i\ge 2} \rho_i\chi_{F_i}$; here 
$F_i=\{\varphi \in [0,2\pi]:  \alpha_{\varphi}\cap Q_i\ne \emptyset\}$ is a closed set for $i\ge 2$). Hence $E$   is measurable, and it must have $1$-dimensional Lebesgue measure  zero, 
since the corresponding family of  paths $\{\alpha_\varphi: \varphi\in E\}$ is  contained in $\Gamma_0$ and so is a  family with vanishing modulus.  
Thus,   \eqref{aephi} is valid for almost every $\varphi\in [0,2\pi]$.

By integrating this inequality  over $\varphi$, and using Fubini's theorem, the Cauchy-Schwarz inequality and \eqref{measallQ}, we arrive at
\begin{eqnarray}
2\pi &\le &
\sum_{i\ge 2} \ell(Q_i)\rho_i  \
\,\le \, \biggl(\sum_{i\ge 2}\ell(Q_i)^2\biggr)^{1/2}
\biggl(
\sum_{i\ge 2} \rho_i^2\biggr)^{1/2}\label{CS3}\\
&=& (2\pi h_A)^{1/2}\biggl(
\sum_{i\ge 2} \rho_i^2\biggr)^{1/2}\le (2\pi h_A)^{1/2}\biggl(
\sum_{i\in \N_0} \rho_i^2\biggr)^{1/2}.\nonumber
\end{eqnarray}

It follows that 
$$
\sum_{i\in \N_0} \rho_i^2\ge \frac{2\pi}{h_A} $$
for every weight sequence  that is admissible for $\Ga'$.
This implies the other  inequality
$\M_{T'}(\Ga')\ge {2\pi}/{h_A}$, and so  $\M_{T}(\Ga)=\M_{T'}(\Ga')= {2\pi}/{h_A}$ as desired. 

If we have 
$$
\sum_{i\in \N_0} \rho_i^2= \frac{2\pi}{h_A} $$
for an admissible weight sequence,  then all inequalities
in  \eqref{CS3}   must be  equalities. This implies that $\rho_0=\rho_2=0$ 
and that  there 
exists $\lambda>0$ such that  $\rho_i=\la \ell(Q_i)$ for 
$i\ge 2$.  Then  $\la =1/h_A$, and  so  $\rho_i= \ell(Q_i)/h_A$ 
for $i\ge 2$.  This shows that \eqref{extdens}  gives the unique extremal weight sequence  for 
$\M_{T'}(\Gamma')$.
 Since admissible weight sequences  for $\M_{T'}(\Gamma')$ and $\M_{T}(\Gamma)$ correspond to each other by the map $F$ (see the proof of Proposition~\ref{prop:invcarmod}), we see that the weight sequence \eqref{extdens} is also the unique extremal weight sequence  for 
 $\M_T(\Gamma)$. \qed\medskip 
 
 Similarly as classical modulus and transboundary modulus are useful for   proving uniqueness results for conformal maps (see Corollary~\ref{cor:uniqsquare}), carpet modulus can be employed to establish rigidity statements  for quasisymmetric maps on carpets. For example, using  this 
 concept (in combination with other ideas) one can show that every quasisymmetric self-homeomorphism of the standard Sierpi\'nski carpet 
 (equipped with the restriction of the Euclidean metric) is an isometry. 
 In particular, there are precisely $8$ such quasisymmetries (the obvious rotations and reflections). See \cite{BM}
 for this result and related investigations.

\section{Hyperbolic groups with carpet boundary}
\label{s:hypgroups} 

\no 
The material in this section is independent of the rest of the paper. Its purpose is the proof of Proposition~\ref{prop:groupcarpets} that motivates the study of carpets whose peripheral circles are uniformly relatively separated uniform quasicircles.

We quickly
review some standard facts on Gromov hyperbolic groups. See \cite{gh}
and \cite{BuS} for general background on Gromov hyperbolic groups and Gromov hyperbolic spaces.

Let  $G$ be  a finitely generated group, 
and $S$ a finite set of generators of $G$ that is {\em symmetric} (i.e., if
 $s\in S$, then $s^{-1}\in S$). The  group   $G$ is called {\em Gromov hyperbolic} if the  Cayley  
graph $\mathcal{G}(G, S)$  of $G$ with respect to $S$  
is  Gromov hyperbolic as a metric space.   In this  case, 
$\mathcal{G}(G,S')$ is Gromov hyperbolic for each  (finite and symmetric) generating 
sets $S'$.  For the  basic definitions and facts here, see \cite[Ch.~1]{gh}.   

Associated with  every Gromov hyperbolic metric space $X$ is a boundary 
at infinity $\partial_\infty X$ equipped with a natural class of visual metrics \cite[Ch.~2]{BuS}.  Accordingly, one  defines the boundary at infinity of a Gromov hyperbolic group $G$ as $\partial_\infty G=\partial_\infty \mathcal{G}(G,S)$. This is well-defined, because if $S'$ is another  generating set, then there is a natural identification 
 $\partial_\infty \mathcal{G}(G,S')=\partial_\infty \mathcal{G}(G,S)$ (the elements in both spaces can be represented by equivalence classes of sequences {\em in} $G$ converging to infinity; moreover, equivalence of such sequences is independent of the generating sets $S$ and $S'$).  If $d'$ and $d$ are  visual metrics on $\partial_\infty \mathcal{G}(G,S')$ and $\partial_\infty  \mathcal{G}(G,S)$,
respectively, then there are quasisymmetrically equivalent, i.e., 
the identity map between $(\partial_\infty  \mathcal{G}(G,S'), d')$ and $(\partial_\infty  \mathcal{G}(G,S), d)$ is a quasisymmetry (this follows from the fact that $\mathcal{G}(G,S)$ and $\mathcal{G}(G,S')$
are quasi-isometric, and that every quasi-isometry between geodesic Gromov hyperbolic metric spaces induces a quasisymmetric map between their boundaries at infinity; see \cite[5.35 Thm.]{Va3} for a precise quantitative version of the last fact).   So if we equip 
 $\partial_\infty  G$ with any of these visual metrics $d$, then we can unambiguously speak of quasisymmetric and quasi-M\"obius maps 
 on $\partial_\infty G$.  Moreover, the space $(\partial_\infty G, d)$ is doubling (see   \cite[Thm.~9.2]{BS}  and the remarks after  this theorem; note that the proof of \cite[Thm.~9.2]{BS} contains some inaccuracies; they can easily be corrected). 
 
 The natural left-action of $G$ on $ \mathcal{G}(G,S)$ by isometries  induces an action 
 of $G$ on $\partial_\infty G$ by quasisymmetries. So each $g\in G$ can be considered as a quasisymmetry  on $\partial_\infty G$, and  we write $g(x)$ for the image of a point $x\in \partial_\infty G$  under $g\in G$. In general the  action of $G$ on $\partial_\infty G$ is not {\em effective}, i.e., there can be elements $g\in G$ that act as the identity on 
  $\partial_\infty G$.  If $G$ is {\em non-elementary} (i.e., if $\#\partial_\infty G\ge 3$), then the elements of $G$ acting on $\partial_\infty G$ form a finite and normal subgroup of $G$, the {\em ineffective kernel} (this follows  from \cite[Ch.~8, 36.-Cor.]{gh}; note that every element in the ineffective kernel is {\em elliptic}  and hence has finite order \cite[Ch.~8, 28.-Prop.]{gh}).  
  
  Two properties of the action of $G$ on $\partial_\infty G$ (equipped with a fixed visual metric $d$) will be used in the following. This action is 
  {\em uniformly quasi-M\"obius}, i.e., there exists a
  homeomorphism $\eta\:[0,\infty] \ra [0,\infty]$ such that each 
  $g\in G$ acts as a $\eta$-quasi-M\"obius map on $\partial_\infty G$  (this goes back to the remark preceding \cite[Thm.~5.4]{Pau}; it  easily follows from \cite[5.38 Thm.]{Va3}). 
  
  Moreover, the action is {\em cocompact on triples}. This means that 
  there exists a constant $\eps_0>0$ with the following property:
whenever $z_1,z_2,z_3$ are three distinct points in $\geo G$,
then  
there exists $g\in G$ such that 
\begin{equation}\label{eq:coco}
 d(g(z_i),g(z_j))\ge\eps_0 \forr i,j=1,2,3,\ i\ne j 
\end{equation} 
(see  the discussion in \cite[pp.~215--216]{Gr}).

Before we turn to the proof of Proposition~\ref{prop:groupcarpets}, we have to explain the terminology used in its statement.  
Let $T$ be a metric carpet, and $\mathcal{S}=\{S_i:i\in I\}$ be  the collection of its  peripheral circles labeled by a countable index set. Recall from Section~\ref{s:quasicircles} that we call the collection  $\mathcal{S}$ uniformly relatively separated if there exists $s>0$ such that $\Delta(S_i, S_j)\ge s$ whenever $i,j\in I$, $i\ne j$. 
We say that $\mathcal{S}$ consists of {\em uniform quasicircles}
if any of the quantitatively equivalent conditions in Proposition~\ref{prop:metricqcirc}  is satisfied for each peripheral circle $S_i$, $i\in I$,
with the same parameters. If $T$ is doubling, then the peripheral circles are {\em uniformly doubling}, i.e., there exists $N\in \N$ such that $S_i$ is $N$-doubling for each $i\in I$. In this case one can establish that 
$\mathcal{S}$ consists of  uniform quasicircles by showing that there exists $\delta>0$ such that whenever $i\in I$ and $x_1,x_2,x_3,x_4$ are four distinct points in cyclic order on $S_i$, then 
$$[x_1,x_2,x_3,x_4]\ge \delta.$$

Finally, we say that the peripheral circles of $T$ {\em occur on all locations and scales}  if there exists a constant $c>0$ such that for each $x\in T$ and each $0<r\le \diam(T)$ there exists a peripheral circle 
$S$   of $T$ with $S\sub B(x,r)$ and $\diam(S)\ge cr$. Note that in this case 
we also have $\diam(S)\le2 r$. So  the peripheral circles occur on all locations and scales if every ball in $T$ of radius $r\le \diam(T)$ contains a peripheral circles of diameter comparable to $r$.

The ensuing  proof of Proposition~\ref{prop:groupcarpets} uses a well-known idea in complex dynamics and in the theory of Kleinian groups, namely the ``principle of the conformal elevator" (see \cite{HP} for more discussion): in order to establish a geometric property on all scales, one uses the dynamics to map to the ``top scale", verifies the relevant condition there, and uses suitable distortion estimates 
to translate between scales.

\medskip
\no {\em Proof of Proposition~\ref{prop:groupcarpets}. }
Let $G$ be a Gromov hyperbolic group whose boundary at infinity $\partial_\infty G$ is a carpet. We equip 
$\partial_\infty G$ with a fixed visual metric $d$. We denote 
the peripheral circles of $T=\partial_\infty G$ by
$S_i$, $i\in \N$. Since the action of $G$ on $\geo G$ is 
uniformly quasi-M\"obius, there exists a distortion function 
$\eta$ such that $g\:\geo G\ra \geo G$ is an 
$\eta$-quasi-M\"obius homeomorphism for each $g\in G$. Moreover, since the
action of $G$ on $\geo G$ is cocompact on triples, there exists a constant $\eps_0>0$ as in \eqref{eq:coco}.  

The basic idea now  is to  apply  the conformal elevator principle mentioned before the proof. Since the action of $G$ on  $\geo G$ is cocompact on triples, we  will be able to ``map every scale to the top scale" by a suitable  group element. The relevant distortion estimates will be derived from the fact that the action of $G$ on $\geo G$ is 
uniformly quasi-M\"obius. Accordingly, we will  formulate  the  geometric conditions in question in terms of  cross-ratios. 

 Since   $\partial_\infty G$ is doubling,  there exists 
$N\in \N$ such that each circle $S_i$, $i\in \N$, is $N$-doubling.
So for  proving that the collection $S_i$, $i\in \N$, consists of uniform quasicircles it is by Proposition~\ref{prop:metricqcirc} enough to find $\de>0$ such that 
$$[x_1,x_2,x_3,x_4]\ge \de, $$
whenever $x_1, x_2,x_3, x_4$ are four distinct points on one of the circles $S_i$ that are in cyclic order on $S_i$. 

We argue by contradiction and assume that no such $\delta>0$ exists. Then for   
 $n\in \N$ we can find distinct  points $x^n_1, x^n_2,  x^n_3, x^n_4$ 
that lie in cyclic order on some  peripheral circle $S'_n\in \{S_i:i\in \N\}$ such that 
$$ [x^n_1, x^n_2, x^n_3,  x^n_4]\to 0 \quad\text{as}\quad n\to \infty.$$
Since  the action of $G$ on $ \geo G$ is  cocompact on triples,
  for each  $n\in \N$ there exists 
 $g_n\in G$     such that 
 \begin{equation}\label{uptsep0}
  d(y^n_i,y^n_j)\ge \eps_0 \forr i,j=1,2,3,\  i\ne j. 
  \end{equation}
 Here we set  $y_i^n=g_n(x^n_i)$ for $i=1,2,3,4$, $n\in \N$.

Since the action $G$ on $ \geo G$ is uniformly 
quasi-M\"obius, we have
$$ [y^n_1, y^n_2,  y^n_3, y^n_4]\to 0 \quad\text{as}\quad n\to \infty.$$
Every homeomorphism on a carpet preserves the collection of peripheral circles  and the cyclic order of points on peripheral circles. It follows that for each $n\in \N$ the set $J_n=g_n(S'_n)$ is a peripheral circle of $\geo G$ on which the points $y^n_1, y^n_2, y^n_3, y^n_4$ are in cyclic order. 
By \eqref{uptsep0} we have 
$$ \diam (J_n)\ge \eps_0>0 \foral n\in \N.$$
Since every carpet has only finitely many peripheral circles
whose diameter exceeds a given positive constant (this follows from the corresponding fact from the standard carpet), 
there are only finitely many peripheral circles among the sets 
$J_n$, $n\in \N$. In particular, one circle, say $J:=J_{n_0}$, 
is repeated infinitely often in the sequence $J_1, J_2,\dots$.
So by passing to a subsequence if necessary, we may assume that all points $y^n_1, y^n_2, y^n_3,  y^n_4$, $n\in \N$, lie on the peripheral circle $J$.  By passing to further subsequences if necessary, we may assume that 
$$ y_i^n \to y_i\in J \quad\text{as}\quad n\to \infty \forr i=1,2,3,4. $$

By \eqref{uptsep0} we have 
$$ y_i\ne y_j \forr i=1,2,3,\ i\ne j.$$
Moreover, since the points $y^n_1, y^n_2, y^n_3,  y^n_4$ are  in cyclic order on $J$, the point $y_4$ is contained in the subarc $\alpha$ of $J$ 
with endpoints $y_1$ and $y_3$ that does not contain 
$y_2$. Hence $y_2\ne y_4$, and it follows that
\begin{eqnarray*}
0&= & \lim_{n\to \infty} [y^n_1, y^n_2, y^n_3,  y^n_4] \\
&=& \lim_{n\to \infty}
 \frac{d(y^n_1, y^n_3)d(y^n_2,y^n_4)}
 {d(y^n_1, y^n_4)d(y^n_2,y^n_3)} \\
 &=& \frac{d(y_1, y_3)d(y_2,y_4)}{d(y_1, y_4)d(y_2,y_3)}\in 
 (0,+\infty].
\end{eqnarray*}
Here the last expression is interpreted as $+\infty$ if 
$d(y_1, y_4)=0$, and is a finite non-zero number if 
$d(y_1, y_4)\ne 0$. Note that all other terms are non-zero. In any case we get  a contradiction showing that 
the peripheral circles of $\geo G$ are uniform quasicircles.

The argument for showing uniform relative separation 
of the peripheral circles uses similar ideas. Again we argue by contradiction and assume that 
there is a sequence of pairs $S'_n$ and $S''_n$ of two distinct  peripheral circles of $\geo G$ such that 
$$ \Delta(S'_n, S''_n)\to 0 \quad\text{as}\quad n\to \infty.$$
By Lemma~\ref{lem:crrelsep} and Lemma~\ref{lem:modcr} we can then find points 
$x^n_1,x^n_4\in S'_n$ and $x^n_2,x^n_3\in S''_n$
such that 
$$ [x^n_1, x^n_2, x^n_3, x^n_4]\to 0 \quad\text{as}\quad n\to \infty.$$

Again using that the action of $G$  on $\geo G$ is cocompact on triples,  
we can find  
 $g_n\in G$ for  $n\in \N$ such that 
 \begin{equation}\label{uptsep}
  d(y^n_i,y^n_j)\ge \eps_0 \forr i,j=1,2,3,\  i\ne j, 
  \end{equation}
 where  $y_i^n=g_n(x^n_i)$ for $i=1,2,3,4$, $n\in \N$.
Since the action of $G$ on $\geo G$ is uniformly
quasi-M\"obius, we see that 
\begin{equation}\label{crto0}
 [y^n_1, y^n_2, y^n_3, y^n_4]\to 0 \quad\text{as}\quad n\to \infty.
\end{equation} 
Let $J_n=g_n(S_n')$ and $J'_n=g_n(S''_n)$ for $n\in \N$.
For each $n\in \N$ the sets  $J_n$ and $J'_n$ are two distinct peripheral circles of $\geo G$ with $y^n_1,y^n_4\in 
J_n$ and $y^n_2,y^n_3\in J'_n$. Using \eqref{crto0} in combination with Lemma~\ref{lem:crrelsep} and Lemma~\ref{lem:modcr}, we conclude that 
\begin{equation}\label{usepto0}
 \Delta(J_n,J_n') \to 0\quad\text{as}\quad n\to \infty.
 \end{equation}
Note that 
$$\diam(J'_n)\ge d(y^n_2, y^n_3)\ge \eps_0 \forr n\in \N,$$
and
$$ [y^n_1,  y^n_2, y^n_3, y^n_4]=
\frac{d(y^n_1, y^n_3)d(y^n_2,y^n_4)}
 {d(y^n_1, y^n_4)d(y^n_2,y^n_3)} \ge 
 \frac{ \eps_0d(y^n_2,y^n_4)}{\diam(\geo G)^2}.$$
 This forces the relation $d(y^n_2,y^n_4)\to 0$ as $n\to \infty$,
 and hence 
$$ \diam(J_n)\ge d(y^n_1,y^n_4)\ge d(y^n_1,y^n_2)-
d(y^n_2,y^n_4) \ge \eps_0/2$$
for large $n$.

So all but finitely many of the peripheral circles  $J_n$ and $J_n'$ have diameter $\ge \eps_0/2>0$. As in the first part of the proof, this  shows that the collection of all peripheral circles $J_n$ and $J_n'$, $n\in \N$, is 
 finite, and hence there are only finitely many pairs
$(J_n,J'_n)$. Since for each pair $\Delta(J_n,J_n')>0$,  we must have
$$ \inf_{n\in \N} \Delta(J_n,J_n')>0. $$
This contradicts \eqref{usepto0}, showing that the peripheral circles of $\geo G$ are indeed uniformly 
relatively separated. 

To prove the final statement we start with two general remarks about arbitrary carpets. Namely, if $T$ is a carpet, then 
every nonempty open set $U\sub T$ contains a peripheral circle.
This is obviously true for the standard Sierpi\'nski carpet, and so it holds for all carpets.

Secondly, if $T$ is an arbitrary metric carpet, then for every 
$r>0$ there exists $\de>0$ such that every open ball in $T$ of radius $r$ contains a peripheral circle $J$ with $\diam(J)>\de$. For otherwise, there exists $r>0$,  and a sequence of balls 
$B_n=B(x_n,r)$ in $T$ such that $B_n$ does not contain any 
peripheral circle of diameter $\ge 1/n$. Using the compactness of $T$ and passing to a subsequence 
if necessary we may assume that $x_n\to x\in T$ as $n\to \infty$. Then $B=B(x,r/2)\sub  B(x_n,r) $ for large $n$ and so 
the open and nonempty set $B$ cannot contain any peripheral circle of $T$. This contradicts the first remark.

Now let $G$ be a Gromov hyperbolic group with carpet boundary $\geo G$ as in the beginning of the proof. Let $B=B(x,r)$ with $x\in\geo G$ and 
$0<r\le \diam(\geo G)$ be arbitrary. Let  $\la\ge 2$ be a large constant whose precise value we will determine later.
Define $x_1=x$. Since $\geo G$ is connected, we can find 
points $x_2,x_3\in B(x,r/\la)$ such that 
$$ d(x_i,x_j)\ge r/(4\la)\forr i,j=1,2,3, \ i\ne j. $$
Since the action of $G$ on $\geo G$ is cocompact on triples,
we can find $g\in G$ such that 
$$ d(y_i,y_j)\ge \eps_0\forr i,j=1,2,3, \ i\ne j, $$
where $y_i=g(x_i)$ for $i=1,2,3$. 

We claim that  if $\la$ is large enough, only depending on 
$\eta$, $\eps_0$ and $\diam(\geo G)$, then 
\begin{equation}\label{diambdd}
 \diam(\geo G\setminus g(B))=
\diam(g(\geo G\setminus B))< \eps_0/2. 
\end{equation}
To find such $\la$ let $u,v\in \geo G\setminus B$ be arbitrary.
Then using the inequalities 
$$d(x_1,x_3)\le r/\la\le r/2 \le \frac 12 d(u,x_1)$$ and 
$$d(u,x_3)\ge d(u,x_1)-d(x_3,x_1)\ge \frac12 d(u,x_1)$$
we obtain 
\begin{eqnarray*}
[g(x_1), g(u), g(x_3), g(v)] &\le & \eta([x_1, u, x_3, v])\\
&=& \eta\biggl(\frac{d(x_1,x_3) d(u,v)}{d(x_1,v) d(u,x_3)}\biggr)\\
&\le&\eta\biggl(\frac{2r}\la\cdot \frac{ d(u, x_1) +d(v,x_1)}
{d(v,x_1) d(u,x_1)}\biggr)\\
&\le &
\eta\biggl(\frac{2r}\la\cdot \frac{2}
{d(v,x_1)\wedge d(u,x_1)}\biggr)\\
&\le &
\eta(4/\la). 
\end{eqnarray*}
On the other hand, 
\begin{eqnarray*}
[g(x_1), g(u), g(x_3), g(v)] &= & 
 \frac{d(y_1,y_3) d(g(u),g(v))}{d(y_1,g(v)) d(g(u),y_3)}\\
&\ge&\frac{\eps_0d (g(u), g(v))}{\diam(\geo G)^2}.
 \end{eqnarray*}
This implies that
$$ \diam(\geo G\setminus g(B))= \sup_{u,v\in \geo G\setminus B} d(g(u), g(v)) \le \frac 1{\eps_0} \diam(\geo G)^2
\eta(4/\la). $$
As $\eta(t)\to 0$ for $t\to 0$ this shows that we can indeed
find $\la=\la(\eps_0, \eta, \diam(\geo G))\ge 2$ independent of our initial choice of $B$ such that 
 \eqref{diambdd} holds.

By the remark above we can find $\de>0$ such that every open ball in $\geo G$ of radius $\eps_0/4$ contains a peripheral circle of diameter $\ge \de$. Hence each ball 
$B_i=B(y_i, \eps_0/4)$, $i=1,2,3$, contains a peripheral circle of 
diameter $\ge \de$. 
Note that $\dist (B_i,B_j)\ge d(y_i,y_j)-\eps_0/2\ge\eps_0/2$.  Therefore, the set $\geo G\setminus g( B)$ can meet at most one of the balls, and we can 
pick one of the balls,
say $B':=B_k$, where $k\in \{1,2,3\}$, so that 
 $B'\cap \geo G\setminus g(B)=\emptyset.  $
The ball $ B'$ contains a peripheral circle $J'$  with $\diam(J')
\ge \delta $. Let 
$J:=g^{-1}(J')$. Then $J$ is a peripheral circle 
with $$J\sub g^{-1}( B')\sub g^{-1}(g(B))=B.
$$
It remains to show that $J$ has a diameter comparable to 
$r$. To see this pick $u,v\in J$ such that 
$$ d(g(u), g(v))=\diam (g(J)) =\diam(J')\ge \de. $$ 
Two of the points $x_1,x_2,x_3$ must have distance $\ge 
r/(8\la)$ to $u$. Of these two, one must have distance 
$\ge r/(8\la)$ to $v$. It follows that there exist
$k,l\in \{1,2,3\}$, $k\ne l$, such that 
$d(x_k,u)\ge r/(8\la)$ and $d(x_l,v)\ge r/(8\la)$.
Then
\begin{eqnarray*}
[g(x_k), g(u), g(x_l), g(v)] &\le & \eta([x_k, u, x_l, v])\\
&=& \eta\biggl(\frac{d(x_k,x_l) d(u,v)}{d(x_k,v) d(u,x_l)}
\biggr)\\
&\le&\eta\biggl(128 \la\cdot \frac{ d(u,v)}r
\biggr). 
\end{eqnarray*}
On the other hand,
\begin{eqnarray*}
[g(x_k), g(u), g(x_l), g(v)] &= & 
 \frac{d(y_k,y_l) d(g(u),g(v))}{d(y_k,g(v)) 
d(g(u),y_l)}
\\
&\ge&\frac{\eps_0\de}{\diam(\geo G)^2}=:c_1>0.
\end{eqnarray*}
Hence 
$$  \frac 1r\diam(J) \ge \frac 1r d(u,v)\ge \frac 1{128\la}
\eta^{-1}(c_1) =:c_2>0.$$ 
Since $c_2>0$ is a positive constant independent of the ball 
$B$, it follows that every ball $B$ in $\geo G$ of radius $r\le \diam(\geo G)$ contains a peripheral circle of comparable size, where the constant of comparability is independent of the ball. The proves the last statement. 
\qed

\end{document}